\documentclass[11pt, a4paper, reqno]{amsart}

\usepackage[utf8]{inputenc}
\usepackage{a4wide, microtype, url, amssymb, amsmath, pifont, graphicx}

\usepackage[usenames,dvipsnames,svgnames]{xcolor}
\definecolor{darkblue}{rgb}{0.0,0.0,0.7}
\usepackage[pdflang=en-UK, colorlinks, urlcolor=darkblue, linkcolor=darkblue, citecolor=darkblue]{hyperref}

\usepackage{tikz, tikz-cd}
\usepackage{float}
\usetikzlibrary{calc}
\usetikzlibrary{arrows,decorations.markings,shapes.geometric,matrix,fit}
\pgfarrowsdeclarecombine{twohead}{twohead}
{>}{>}{>}{>}
\usepackage[noabbrev,capitalise]{cleveref} \usepackage[numbers,sort]{natbib}
\usepackage{enumitem}
\setlist[enumerate,1]{label={(\roman*)}}
\usepackage{caption}

\usepackage{quiver}

\usepackage{adjustbox}

\usepackage{amsthm}
\theoremstyle{plain}
\newtheorem{theorem}{Theorem}[section]
\newtheorem*{theorem*}{Theorem}
\newtheorem{corollary}[theorem]{Corollary}
\newtheorem{proposition}[theorem]{Proposition}
\newtheorem{lemma}[theorem]{Lemma}

\newtheorem{fact}[theorem]{Fact}
\newtheorem{problem}[theorem]{Problem}
\newtheorem*{problem*}{Problem}
\newtheorem{thmx}{Theorem}

\theoremstyle{definition}
\newtheorem{definition}[theorem]{Definition}
\newtheorem{example}[theorem]{Example}
\newtheorem{remark}[theorem]{Remark}

\numberwithin{equation}{section}
\numberwithin{figure}{section}

   \newcommand{\Z}{\mathbb{Z}}   \newcommand{\FF}{\mathbb{F}}  \newcommand{\EE}{\mathbb{E}} \newcommand{\K}{k}  \newcommand{\cal}[1]{{\mathcal{#1}}}

\newcommand*\cocolon{\nobreak
        \mskip6mu plus1mu
        \mathpunct{}\nonscript
        \mkern-\thinmuskip
        {:}\mskip2mu
        \relax
}
\renewcommand{\xrightarrow}[1]{\overset{#1}\to}
\renewcommand{\subset}{\subseteq}

\newcommand{\cA}{\mathcal{A}}

\newcommand{\cC}{\mathcal{C}}
\newcommand{\cD}{\mathcal{D}}
\newcommand{\cE}{\mathcal{E}}
\newcommand{\cF}{\mathcal{F}}

\newcommand{\cI}{\mathcal{I}}

\newcommand{\cP}{\mathcal{P}}

\newcommand{\cT}{\mathcal{T}}

\newcommand{\cW}{\mathcal{W}}

\newcommand{\bE}{\mathbb{E}}

\newcommand{\EbbE}{\mathbb{E}_{\mathrm{emb}}}

\newcommand{\isoto}{\stackrel{\sim}{\to}}

\newcommand{\longmap}[1]{\stackrel{#1}{\longrightarrow}}
\newcommand{\map}[1]{\stackrel{#1}{\to}}

\DeclareMathOperator{\add}{add}
\newcommand{\blank}{-}
\newcommand{\cohom}[1]{\mathrm{H}^{#1}}
\DeclareMathOperator{\coker}{coker}
\DeclareMathOperator{\cone}{cone}
\DeclareMathOperator{\End}{End}
	\DeclareMathOperator{\stabEnd}{\underline{End}}
\newcommand{\Ab}{\mathrm{Ab}}
\DeclareMathOperator{\Ext}{Ext}
\DeclareMathOperator{\gldim}{gl.dim}
\DeclareMathOperator{\Hom}{Hom}
	\DeclareMathOperator{\RHom}{\mathbf{R}Hom}
	\DeclareMathOperator{\stabHom}{\underline{Hom}}
	\DeclareMathOperator{\injstabHom}{\overline{Hom}}
\newcommand{\id}{\mathrm{id}}
\DeclareMathOperator{\im}{im}
\newcommand{\ltens}{\mathbin{\stackrel{\mathbf{L}}{\otimes}}}
\newcommand{\varltens}[1]{\mathbin{\otimes^{\mathbf{L}}_{#1}}}
\DeclareMathOperator{\Mor}{Mor}
\newcommand{\op}{\mathrm{op}}

\DeclareMathOperator{\Transp}{Tr}
\newcommand{\GrothGp}[1]{\mathrm{K}_0(#1)}

\newcommand{\bdd}{\mathrm{b}}
\newcommand{\Kb}[1]{\cal{K}^{\bdd}(#1)}
\newcommand{\dcat}[1][]{\cal{D}^{#1}}

\newcommand{\GP}{\operatorname{GP}}
\newcommand{\har}[1]{\cal{K}^{[-1,0]}(#1)}

\newcommand{\harp}[2][]{\cal{K}^{[-1,0]}_{#1}(#2)}

\DeclareMathOperator{\proj}{proj}
\DeclareMathOperator{\per}{per}
\renewcommand{\mod}{\operatorname{mod}}

\newcommand{\lperp}{{}^\perp}
	\newcommand{\rperp}{^\perp}
\newcommand{\stab}{\underline}

\newcommand{\rel}{\mathrm{rel}}
\newcommand{\FundDom}[1]{\mathcal{F}_{#1}}
\newcommand{\RelFundDom}[1]{\mathcal{F}_{#1}^{\rel}}

\newcommand{\der}[2]{\partial_{#1}{#2}}

\newcommand{\pp}[1]{\overline{#1}}
\DeclareMathOperator{\pvd}{pvd}
\newcommand{\RelClustCat}[1]{\mathcal{C}_{#1}^{\rel}}
\newcommand{\HiggsCat}[1]{\mathcal{H}_{#1}}
\newcommand{\HiggsCatWu}[1]{\mathcal{H}_{#1}}
\newcommand{\HiggsCatDual}[1]{\mathcal{H}^\vee_{#1}}
\newcommand{\StabHiggsCat}[1]{\underline{\mathcal{H}}_{#1}}

\newcommand{\infl}{\rightarrowtail}
\newcommand{\defl}{\twoheadrightarrow}
\newcommand{\s}{\mathfrak{s}}

\tikzcdset{infl/.style={rightarrowtail},defl/.style={twoheadrightarrow},ext/.style={dashed}}

\newcommand{\yo}{\text{\usefont{U}{min}{m}{n}\symbol{"48}}}
\newcommand{\jpne}{\text{\usefont{U}{min}{m}{n}\symbol{"2D}}}
\newcommand{\da}{\text{\usefont{U}{min}{m}{n}\symbol{'040}}}
\DeclareFontFamily{U}{min}{}
\DeclareFontShape{U}{min}{m}{n}{<-> dmjhira}{}

\newcommand{\blossom}{^\text{\ding{96}}} 

\usepackage{todonotes}

\title[Extriangulated ideal quotients]
{Extriangulated ideal quotients, with applications to cluster theory and gentle algebras}

\author[X. Fang] {Xin Fang}
\address[Xin Fang]
{Lehrstuhl für Algebra und Darstellungstheorie\\
RWTH Aachen University\\
Pontdriesch~10--16\\
52062 Aachen\\
Germany}
\email{\href{mailto:xinfang.math@gmail.com}{xinfang.math@gmail.com}}
\urladdr{\url{http://www.mi.uni-koeln.de/~xfang/}}

\author[M. Gorsky] {Mikhail Gorsky}
\address[Mikhail Gorsky]
{Department of Mathematics\\
University of Vienna\\
Oscar-Morgenstern Platz~1\\
1090 Vienna\\
Austria}
\email{\href{mailto:mikhail.gorskii@univie.ac.at}{mikhail.gorskii@univie.ac.at}}
\urladdr{\url{https://sites.google.com/site/homepageofmikhailgorsky/}}

\author[Y. Palu] {Yann Palu}
\address[Yann Palu]
{LAMFA\\
Universit\'e Picardie Jules Verne\\
33 Rue Saint-Leu\\
80039 Amiens\\
France}
\email{\href{mailto:yann.palu@u-picardie.fr}{yann.palu@u-picardie.fr}}
\urladdr{\url{http://www.lamfa.u-picardie.fr/palu/}}

\author[P.-G. Plamondon] {Pierre-Guy Plamondon}
\address[Pierre-Guy Plamondon]
{Laboratoire de Mathématiques de Versailles\\
UVSQ\\
CNRS\\
Université Paris-Saclay\\
45 Avenue des États Unis\\
78000 Versailles\\
France}
\email{\href{mailto:pierre-guy.plamondon@uvsq.fr}{pierre-guy.plamondon@uvsq.fr}}
\urladdr{\url{https://lmv.math.cnrs.fr/laboratoire/annuaire/membres-du-laboratoire/pierre-guy-plamondon/}}

\author[M. Pressland] {Matthew Pressland}
\address[Matthew Pressland]
{School of Mathematics \& Statistics\\
University of Glasgow\\
University Place\\
Glasgow\\
G12 8QQ\\
United Kingdom}
\email{\href{mailto:matthew.pressland@glasgow.ac.uk}{matthew.pressland@glasgow.ac.uk}}
\urladdr{\url{http://mdpressland.github.io}}

\begin{document}

\begin{abstract}
We extend results of Br\"ustle--Yang on ideal quotients of $2$-term subcategories of perfect derived categories of non-positive dg algebras to a relative setting. We find a new interpretation of such quotients: 
 they appear as prototypical examples of a new construction of  quotients of extriangulated categories by ideals generated by morphisms from injectives to projectives. We apply our results to Frobenius exact cluster categories and Higgs categories with suitable relative extriangulated structures, and to categories of walks related to gentle algebras.
In all three cases, the extriangulated structures are well-behaved (they are $0$-Auslander) and their quotients are equivalent to homotopy categories of two-term complexes of projectives over suitable finite-dimensional algebras.
\end{abstract}

\maketitle
\tableofcontents

\section{Introduction}

The additive categorification of cluster algebras was achieved through the introduction of cluster categories \cite{BMRRT, Amiot09}, which form a class of triangulated categories with useful homological properties.  Amiot defines a (generalised) cluster category as a Verdier quotient $\cC = \per \Gamma / \pvd \Gamma$ for a certain dg algebra $\Gamma$.  Importantly, she shows that there is a fundamental domain~$\cF$ in~$\per\Gamma$ which is additively equivalent to the Verdier quotient.  These constructions and properties were generalised to the relative setting, under the name of Higgs categories, by Wu \cite{Wu23}, requiring the more flexible setting of extriangulated categories \cite{NakaokaPalu}.  

It was observed by Brüstle and Yang in \cite{BrustleYang} that the cluster category defined from a dg algebra~$\Gamma$ is related to the category~$\harp{\proj \cohom{0}\Gamma}$ of~$2$-term complexes over its~$0$-th cohomology~$\cohom{0}\Gamma$: this latter category is the quotient of the fundamental domain~$\cF$ by the ideal of morphisms with codomain in~$\add\Gamma$ and domain in~$\add\Sigma\Gamma$.  The practical upshot is that~$\harp{\proj \cohom{0}\Gamma}$ can also be used for categorifying cluster algebras. This was done for instance in~\cite{PPPP}, where similarities between the extriangulated structure of~$\harp{\proj \cohom{0}\Gamma}$ and a suitable relative extriangulated structure of the cluster category (in the sense of~\cite{HerschendLiuNakaoka}) were exploited.  This observation still holds for Higgs categories, as demonstrated in the Auslander--Reiten quivers depicted in Figure~\ref{fig::Higgs}.

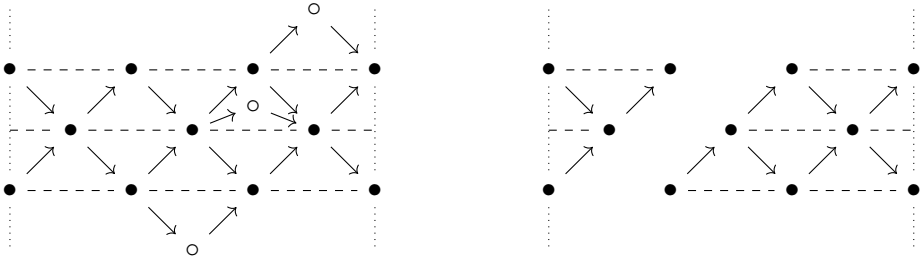
\begin{figure}[H]
\hspace{-2cm}
\begin{minipage}{.2\textwidth}
\begin{tikzpicture}[scale=.8]
\foreach \x/\y/\n in {0/0/M4, 0/-2/M9, 1/-1/M8, 2/0/M7, 2/-2/M3, 3/-1/M2, 4/0/M1, 4/-2/M6, 5/-1/M5, 6/0/M9b, 6/-2/M4b}
{\node at (\x,\y) (\n) {$\bullet$};}

\foreach \x/\y/\n in {3/-3/P6, 4/-0.6/P4, 5/1/P2}
{\node at (\x,\y) (\n) {$\circ$};} 

\path[->]
	(M4) edge (M8)
	(M9) edge (M8)
	(M8) edge (M7)
	(M8) edge (M3)
	(M7) edge (M2)
	(M3) edge (M2)
	(M3) edge (P6)
	(M2) edge (M1)
	(M2) edge (P4)
	(M2) edge (M6)
	(P6) edge (M6)
	(M1) edge (P2)
	(M1) edge (M5)
	(P4) edge (M5)
	(M6) edge (M5)
	(P2) edge (M9b)
	(M5) edge (M9b)
	(M5) edge (M4b);
\path[dashed]
	(0,-1) edge (M8)
	(M4) edge (M7)
	(M9) edge (M3)
	(M8) edge (M2)
	(M7) edge (M1)
	(M3) edge (M6)
	(M2) edge (M5)
	(M1) edge (M9b)
	(M6) edge (M4b)
	(M5) edge (6,-1);
\path[dotted]
	(0,1) edge (M4)
	(M4) edge (M9)
	(M9) edge (0,-3)
	(6,1) edge (M9b)
	(M9b) edge (M4b)
	(M4b) edge (6,-3);
\end{tikzpicture}
\end{minipage}
\qquad \qquad \qquad \qquad \qquad
\begin{minipage}{.2\textwidth}
\begin{tikzpicture}[scale=.8]
\foreach \x/\y/\n in {0/0/M4, 0/-2/M9, 1/-1/M8, 2/0/M7, 2/-2/M3, 3/-1/M2, 4/0/M1, 4/-2/M6, 5/-1/M5, 6/0/M9b, 6/-2/M4b}
{\node at (\x,\y) (\n) {$\bullet$};}

\path[->]
	(M4) edge (M8)
	(M9) edge (M8)
	(M8) edge (M7)
	(M3) edge (M2)
	(M2) edge (M1)
	(M2) edge (M6)
	(M1) edge (M5)
	(M6) edge (M5)
	(M5) edge (M9b)
	(M5) edge (M4b);
\path[dashed]
	(0,-1) edge (M8)
	(M4) edge (M7)
	(M3) edge (M6)
	(M2) edge (M5)
	(M1) edge (M9b)
	(M6) edge (M4b)
	(M5) edge (6,-1);
\path[dotted]
	(0,1) edge (M4)
	(M4) edge (M9)
	(M9) edge (0,-3)
	(6,1) edge (M9b)
	(M9b) edge (M4b)
	(M4b) edge (6,-3);
\end{tikzpicture}
\end{minipage}
\caption{On the left, a Higgs category in type $\sf{A}_3$ with principal coefficients.  On the right, its quotient by the ideal generated by morphisms from injectives to projectives in an appropriate relative extriangulated structure, which is equivalent to~$\harp{\proj k\vec{\sf{A}}_3}$.  Empty circles are projective-injective objects; dashed horizontal lines represent almost split conflations; dotted vertical lines are identified after a reflection.}
\label{fig::Higgs}
\end{figure}

The setting used in~\cite{PPPP} was axiomatised in~\cite{GNP2}, yielding the notion of a~$0$-Auslander extriangulated category.  This axiomatisation turns out to apply to a seemingly unrelated context: the study of~$\tau$-tilting theory for gentle algebras. In \cite{PaluPilaudPlamondon-nonkissing} (see also \cite{McConville,BrustleDouvilleMousavandThomasYildirim}), the $\tau$-rigidity of modules over a gentle algebra~$A$ given by a gentle quiver with relations~$(Q,R)$ was described using the combinatorics of walks on the blossoming quiver~$(Q\blossom,R\blossom)$ (which first appeared in \cite{Asashiba12}).  The model in~\cite{PaluPilaudPlamondon-nonkissing} contains seemingly superfluous combinatorial objects called straight walks; these were successfully categorified in~\cite{GNP2} as projective-injective objects in an exact~$0$-Auslander category called the category of walks and denoted by~$\mathcal{W}_A$.
It turns out that the category~$\harp{\proj A}$ can be recovered from an ideal subquotient of~$\harp{\proj A\blossom}$ or, alternatively, as an ideal quotient of~$\mathcal{W}_A$, as illustrated in Figure~\ref{fig::gentle}.

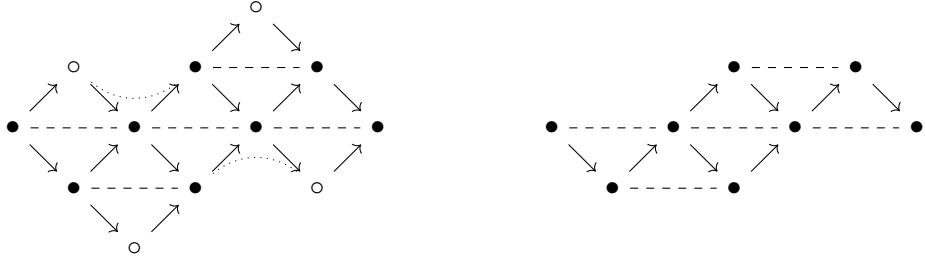
\begin{figure}[H]
\hspace{-2cm}
\begin{minipage}{.2\textwidth}
\begin{tikzpicture}[scale=.8]
\foreach \x/\y/\n in {0/-2/M1, 1/-3/M2, 2/-2/M3, 3/-3/M4, 3/-1/M5, 4/-2/M6, 5/-1/M7, 6/-2/M8}
{\node at (\x,\y) (\n) {$\bullet$};}

\foreach \x/\y/\n in {1/-1/P1, 2/-4/P2, 4/0/P3, 5/-3/P4}
{\node at (\x,\y) (\n) {$\circ$};} 

\path[->]
	(M1) edge (P1)
	(M1) edge (M2)
	(M2) edge (M3)
	(M2) edge (P2)
	(P1) edge (M3)
	(M3) edge (M4)
	(M3) edge (M5)
	(P2) edge (M4)
	(M5) edge (P3)
	(M5) edge (M6)
	(M4) edge (M6)
	(P3) edge (M7)
	(M6) edge (M7)
	(M6) edge (P4)
	(M7) edge (M8)
	(P4) edge (M8);
\path[dashed]
	(M1) edge (M3)
	(M3) edge (M6)
	(M6) edge (M8)
	(M5) edge (M7)
	(M2) edge (M4);
\path[dotted]
    (P1) edge[bend right=40] (M5)
    (M4) edge[bend left=40] (P4);
\end{tikzpicture}
\end{minipage}
\qquad \qquad \qquad \qquad \qquad
\begin{minipage}{.2\textwidth}
\begin{tikzpicture}[scale=.8]
\foreach \x/\y/\n in {0/-2/M1, 1/-3/M2, 2/-2/M3, 3/-3/M4, 3/-1/M5, 4/-2/M6, 5/-1/M7, 6/-2/M8}
{\node at (\x,\y) (\n) {$\bullet$};}

\path[->]
	(M1) edge (M2)
	(M2) edge (M3)
	(M3) edge (M4)
	(M3) edge (M5)
	(M5) edge (M6)
	(M4) edge (M6)
	(M6) edge (M7)
	(M7) edge (M8);
\path[dashed]
	(M1) edge (M3)
	(M3) edge (M6)
	(M6) edge (M8)
	(M5) edge (M7)
	(M2) edge (M4);
\end{tikzpicture}
\end{minipage}
\caption{On the left, the category of walks~$\cW_A$ of a gentle algebra of type~$\sf{A}_3$ with a relation.  On the right, the quotient of~$\cW_A$ by the ideal generated by morphisms from injectives to projectives, which is equivalent to~$\harp{\proj A}$.  Empty circles are projective-injective objects; dashed horizontal lines represent almost split conflations; curved dotted lines represent zero-relations that cannot be deduced from almost split conflations.}
\label{fig::gentle}
\end{figure}

In this paper, we show that the two situations above, where a category~$\harp{\proj A}$ is recovered as an ideal quotient of an extriangulated category, are special cases of a general phenomenon.  

Our first main result asserts that, under certain circumstances, ideal quotients of extriangulated categories inherit an extriangulated structure.  

\begin{thmx}[= Theorem \ref{t:extriangulatedQuotient}]
 Let $(\cal{C}, \mathbb{E}_{\cC}, \mathfrak{s})$ be an extriangulated category, let~$J_0$ be a class of morphisms with injective domain and projective codomain, and let~$J$ be the ideal generated by~$J_0$.  Let~$\pi\colon \cC \to \cC/J$ be the quotient functor.  Then there exists an (explicit) extriangulated structure on~$\cC/J$ such that~$\mathbb{E}_{\cC/J}(\pi(\blank), \pi(\blank)) = \mathbb{E}_{\cC}(\blank,\blank)$ and $(\pi, \id_{\mathbb{E}})$ is an extriangulated functor.
\end{thmx}

Such extriangulated ideal quotients behave well with respect to rigid objects; see Proposition~\ref{prop::rigid}.

Our second main result exhibits a setting general enough to cover all instances of Higgs categories and examples from gentle algebras, using non-positively graded (also called connective) dg categories.  It explains the appearance of categories of the form~$\harp{\proj A}$ as ideal quotients of~$0$-Auslander extriangulated categories.

\begin{thmx}[= Theorem~\ref{t:rel-2-term}, Theorem~\ref{t:rel-2-term-functors}, and Corollary~\ref{c:kerG}]
\label{theo::main2}
 Let $\cal{A}$ be a non-positively graded dg category, and let $\cal{E}\subseteq\cal{A}$ be a full subcategory which is covariantly finite as a subcategory of~$\per\cA$. Consider the subcategory $\cal{A}*\Sigma\cal{A}\subseteq\per\cal{A}$ of $2$-term objects, and write
\[\lperp(\Sigma\cal{E})=\{X\in\cal{A}*\Sigma\cal{A}:\text{$\Hom_{\per\cal{A}}(X,\Sigma E)=0$ for all $E\in\cal{E}$}\}\]
for the left perpendicular category of $\Sigma\cal{E}\subseteq\cal{A}*\Sigma\cal{A}$. Then
\begin{enumerate}
\item with the extriangulated structure induced by the embedding into $\cal{A}*\Sigma\cal{A}$ as an extension-closed subcategory, $\lperp(\Sigma\cal{E})$ is a $0$-Auslander extriangulated category;
\item the projective objects in $\lperp(\Sigma\cal{E})$ are those of $\cal{A}$, while the injective objects are those of the Bongartz co-completion of $\cal{E}$ in~$\cal{A}*\Sigma\cal{A}$;

\item in the commutative diagram
\begin{equation}
\begin{tikzcd}[column sep=40pt]
\per{\cal{A}}\arrow{r}{\blank\ltens_{\cal{A}}\cohom{0}\cal{A}}&\per\cohom{0}\cal{A}\arrow{r}{\blank\ltens_{\cohom{0}\cal{A}}\cohom{0}\cal{A}/\cohom{0}\cal{E}}&\per\cohom{0}\cal{A}/\cohom{0}\cal{E}\\
\lperp(\Sigma\cal{E})\arrow[hookrightarrow]{u}\arrow{r}{G}&\lperp(\Sigma\cohom{0}\cal{E})\arrow[hookrightarrow]{u}\arrow{r}{F}&\harp{\proj\cohom{0}\cal{A}/\cohom{0}\cal{E}},\arrow[hookrightarrow]{u}
\end{tikzcd}
\end{equation}
the functors $F$ and $G$, given by restricting the derived tensor products, are full and essentially surjective extriangulated functors. Moreover, $\ker(FG)$ is the ideal of $\lperp(\Sigma\cal{E})$ generated by morphisms from an injective object to a projective object. 
\end{enumerate}
\end{thmx}

This result has a dual counterpart; see \cref{t:rel-2-term-dual,t:rel-2-term-dual-functors}.

In fact, this setting is ``as general as can be'' for our purposes: indeed, after some of the results of this paper were presented at ARTA IX in June 2023, Xiaofa Chen proved in \cite{chen20230} that all algebraic~$0$-Auslander extriangulated categories (i.e.\ those admitting an enhancement by an exact dg category in the sense of \cite{chen2023exact}) can be expressed as in the theorem using a non-positively graded dg category.
We also note a result by Yang~\cite{Yang-2term}, analogous to \cref{theo::main2}, in the setting of algebraic triangulated categories, and motivated by different applications.

We find three main applications of the above theorem.

\begin{enumerate}
 \item {\bf Frobenius exact categories.}  Motivated by the categorification of cluster algebras
by means of Frobenius exact categories (see for instance \cite{GLS-ParFlag, JKS, Pressland22}), we consider in Section~\ref{section:Frobenius} the following setting.  Let~$\cC$ be a Frobenius exact category with a cluster-tilting subcategory~$\cT$.  Then Theorem~\ref{t:addEquivalenceYoneda} allows us to apply Theorem~\ref{theo::main2} to the category $\cC$ equipped with certain exact substructures depending on $\cT$.  In this application,~$\cA = \cT$ and~$\cE$ is the subcategory of projective-injective objects of $\cC$.

 \item {\bf Higgs categories.} Cluster categories~\cite{Amiot09}, and more generally Higgs categories~\cite{Wu23}, are triangulated or extriangulated categories that play an important role in the categorification of cluster algebras.  If~$\cC$ is a Higgs category, then we show in Theorem~\ref{t:ClustCat-to-2term} that there is an extriangulated category~$\cC_T$, given by passing to an extriangulated substructure of $\cC$, to which Theorem~\ref{theo::main2} applies (with~$\cA = \Gamma$  the Ginzburg dg algebra of an ice quiver and~$\cE = \Gamma e$ its projective-injective direct summand): the  quotient of~$\cC_T$ by the ideal generated by morphisms from injectives to projectives is equivalent to~$\harp{\proj \underline{A}}$ for some algebra~$\underline{A}$.  This generalises a result observed for cluster categories in~\cite{BrustleYang}.

 \item {\bf Gentle algebras.}  Let~$(Q,R)$ be a gentle quiver with relations, let~$(Q\blossom, R\blossom)$ be its blossoming, and let~$A$ and~$A\blossom$ be the corresponding (finite-dimensional) gentle algebras.  The combinatorics of walks on~$(Q\blossom,R\blossom)$ was categorified in~\cite{GNP2} using a subcategory~$\cW_A$ of~$\mod A\blossom$.  In Section~\ref{section:gentle}, we apply Theorem~\ref{theo::main2} with~$\cA = A\blossom$ and~$\cE = e\blossom A\blossom$ to obtain the equivalence
 \[
  \lperp(\Sigma e\blossom A\blossom)/(e\blossom A\blossom) \xrightarrow{\sim} \harp{\proj A}
 \]
 in Proposition~\ref{prop::gentle-har}, and interpret~$\cW_A$ as an intermediate quotient of~$\lperp(\Sigma e\blossom A\blossom)$, so that the quotient of~$\cW_A$ by the ideal of morphism from injectives to projectives is also equivalent to~$\harp{\proj A}$ (Corollary~\ref{corollary: category of walks}).
\end{enumerate}

In all three applications, the extriangulated categories that appear are algebraic. 
We end with a problem that would strengthen the results of this paper by lifting the algebraicity assumptions.

\begin{problem*}[= Problem~\ref{p:0-auslander}]
 Determine whether the quotient of any (possibly non-algebraic) $0$-Auslander extriangulated category by the ideal generated by morphisms from injective objects to projective objects is equivalent to $\harp{\cal{A}}$ as an extriangulated category, where $\cal{A}$ is an additive category.
\end{problem*}

\section*{Structure of the paper}
The paper is organised as follows. In Section \ref{section:quotients}, we give a brief reminder on extriangulated categories and prove our first result on their ideal quotients. In Section \ref{section:Auslander}, we recall the definition, basic examples and properties of $0$-Auslander extriangulated categories, and discuss the appearance of such ideal quotients in this setting. 

Section \ref{s:rel-2-term} forms the technical core of the paper. We introduce relative $2$-term categories as certain subcategories of perfect derived categories. We prove that they are $0$-Auslander and show that restrictions of certain derived tensor products are examples of the aforementioned extriangulated ideal quotients.

The last sections are concerned with the applications of these constructions. In Sections \ref{section:Frobenius}, \ref{section:Higgs}, and \ref{section:Examples}, we discuss the applications to additive categorification of cluster algebras. In \ref{section:Frobenius},  we realise Frobenius exact cluster categories with relative exact structures as subcategories of relevant  $2$-term categories of projectives. In Section \ref{section:Higgs}, we show that endowing Wu's Higgs categories with suitable relative structures allows us to upgrade the $k$-linear equivalence between the Higgs category and a certain relative $2$-term category to an extriangulated equivalence. We use this to prove that a certain subcategory of $\harp{\proj A}$, where $A$ is the endomorphism algebra of the initial cluster-tilting object, is an extriangulated ideal quotient of the Higgs category. Both cases are illustrated by examples in Section \ref{section:Examples}.
In Section \ref{section:gentle}, we show that $2$-term categories of projectives over gentle algebras are ideal quotients of exact categories of walks.

\section*{Acknowledgements}

The authors would like to thank Thomas Brüstle, Xiaofa Chen, Yifei Cheng, Bernhard Keller, Hiroyuki Nakaoka, Hipolito Treffinger, Nicholas Williams, Dong Yang, and Alexandra Zvonareva for many interesting discussions related to this project.

The project began with discussions between some of the authors and Andrea Pasquali at the conference \emph{Journées d'Algèbre} held in March 2020 in Caen on the occasion of Bernard Leclerc's 60th birthday. We are very grateful to Andrea and to the organisers of the conference, Paolo Bellingeri and Cédric Lecouvey.

X.F.\ was partially supported by SFB-TRR 195 \emph{Symbolic tools in mathematics and their applications}.
M.G., Y.P., and P.-G.P.\ were partially supported by the French ANR grant CHARMS~(ANR-19-CE40-0017). P.-G.P.\ is partially funded by the Institut Universitaire de France (IUF). M.P.\ is supported by a postdoctoral fellowship (EP/T001771/2) from the UK's Engineering and Physical Sciences Research Council (EPSRC). 
This work is a part of a project that has received funding from the European Research Council (ERC) under the European Union's Horizon 2020 research and innovation programme (grant agreement No.~101001159).

Parts of this work were done when M.G., Y.P., and M.P.\ participated in the Junior Trimester Program \emph{New Trends in Representation Theory} at the Hausdorff Institute for Mathematics in Bonn, when M.P.\ and P.-G.P.\ participated in the programme \emph{Cluster Algebras and Representation Theory} at the Isaac Newton Institute for Mathematical Sciences, Cambridge (supported by EPSRC grant EP/R014604/1), and when P.-G.P.\ participated in the programme \emph{Representation Theory: Combinatorial Aspects and Applications} at the Centre for Advanced Study, Oslo.
Further work on this project was done during M.G.'s stays at the University of Stuttgart, and he is very grateful to Steffen Koenig for the hospitality.

\section{Extriangulated ideal quotients} \label{section:quotients}

\subsection{Extriangulated categories}
Throughout the paper, we work with extriangulated categories, introduced in~\cite{NakaokaPalu} and generalising both exact categories and triangulated categories. They axiomatise extension-closed subcategories of triangulated categories.

An \emph{extriangulated category} is a triple $(\cal{C}, \mathbb{E}, \mathfrak{s})$ consisting of
\begin{itemize}
 \item an additive category $\cal{C}$,
 \item an additive bifunctor $\mathbb{E}\colon\cal{C}^\mathrm{op}\times\cal{C}\to \mathrm{Ab}$ playing the role of an $\operatorname{Ext}^1$-bifunctor, and
 \item an \emph{additive realization} $\mathfrak{s}$ sending each element $\delta\in\EE(Z,X)$ to an equivalence class of weak kernel--cokernel pairs $X\overset{f}\to Y \overset{g}\to Z$ playing the role of short exact sequences or of distinguished triangles. In each such pair, we say that $f$ is an \emph{inflation}, $g$ is a \emph{deflation}, and the pair itself is a \emph{conflation}.
\end{itemize}
These data must satisfy certain axioms, which can be found in \cite{NakaokaPalu} and \cite[Appendix A]{GNP2} (see also \cite{Palu2023}), along with some basic notions and properties that will be used throughout the paper.
One also says that $(\cal{C}, \mathbb{E}, \mathfrak{s})$ is an \emph{extriangulated structure} on $\cal{C}$.

The first main property of extriangulated categories is that for any conflation $X\xrightarrow{f}Y\xrightarrow{g}Z$ 
realising $\delta \in \EE(Z, X)$, one has two associated exact sequences of natural transformations
\begin{equation} \label{eq:6term1}
\cC(Z,\blank)\overset{\blank\circ g}{\longrightarrow}\cC(Y,\blank)\overset{\blank\circ f}{\longrightarrow}\cC(X,\blank)\overset{\delta^\sharp}{\longrightarrow}\EE(Z,\blank)\overset{g^\ast}{\longrightarrow}\EE(Y,\blank)\overset{f^\ast}{\longrightarrow}\EE(X,\blank), 
\end{equation} 
\begin{equation} \label{eq:6term2}
\cC(\blank,X)\overset{f\circ\blank}{\longrightarrow}\cC(\blank,Y)\overset{g\circ\blank}{\longrightarrow}\cC(\blank,Z)\overset{\delta_\sharp}{\longrightarrow}\EE(\blank,X)\overset{f_\ast}{\longrightarrow}\EE(\blank,Y)\overset{g_\ast}{\longrightarrow}\EE(\blank,Z), 
\end{equation}
where $\delta^\sharp, \delta_\sharp$ are induced by the functoriality of $\EE$ in both arguments.

We recall that an object $X$ is called \emph{projective} (resp.\ \emph{injective}) if $\mathbb{E}(X,\blank) = 0$ (resp.\ $\mathbb{E}(\blank, X) = 0$). The \emph{stable category} $\underline{\cal{C}}$ of an extriangulated category $\cal{C}$ having enough projective objects is the additive category given by the quotient of $\cal{C}$ by the ideal of morphisms factoring through projective objects. Dually, if $\cal{C}$ has enough injectives, the costable category $\overline{\cal{C}}$  is the quotient of $\cal{C}$ by  the ideal of morphisms factoring through injective objects.\footnote{More general definitions apply when the categories do not have enough projectives or injectives, but we will never use these in this paper.}

An \emph{extriangulated substructure}, or \emph{relative extriangulated structure}, of $(\cal{C},\EE,\s)$ is a sub-bifunctor $\FF\leq\EE$ such that $(\cal{C},\FF,\s|_{\FF})$ is itself an extriangulated category. With each full subcategory $\cal{S} \subseteq \cal{C}$, one may associate two substructures of  $(\cal{C}, \mathbb{E}, \mathfrak{s})$ using the following statement, which is a direct application of results of Herschend, Liu and Nakaoka \cite[Prop.~3.16,~3.19]{HerschendLiuNakaoka}.
\begin{proposition}
\label{p:HLN-substruct}
Let $\cal{C}$ be an extriangulated category and let $\cal{S}\subseteq\cal{C}$ be a (full, additively closed) subcategory of $\cal{C}$. For $X,Z\in\cal{C}$ write
\begin{align*}
\EE_{\cal{S}}(Z,X)&:=\{\delta\in\EE(Z,X):\text{$(\delta_{\sharp})_S=0$ for all $S\in\cal{S}$}\},\\
\EE^{\cal{S}}(Z,X)&:=\{\delta\in\EE(Z,X):\text{$(\delta^{\sharp})_S=0$ for all $S\in\cal{S}$}\},
\end{align*}
where $(\delta_{\sharp})_S\colon\Hom_{\cal{C}}(S,Z)\to\EE_{\cal{C}}(S,X)$ and $(\delta^{\sharp})_S\colon\Hom_{\cal{C}}(X,S)\to\EE_{\cal{C}}(Z,S)$ are the connecting homomorphisms as in \eqref{eq:6term1} and \eqref{eq:6term2}. Then $\EE_{\cal{S}}$ and $\EE^{\cal{S}}$ are extriangulated substructures of $(\cal{C},\EE,\s)$.
\end{proposition}

To simplify the notation, we will abbreviate $\cal{C}_{\cal{S}}=(\cal{C},\EE_{\cal{S}},\s|_{\EE_{\cal{S}}})$ and $\cal{C}^{\cal{S}}=(\cal{C},\EE^{\cal{S}},\s|_{\EE^{\cal{S}}})$, analogous to the abbreviation $\cal{C}=(\cal{C},\EE,\s)$ for $\cal{C}$ with its original extriangulated structure. 
We will use substructures of this form extensively in Sections \ref{section:Frobenius} and \ref{section:Higgs}.
Various further constructions of relative structures on extriangulated categories can be found in \cite{HerschendLiuNakaoka, chen2023exact, ogawa2022localization, PPPP, JorgShah, Sakai, Christ}.

\begin{definition}[{cf.~\cite[Defn.~2.32]{BTS}}]
Let $(\cal{C}_1,\EE_1,\s_1)$ and $(\cal{C}_2,\EE_2,\s_2)$ be extriangulated categories. An \emph{extriangulated functor} from $(\cal{C}_1,\EE_1,\s_1)$ to $(\cal{C}_2,\EE_2,\s_2)$ is a pair of an additive functor $F\colon\cal{C}_1\to\cal{C}_2$  together with a natural transformation $\alpha\colon\EE_1 \Rightarrow \EE_2 \circ (F^\op \times F)$ such that
\[\s_2(\alpha(\delta))=F(\s_1(\delta))\]
for any $\delta\in\EE_1(X,Y)$, where $F$ acts on conflations in the natural way. 
\end{definition}

Some authors call such pairs $(F, \alpha)$ \emph{exact functors} of extriangulated categories. We suppress $\alpha$ from the notation if it is clear from the context, but write the functor as the pair $(F,\alpha)$ when we wish to emphasise this extra data. Observe that if $\FF\leq\EE$ is an extriangulated substructure of $(\cal{C},\EE,\s)$, then $(\id_{\cal{C}},\iota)$ is an extriangulated functor, where $\iota\colon\FF\to\EE$ is the inclusion.

\begin{fact} \label{fact: equivalences}
    \begin{enumerate}
        \item \emph{(\cite[Prop.~2.13]{NOS})}
        An  extriangulated functor $(F, \alpha)$ is an equivalence of extriangulated categories, in the sense that it admits a quasi-inverse functor $(G, \beta)$, if and only if $F$ is an equivalence of additive categories and $\alpha$ is a natural isomorphism.
        \item\label{fact:equivalences-BTS} \emph{(\cite[Thm.~3.9]{BTS})} 
        Given an extriangulated category $(\cal{C}_1,\EE_1,\s_1)$ and an equivalence of additive categories $F\colon \cal{C}_1 \to \cal{C}_2$, there exists an extriangulated structure $(\cal{C}_2,\EE_2,\s_2)$ on $\cal{C}_2$ and a natural isomorphism $\alpha$ such that the pair $(F, \alpha)$ is an equivalence of extriangulated categories $(\cal{C}_1,\EE_1,\s_1) \to (\cal{C}_2,\EE_2,\s_2)$.
    \end{enumerate}
\end{fact}

In the setting of extriangulated categories, various reasonable definitions of cluster-tilting subcategory turn out to be inequivalent.
For our purposes, we choose the one that is closest to the original definition of~\cite[Defn.~2.2]{Iyama-HigherAR}.

\begin{definition}
\label{def:CT}
 A full subcategory~$\cal{T}$ of~$\cal{C}$ is \emph{cluster-tilting} if it is functorially finite (i.e.\ it admits left and right approximations) and satisfies
\begin{align*}
 \cal{T}  &= \{X\in\cal{C}:\text{$\mathbb{E}(X,T)=0$ for all $T\in\cal{T}$}\} \\
          &= \{X\in\cal{C}:\text{$\mathbb{E}(T,X)=0$ for all $T\in\cal{T}$}\}.
\end{align*}
\end{definition}

\begin{remark}
\label{r:sff}
 Most of the time, we will be assuming that~$\cal{C}$ has enough projectives and enough injectives, in which case any cluster-tilting subcategory is \emph{strongly functorially finite} in the sense of \cite{zhou2018triangulated}: it admits right approximations that are deflations and left approximations that are inflations \cite[Rem.~3.22]{HLNII}. In general, a cluster-tilting subcategory as defined here may not be strongly functorially finite.
\end{remark}

\subsection{Extriangulated ideal quotients}

If~$\cC$ is an additive category and~$J$ is an ideal of~$\cC$, then the obvious functor~$\cC\to \cC/J$ is called an~\emph{ideal quotient}.  We consider a generalization of this notion to extriangulated functors.

\begin{definition}
We say that an extriangulated functor $(F, \alpha)\colon (\cal{C}_1,\EE_1,\s_1) \to (\cal{C}_2,\EE_2,\s_2)$ is an \emph{extriangulated ideal quotient} if $F$ is an ideal quotient of additive categories and $\alpha$ is the identity. 
\end{definition}

\begin{remark}
If $(F, \alpha)\colon (\cal{C}_1,\EE_1,\s_1) \to (\cal{C}_2,\EE_2,\s_2)$ is an extriangulated ideal quotient and $G\colon \cal{C}_2 \to \cal{C}_3$ is an additive equivalence, then by \cref{fact: equivalences}\ref{fact:equivalences-BTS} there exists an extriangulated structure on $\cal{C}_3$ such that $GF$ admits a natural upgrade to an extriangulated functor. In a sense, the data of
\begin{itemize}
    \item the category $ (\cal{C}_1,\EE_1,\s_1)$,
    \item  the ideal $\ker F$, and
    \item the equivalence $G$
\end{itemize}
give rise to an extriangulated structure on the category $\cal{C}_3$.
\end{remark}

We will now present a novel construction  of extriangulated ideal quotients which we will use throughout the paper.

Let $(\cal{C},\mathbb{E},\s)$ be an extriangulated category, let $J_0$ be a class of morphisms with injective domains and projective codomains, and let $J$ be the ideal generated by $J_0$.
We write $\underline{f}$ for the class in $\cal{C}/J$ of a morphism $f\in\Mor\cal{C}$.

Because any morphism $f\in J$ satisfies $\mathbb{E}(f,\blank)=0$ and $\mathbb{E}(\blank,f)=0$, there is an induced additive bifunctor $\mathbb{E}_J \colon (\cal{C}/J)^\text{op} \times \cal{C}/J \to \Ab$ such that
\begin{itemize}
\item for any $X,Z\in\cal{C}$, $\mathbb{E}_J(Z,X) = \mathbb{E}(Z,X)$,
\item for any $X,X',Z\in\cal{C}$, any $\delta\in\mathbb{E}(Z,X)$ and any $f:X\to X'$ in $\mbox{Mor} \cal{C}$, $(\underline{f})_\ast\delta = f_\ast\delta$, and
\item for any $X,Z,Z'\in\cal{C}$, any $\delta\in\mathbb{E}(Z,X)$ and any $g:Z'\to Z$ in $\mbox{Mor} \cal{C}$, $(\underline{g})^\ast\delta = g^\ast\delta$.
\end{itemize}

For any $X,Z\in\cal{C}$ and $\delta\in\mathbb{E}(Z,X)$, we let $\underline{\s}(\delta)$ be the class $[X\overset{\underline{i}}{\to}Y\overset{\underline{p}}{\to}Z]$, where $\s(\delta)=[X\overset{i}{\to}Y\overset{p}{\to}Z]$.
This assignment is well-defined, and yields an additive realization $\underline{\s}$ of $(\cal{C},\mathbb{E}_J)$.
In other words, conflations in $\cal{C}/J$ are of the form
$X\overset{\underline{i}}{\infl}Y\overset{\underline{p}}{\defl}Z\overset{\delta}{\dashrightarrow}$, for
$X\overset{i}{\infl}Y\overset{p}{\defl}Z\overset{\delta}{\dashrightarrow}$
a conflation in $\cal{C}$.

\begin{theorem}
\label{t:extriangulatedQuotient}
Endowed with the structure defined above, $\cal{C}/J$ is extriangulated and the projection functor $\cal{C} \to \cal{C}/J$ is an extriangulated ideal quotient.
\end{theorem}

\begin{proof}
Axioms (ET1) and (ET2) follow from the definitions before \cref{t:extriangulatedQuotient} and axioms (ET4) and (ET4\textsuperscript{op}) immediately follow from the corresponding axioms in $\cal{C}$.
We give a detailed proof of axioms (ET3) and (ET3\textsuperscript{op}) so as highlight the role played by our assumptions on the morphisms in $J_0$.
We note the importance of factoring first through an injective and then through a projective.

Any commutative diagram in $\cal{C}/J$ of the form
\[\begin{tikzcd}
X \arrow[rightarrowtail]{r}{\underline{i}} \arrow{d}[swap]{\underline{f}} &
Y \arrow[twoheadrightarrow]{r}{\underline{p}} \arrow{d}{\underline{g}} &
Z \arrow[dashrightarrow]{r}{\delta} &
\phantom{\Sigma X} \\
X' \arrow[rightarrowtail]{r}{\underline{j}} &
Y' \arrow[twoheadrightarrow]{r}{\underline{q}} & 
Z' \arrow[dashrightarrow]{r}{\delta'} &
\phantom{\Sigma X'}
\end{tikzcd}\]
lifts to a diagram
\[\begin{tikzcd}[row sep=tiny]
X \arrow[rightarrowtail]{rrr}{i} \arrow{ddd}[swap]{f} \arrow{dr}[swap]{b} &
&
&
Y \arrow[twoheadrightarrow]{r}{p} \arrow{ddd}{g} \arrow[dotted]{dll}{c} &
Z \arrow[dashrightarrow]{r}{\delta} &
\phantom{\Sigma X} \\
&  I \arrow{dr}{j_0} & & & & \\
& & P \arrow{dr}{a} & & & \\
X' \arrow[rightarrowtail]{rrr}[swap]{j} & & &
Y' \arrow[twoheadrightarrow]{r}{q} & 
Z' \arrow[dashrightarrow]{r}{\delta'} &
\phantom{\Sigma X'}
\end{tikzcd}\]
in $\cal{C}$, where $j_0\in J_0$ and $jf=gi+aj_0b$.
Because $I$ is injective, there is a morphism $c$ (dotted in the diagram) such that $b=ci$.
Replacing $g$ by $g+aj_0c$, we obtain a commutative square in $\cal{C}$.
We can now apply axiom (ET3) in $\cal{C}$ and conclude that (ET3) holds in $\cal{C}/J$, using that $\underline{g+aj_0c}=\underline{g}$.
The proof for (ET3\textsuperscript{op}) is similar and relies on a diagram of the form:
\[\begin{tikzcd}[row sep=tiny]
Y \arrow[twoheadrightarrow]{rrr}{p} \arrow{ddd}[swap]{f} \arrow{dr} &
&
&
Z \arrow{ddd}{g} \\
&  I \arrow{dr} & & \\
& & P \arrow{dr} \arrow[dotted]{dll} & \\
Y' \arrow[twoheadrightarrow]{rrr}[swap]{q} & & &
Z'
\end{tikzcd}\]
\end{proof}

\begin{remark}
While this construction might not look symmetric in terms of the direction of morphisms, it actually is: the ideal $J^{\op}$ in $\cal{C}^{\op}$ is also generated by morphisms with injective domain and projective codomain, since injectives and projectives swap places in passing from $\cal{C}$ to $\cal{C}^{\op}$. We also have $\cal{C}^{\op} / J^{\op} = (\cal{C}/J)^{\op}$ (meaning that we have an equivalence given by the identity on objects and on morphisms). 
\end{remark}

\begin{corollary} \label{cor: 6term_in_quotient}
For any conflation $X\overset{\underline{f}}{\infl}Y\overset{\underline{g}}{\defl}Z\overset{\delta}{\dashrightarrow}$ 
realising $\delta \in \EE_J(Z, X) = \EE(Z, X)$, one has two associated exact sequences of natural transformations
\begin{equation} \label{eq:quotient_6term1}
\cC/J(Z,\blank)\overset{\blank\circ \underline{g}}{\longrightarrow}\cC/J(Y,\blank)\overset{\blank\circ \underline{f}}{\longrightarrow}\cC/J(X,\blank)\overset{\delta^\sharp}{\longrightarrow}\EE_J(Z,\blank)\overset{\underline{g}^\ast}{\longrightarrow}\EE_J(Y,\blank)\overset{\underline{f}^\ast}{\longrightarrow}\EE_J(X,\blank),
\end{equation} 
\begin{equation} \label{eq:quotient_6term2}
\cC/J(\blank,X)\overset{\underline{f}\circ\blank}{\longrightarrow}\cC/J(\blank,Y)\overset{\underline{g}\circ\blank}{\longrightarrow}\cC/J(\blank,Z)\overset{\delta_\sharp}{\longrightarrow}\EE_J(\blank,X)\overset{\underline{f}_\ast}{\longrightarrow}\EE_J(\blank,Y)\overset{\underline{g}_\ast}{\longrightarrow}\EE_J(\blank,Z). 
\end{equation}
\end{corollary}

\begin{proof}
These sequences are just \eqref{eq:6term1} and \eqref{eq:6term2} for the extriangulated category $\cC/J$.
\end{proof}

\begin{remark}
As a general observation, which we will discuss in detail in our examples of $0$-Auslander categories, let us note that by taking relative extriangulated structures on their domain, one can interpret some interesting ideal quotients as extriangulated ideal quotients. More precisely, given an ideal $J$ of morphisms in an extriangulated category $(\cal{C},\EE,\s)$, there may exist an extriangulated substructure $(\cal{C},\FF,\s|_{\FF})$ such that in this substructure $J$ is generated by morphisms with injective domain and projective codomain. In this case, if we equip $\cal{C}$ with this relative extriangulated structure then the ideal quotient $\cal{C}\to\cal{C}/J$ is extriangulated. See \cref{t:ClustCat-to-2term} for an example of this phenomenon.
\end{remark}

\begin{example} \label{ex:quotient_by_proj-inj}
Theorem \ref{t:extriangulatedQuotient} recovers \cite[Prop.~3.30]{NakaokaPalu}: the ideal quotient of an extriangulated category by any full additive subcategory closed under isomorphisms and direct summands, and whose objects are both projective and injective, carries an induced extriangulated structure. It is immediate that the quotient functor is then an extriangulated ideal quotient. See also \cite[Rem.~3.35]{NOS}.

\end{example}

\begin{remark} 
After we shared Theorem \ref{t:extriangulatedQuotient}  with Xiaofa Chen, he proved \cite[Prop.~6.22]{chen2023exact} that if $\cal{C}$ is algebraic, then so is $\cal{C}/J$. Here algebraicity means that the category is equivalent to the quotient of an exact category by a class of projective-injective objects. Equivalently, the category admits an enhancement by a dg category endowed with an exact dg structure in the sense of \cite{chen2023exact}. In the applications we consider in Sections \ref{section:Frobenius}, \ref{section:Higgs}, \ref{section:Examples} and \ref{section:gentle}, all the extriangulated categories involved are in fact algebraic.
\end{remark} 

\begin{definition}
 Recall that an object~$X\in\cal{C}$ is called \emph{rigid} if~$\mathbb{E}(X,X)=0$.
 It is \emph{maximal rigid} if $X\oplus Y$ is rigid if and only if~$Y\in\add X$.
 
\end{definition}

\begin{proposition}
\label{prop::rigid}
Let $\cal{C}$ be an extriangulated category, and $J$ an ideal generated by morphisms with injective domain and projective codomain.
\begin{enumerate}
\item The extriangulated functor $F\colon \cal{C} \to \cal{C}/J$ sends (maximal) rigid objects to (maximal) rigid objects.
\item\label{p:rigid_proj} For an object $X \in \cC$, the object $F(X)$ is projective (resp.\ injective) in $\cC/J$ if and only if $X$ is projective (resp.\ injective) in $\cC$.
\end{enumerate}
\end{proposition}

\begin{proof}
Both parts follow immediately from the fact that $\mathbb{E}(Z,X)=\mathbb{E}_J(Z,X)$ for~$X, Z\in\cal{C}$.
\end{proof}

\begin{remark} \label{rem:various_quotients}
Assume that in $(\cal{C},\mathbb{E},\s)$ we have enough projectives $\cal{P}$ and enough injectives $\cal{I}$. Let $J$ be the ideal generated by all morphisms with injective domain and projective codomain. Then we have the diagram
\[\begin{tikzcd}[column sep=5pt]
& \cal{C} \arrow[twoheadrightarrow]{d} & \\
& \cal{C}/(\cal{P} \cap \cal{I}) \arrow[twoheadrightarrow]{d} & \\
& \cal{C}/J \arrow[twoheadrightarrow]{dl} \arrow[twoheadrightarrow]{dr} & \\
\underline{\cal{C}} & & \overline{\cal{C}}.
\end{tikzcd}\]
of ideal quotients of additive categories.

As we discussed in Theorem \ref{t:extriangulatedQuotient} and Example \ref{ex:quotient_by_proj-inj}, the categories $\cal{C}/(\cal{P} \cap \cal{I})$ and $\cal{C}/J$ carry extriangulated structures induced from $(\cal{C},\mathbb{E},\s)$, and the vertical functors $\cal{C} \twoheadrightarrow \cal{C}/(\cal{P} \cap \cal{I}) \twoheadrightarrow \cal{C}/J$ are extriangulated ideal quotients. The ideal quotients of additive categories $\cal{C}/J \to \underline{\cal{C}}$ and $\cal{C}/J \to \overline{\cal{C}}$ cannot generally be upgraded to extriangulated ideal quotients with the domain $(\cal{C}/J, \mathbb{E}_J, \underline{\s})$. To obtain $\underline{\cal{C}}$ and $\overline{\cal{C}}$, one could first take the maximal relative structure of $(\cal{C}/J, \mathbb{E}_J, \underline{\s})$ (or of $(\cal{C}, \mathbb{E}, \s)$) making all projective objects injective, resp.\ all injective objects projective, and then take the quotients by the ideals of morphisms factoring through these projective-injectives in these relative structures.
This procedure does upgrade additive ideal quotients to extriangulated ideal quotients, but at the expense of first taking a relative structure on the domain. See \cite{ogawa2022localization} for a detailed discussion of a similar procedure for functors with triangulated domain. See also the discussion in \cref{ss:costable} for the case of $0$-Auslander categories, the (co)stable categories of which are, in fact, abelian.
\end{remark}

\begin{remark}
It may appear natural to consider the quotient of $(\cal{C},\mathbb{E},\s)$ by the ideal $J' = ( \cal{I}) \cap (\cal{P})$ of morphisms which factor both through an injective and though a projective object (maybe in two different factorisations). Indeed, this quotient admits a natural additive bifunctor induced from $\EE$ by the quotient functor $F'\colon \cal{C} \to \cal{C}/J'$, since any  morphism $f \in J'$ satisfies $\EE(f, \blank) = 0$ and $\EE(\blank, f) = 0$. However, it does not always admit an extriangulated structure such that the ideal quotient functor $F'$ is extriangulated (when coupled with the identity natural transformation); in particular, $J'$ may not be generated by morphisms with injective domain and projective codomain. This is illustrated by the following example.
\end{remark}

\begin{example}
\label{e:3-cycleQuotient}
Let $Q$ be the cyclic quiver with $3$ vertices, and $I$ the ideal generated by all paths of length~$2$. In the category $\cal{C} = \cal{K}^{[-1,0]}(\proj kQ/I)$ there are no non-zero morphisms with injective domain and projective codomain, so 
\[\frac{\cal{C}}{( \cal{I} \to \cal{P})} = \cal{C}.\]
The Auslander--Reiten quiver of this category is drawn in Figure \ref{fig:AR_for_harp_cyclic_A3}.
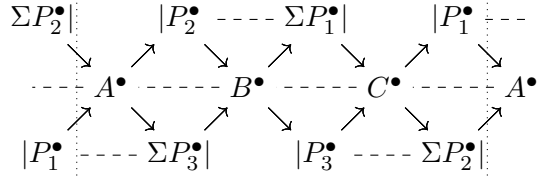
\begin{figure}\centering
\hspace{-3.3em}\begin{tikzpicture}[scale=0.9]
\foreach \n/\x/\y/\o in {1/0/0/|P_1^{\bullet}, 2/2/0/\Sigma P_3^{\bullet}|, 3/4/0/|P_3^{\bullet}, 4/6/0/\Sigma P_2^{\bullet}|, 5/1/1/A^{\bullet}, 6/3/1/B^{\bullet}, 7/5/1/C^{\bullet}, 8/2/2/|P_2^{\bullet}, 9/4/2/\Sigma P_1^{\bullet}|, 4b/0/2/\Sigma P_2^{\bullet}|, 1b/6/2/|P_1^{\bullet}, 5b/7/1/A^{\bullet}}
 {\draw (\x,\y) node (\n) {${\o}$};}
 \draw (-0.3,1) node (l) {};
 \draw (7.2,2) node (ru) {};
 \draw (7.2,0) node (rd) {};
\foreach \n/\i in {1/1, 5/2, 8/3}
 \foreach \s/\t in {1/5, 5/8, 5/2, 8/6, 2/6, 6/9, 6/3, 9/7, 3/7, 7/4, 4b/5, 7/1b, 4/5b, 1b/5b}
 {\path[->] (\s) edge (\t);}
 \foreach \l/\r in {1/2, 3/4, 5/6, 6/7, 8/9, l/5, 7/5b, 1b/ru}
 {\path[dashed] (\r) edge (\l);}
 \draw[dotted] (0.5,-0.3) -- (0.5,2.3);
 \draw[dotted] (6.5,-0.3) -- (6.5,2.3);
 \end{tikzpicture}
    \caption{The Auslander--Reiten quiver of the category $\cal{K}^{[-1,0]}(\proj kQ/I)$, for $Q$ a $3$-cycle and $I$ generated by all paths of length $2$. Here $P_i^{\bullet}$ denotes the stalk complex with the indecomposable projective object $P_i \in \mod kQ/I$ in degree $0$, for $i = 1, 2, 3$.}
    \label{fig:AR_for_harp_cyclic_A3}
\end{figure}

This example also serves to show that the functor~$\cC\to \frac{\cal{C}}{( \cal{I}) \cap (\cal{P})}$ cannot be an extriangulated ideal quotient, as we will now see. For each square mesh in the Auslander--Reiten quiver, the composition of morphisms on one side factors through a projective object (i.e.\ an object of the form $P_i^{\bullet}$) and the composition on the other side factors through an injective object (i.e.\ an object of the form $\Sigma P_j^{\bullet}$). Thus, in the quotient by $( \cal{I}) \cap (\cal{P})$, these compositions are sent to zero. It follows that the class of a non-split conflation represented by such a square is not a weak kernel--cokernel pair in $\frac{\cal{C}}{( \cal{I}) \cap (\cal{P})}$, and hence there is no extriangulated structure on this quotient in which it is a conflation.

Concretely, we have
\[\frac{\cal{C}}{( \cal{I}) \cap (\cal{P})} (A^{\bullet}, B^{\bullet}) = 0 \neq \cal{C}(A^{\bullet}, B^{\bullet}).
\]
Applying $\frac{\cal{C}}{( \cal{I}) \cap (\cal{P})} (\blank, B^{\bullet})$ to the pair of composable morphisms $A \to P_2^{\bullet} \oplus \Sigma P_3^{\bullet} \to B^{\bullet}$, the resulting $3$-term sequence is not exact at the middle term. Thus, the class of this pair of composable morphisms is not a weak kernel--cokernel pair in the additive category $\frac{\cal{C}}{( \cal{I}) \cap (\cal{P})}$. 
\end{example}

\section{\texorpdfstring{$0$-Auslander extriangulated categories}{0-Auslander extriangulated categories}} \label{section:Auslander}

\subsection{Generalities}
Following intuition from \cite{PPPP}, a class of particularly nice extriangulated categories, named $0$-Auslander, was studied in detail in \cite{GNP2}.

\begin{definition} \label{def:0-Auslander}
An extriangulated category $(\cal{C}, \EE, \s)$ is said to be \emph{$0$-Auslander} if it is hereditary, has enough projectives, and has dominant dimension at least one. Explicitly, this means that

\begin{enumerate}
\item for each object $X \in \cal{C}$, there exists a conflation $P_1 \rightarrowtail P_0 \twoheadrightarrow X$ where $P_1$ and $P_0$ are projective in $\cal{C}$, and
\item for each projective object $P \in \cal{C}$, there exists a conflation $P \rightarrowtail Q \twoheadrightarrow I$ where $Q$ is projective-injective and $I$ is injective in $\cal{C}$.
\end{enumerate}    
\end{definition}

By \cite[Prop.~3.6]{GNP2}, this happens if and only if $\cal{C}$ is hereditary, has enough injectives, and has codominant dimension at least one, and so the definition is self-dual.

$0$-Auslander categories have the following properties.

\begin{itemize}
\item The functor sending $X\in\cal{C}$ to $\frac{\cal{C}}{(\cal{P}\cap\cal{I})} (-,X)|_{\overline{\cal{P}}}$ induces an equivalence of categories
\[
\cal{C}/\cal{I} \simeq \mod \overline{\cal{P}},
\]
where $\overline{\cal{P}}$ is the image of $\cal{P}$ in $\cal{C}/{(\cal{P}\cap\cal{I})}$ (see \cite[Prop.~3.11]{GNP2}).
\item Exchanging $P$ and $I$ in conflations $P\infl Q\defl I$, with $Q$ projective-injective, gives  quasi-inverse equivalences of additive categories $\Sigma \colon \overline{\cal{P}} \overset{\sim}{\leftrightarrows} \underline{\cal{I}} \cocolon \Omega$, for $\overline{\cal{P}}$ as above and $\underline{\cal{I}}$ the image of $\cal{I}$ in $\cal{C}/(\cal{P}\cap\cal{I})$ (see \cite[Lem.~3.12]{GNP2}). 
\item Under suitable conditions on $\cal{C}$, several natural notions of maximality for rigid objects agree. In particular, a presilting object is silting if and only if it is tilting, if and only if it is cotilting, if and only if it is complete rigid (see \cite[Thm.~4.3]{GNP2}).
\item There is a natural theory of irreducible mutations of silting objects and subcategories. Notably, there is precisely one exchange triangle for every indecomposable summand of a silting object, and so the difference between left and right mutations is transparent (see \cite[Thm.~4.26, Cor.~4.28]{GNP2}).
\end{itemize}

Various examples of $0$-Auslander categories are discussed in \cite[\S3.3]{GNP2} and \cite[\S5]{Palu2023}. Some prototypical examples include
     \begin{itemize}
         \item Amiot's cluster categories \cite{Amiot09} with certain relative structures,
        \item  certain extriangulated categories related to gentle algebras, blossoming quivers, and walks \cite[\S6]{GNP2} are $0$-Auslander, which will be discussed in Section \ref{section:gentle},
		\item  extended cohearts of co-t-structures in the sense of \cite{PZ}, and
        \item in particular, the full extension-closed subcategory  $\add A \ast \add \Sigma A \subset \per A$, for a non-positive dg algebra $A$.
     \end{itemize}
As a special case of the final example, if $A$ is an ordinary algebra (i.e.\ a dg algebra concentrated in degree $0$), the category $\harp{A} \subset \Kb{\proj A}$ of $2$-term complexes of projectives up to homotopy is $0$-Auslander. Similarly, the category of morphisms of projectives is $0$-Auslander.
 
\subsection{Ideal quotients}

\begin{proposition} \label{prop:quotient_0-Auslander_parts}
Let $(\cal{C}, \EE, \s)$ be an extriangulated category and let $J$ be an ideal generated by morphisms with injective domain and projective codomain.
\begin{enumerate}
\item\label{p:quot_0-Auslander_hereditary} If $(\cal{C}, \EE, \s)$ is hereditary, then so is $\cal{C}/J$. 
\item\label{p:quot_0-Auslander_enough} If $(\cal{C}, \EE, \s)$ had enough projectives, resp.\ enough injectives, then so does $\cal{C}/J$.
\item\label{p:quot_0-Auslander_domdim} If  $(\cal{C}, \EE, \s)$ has (co)dominant dimension at least 1, then so does $\cal{C}/J$.
\end{enumerate}
\end{proposition}

\begin{proof}
By \cite[Prop.~2.1]{GNP2}, one of equivalent definitions for an extriangulated category to be hereditary is the right exactness of the defining bifunctor in the first argument. By definition of $\EE_J$, the sequence \eqref{eq:6term1} becomes exact at $\EE(X, \blank)$ when extended by the $0$ natural transformation if and only if so does the sequence \eqref{eq:quotient_6term1}. This proves part \ref{p:quot_0-Auslander_hereditary}.

Parts \ref{p:quot_0-Auslander_enough} and \ref{p:quot_0-Auslander_domdim} follow from Proposition \ref{prop::rigid}\ref{p:rigid_proj} combined with the fact that $F$ preserves and reflects conflations.
\end{proof}

By comparing Proposition \ref{prop:quotient_0-Auslander_parts} with Definition \ref{def:0-Auslander}, we immediately get the following.

\begin{corollary}
 If  $(\cal{C}, \EE, \s)$ is $0$-Auslander, then so is $\cal{C}/J$.
\end{corollary}

\begin{definition} \label{def:tilting}
Let $(\cal{C}, \EE, \s)$ be a $0$-Auslander extriangulated category.
A subcategory  $\cal{S}\subseteq\cal{C}$ is \emph{tilting} if $\EE(\cal{S},\cal{S})=0$ and for any projective object $P\in\cal{C}$ there is a conflation $P\infl S_0\defl S_1$ with $S_0,S_1\in\cal{S}$. An object $S \in \cal{C}$ is tilting if the category $\add S$ is tilting.
\end{definition}

Moreover, in a $0$-Auslander category being tilting is equivalent to being silting; see \cite[Prop.~5.5]{AdachiTsukamoto} or \cite[Thm.~4.3]{GNP2} for the definition and the proof. A theory of irreducible left and right mutations of silting subcategories in $0$-Auslander categories is given in \cite[Cor.~4.28]{GNP2}, and can be seen as a $2$-term version of the silting mutation considered in \cite{AdachiTsukamoto}.

\begin{proposition} \label{prop:poset_silting}
Assume that $(\cal{C}, \EE, \s)$ is $0$-Auslander. A subcategory $\cal{S}\subseteq\cal{C}$ is tilting (equivalently, silting) if and only if $F(\cal{S})$ is tilting.  Moreover, the functor $F\colon \cal{C}\to\cal{C}/J$ preserves and reflects irreducible mutations.
\end{proposition}

\begin{proof}
The first statement follows from Proposition \ref{prop::rigid} combined with the fact that $F$ preserves and reflects conflations. The second statement follows from the description of mutations via exchange conflations in \cite[Cor.~4.28]{GNP2}, again combined with $F$ preserving and reflecting conflations. 
\end{proof}

As mentioned earlier, for an algebra $A$, the category $\harp{\proj A} \subset \Kb{\proj A}$ of $2$-term complexes of projectives up to homotopy is $0$-Auslander. We will see in later sections that such categories often appear as quotients of more general $0$-Auslander categories by the ideal generated by all morphisms from injective to projective objects. This motivates the following problem.

\begin{problem}
\label{p:0-auslander}
Determine whether the quotient of any (possibly non-algebraic) $0$-Auslander
extriangulated category by the ideal generated by morphisms from injective objects to projective
objects is equivalent, as an extriangulated category, to $\harp{\cal{A}}$ for some additive category $\cal{A}$.
\end{problem}

Xiaofa Chen \cite{chen20230} resolves Problem \ref{p:0-auslander} positively in the case of algebraic $0$-Auslander categories. To further investigate Problem~\ref{p:0-auslander}, it would be interesting to find families of non-algebraic~$0$-Auslander extriangulated categories.

\subsection{(Co)stable categories and comparison of bifunctors} \label{ss:costable}

Let $\cal{C}$ be a $0$-Auslander extriangulated category.
Then the costable category $\overline{\cal{C}}$ is abelian, and the restricted Yoneda functor $X\mapsto\Hom(\blank, X)|_\cal{P}\colon \cal{C} \to \overline{\cal{C}}$  is half-exact. Let us take the maximal relative structure $(\mathbb{F}, \s_{\mathbb{F}})$ on $\cal{C}$ with respect to which this half-exact functor becomes exact. This is a well-defined extriangulated structure \cite{Sakai}. Moreover, this is precisely the maximal relative structure where injective objects become projective, which was discussed in \cref{rem:various_quotients}. Notably, this means that $\overline{\cal{C}}$ is the quotient of $\cal{C}$ by certain projective-injective objects in this relative structure. Since the conflations in the relative structure are precisely the lifts of short exact sequences in $\overline{\cal{C}}$, and each short exact sequence admits such a lift, we have
\[
\mathbb{F}(A, B) \cong \Ext^{1}(\Hom(\blank,  A)|_{\cal{P}}, \Hom(\blank, B)|_{\cal{P}})
\]
for all $A, B \in \cal{C}$, where $\Ext^1$ is the usual bifunctor of first extensions in the abelian category $\overline{\cal{C}}$.

This is a somewhat surprising fact. Indeed, here is what it means in the case of the $2$-term complexes of projectives up to homotopy over a finite-dimensional algebra.

\begin{lemma}
Given a finite-dimensional algebra $A$ over a field $\K$, the bifunctor \[\Ext^1_{A}\colon (\mod A)^{\op} \times \mod A \to \mod\K\] can be lifted to a bifunctor $\mathbb{F}\colon \har{\proj A}^{\op} \times \har{\proj A} \to \mod\K.$ Moreover, $\mathbb{F}$ is a sub-bifunctor of the bifunctor $\mathbb{E}$ defining the standard extriangulated structure on $\har{\proj A}$.
\end{lemma}

Note that one could have rather expected that $\Ext^1_{ A}$ might be a quotient of $\mathbb{E}$ rather than a sub-bifunctor. Indeed, the Auslander--Reiten formulas (see, e.g., \cite[Thm.~IV.2.13]{assem_skowronski_simson_2006}) state that for any two $A$-modules $M, N$, we have an isomorphism
\begin{equation} \label{eq:AR_formulas}
\Ext^1_{A}(M, N) \cong D\injstabHom_{A}(N, \tau M),
\end{equation}
where $\tau$ denotes the Auslander--Reiten translation.
At the same time, it is known that for projective presentations $P_M, P_N$ of $M$ and $N$, we have an isomorphism
\begin{equation} \label{eq:PG_formula}
\mathbb{E}(P_M, P_N) \cong D\Hom_{A}(N, \tau M \oplus I),
\end{equation}
for some finite-dimensional injective $A$-module $I$ which vanishes if the projective presentation $P_M$ is minimal (see \cite[Lem.~2.6]{Plamondon13}). 
It is thus tempting to think that $\Ext^1_{ A}(\blank, \blank)$ can be obtained by taking a quotient of $\mathbb{E}$ by the subspace dual to that of certain morphisms factoring through injective $A$-modules.
This turns out to be a wrong---or, rather, a non-functorial---perspective: the subspace of morphisms one has to quotient out does not have a well-defined lift to the category $\har{\proj A}$. Note also that neither $\tau$ nor the isomorphism \eqref{eq:PG_formula} are functorial, while the isomorphism \eqref{eq:AR_formulas} is functorial in both arguments. 

A similar discussion can be made about stable categories of $0$-Auslander categories.

\subsection{Negative extensions}
For a class of $0$-Auslander categories embedded, as additive categories, into triangulated categories, a version of negative extensions in the sense of \cite{GNP1} (see also \cite{adachi2023intervals}) has been defined in \cite{GNP2}. These are bivariant $\delta$-functors formed by additive bifunctors $\EbbE^i$ for $i < 0,$ together with connecting maps natural with respect to extensions in each argument. 
The intuition behind the construction of $\EbbE^{-1}$ stemmed from the construction in \cref{section:quotients}. Let us briefly explain the setting and the relationship with the present work. 

Let $\cD$ be a triangulated category with a rigid, full subcategory $\cT$.
Define $\cC$ to be the full subcategory $\cT\ast\Sigma\cT$ of $\cD$.
Endow $\cD$ with the largest relative extriangulated structure $(\cD, \EE^{\Sigma \cT}, \s|_{\EE^{\Sigma\cal{T}}})$ making objects in $\Sigma\cT$ injective.
Then $\cC$ is extension-closed in $\cD$ for this extriangulated structure, and thus inherits an extriangulated structure.  The latter is $0$-Auslander, and has $\cT$ and $\Sigma\cT$ as the full subcategories of projectives and of injectives, respectively. The same structure was defined in \cite{GNP2} by restricting the largest relative structure on $\cD$ making objects in $\cT$ projective.

It is proved in \cite[\S7.4]{GNP2}  that the bifunctor $\EbbE^{-1}\colon \cC^{\op} \times \cC \to \Ab$ defined on objects by the formula
\[
 \EbbE^{-1}(C,A) := \frac{\cD}{( \Sigma\cT\to\Sigma^{-1}\cT)}(C,\Sigma^{-1}A),
\]
defined on morphisms in a natural way induced from $\cD(\blank, \blank)\colon \cD^{\op} \times \cD \to \Ab$, and considered together with certain natural transformations $\delta^\sharp_{-1}, \delta_\sharp^{-1}$, gives a \emph{first negative extension} on $(\cC, \EE^{\Sigma \cT}, \s|_{\EE^{\Sigma\cal{T}}})$; see \cite{adachi2023intervals} for the definition.

Note that $\EbbE^{-1}\colon \cC^{\op} \times \cC \to \Ab$ is just a restriction of the bifunctor $\frac{\cD}{( \Sigma\cT\to\Sigma^{-1}\cT)}$ to the full subcategory $\Sigma^{-1}\cT \ast \cT \ast \Sigma \cT \subset \cD$, considered with the maximal structure making objects in $\Sigma \cT$ injective. It is straightforward to check that in this subcategory, $\Sigma^{-1} \cT$ is the full subcategory of projectives. Thus, the bifunctor $\cD/( \Sigma\cT\to\Sigma^{-1}\cT)$, when restricted to this subcategory, is nothing by the quotient of the $\Hom$-bifunctor by the ideal of morphisms with injective domain and projective codomain. The proof given in \cite[\S7.4]{GNP2} of the fact $(\EbbE^{-1}, \delta^\sharp_{-1}, \delta_\sharp^{-1})$ is a first negative extension amounts, roughly speaking, to verifying that suitable lifts of the sequences \eqref{eq:quotient_6term1} and \eqref{eq:quotient_6term2} from  $\frac{\Sigma^{-1}\cT \ast \cT \ast \Sigma \cT}{(\Sigma\cT\to\Sigma^{-1}\cT)}$ to $\Sigma^{-1}\cT \ast \cT \ast \Sigma \cT$ (both considered with the maximal structures making objects in $\Sigma \cT$ injective) are exact.

\section{\texorpdfstring{Relative $2$-term categories}{Relative 2-term categories}}
\label{s:rel-2-term}

Let $\cal{A}$ be a
dg category, and $\per\cal{A}$ the category of perfect dg $\cal{A}$-modules, with morphisms considered up to homotopy. Then the Yoneda functor\footnote{The character $\yo$, pronounced `yo', is the first in the hiragana spelling of Yoneda (\yo\jpne\da). In \cite{RiehlVerity},
Riehl and Verity acknowledge Ravenel for this notation.} $\yo\colon\cal{A}\to\per{\cal{A}}$ defined by $\yo X=\Hom_{\cal{A}}(\blank,X)$ is a fully faithful embedding, and we view $\cal{A}$ as a full subcategory of $\per\cal{A}$ via this embedding.  We denote by~$\cA\ast\Sigma\cA$ the full subcategory of~$\per\cA$ of objects~$X$ that sit in a triangle~$A\to X\to \Sigma A' \to \Sigma A$, where~$A,A'\in \add \cA$. Note that we abuse notation by writing $\cA\ast\Sigma\cA$ for $\add\cA\ast\add\Sigma\cA$. Since~$\cA$ is non-positively graded,~$\cA\ast\Sigma \cA$ is stable under taking direct summands by~\cite[Lem.~2.6]{IyamaYang18}.  It is also stable under extensions, and thus inherits an extriangulated structure.

Recall that an object $Q\in\per{\cal{A}}$ is in the \emph{Bongartz co-completion} of $\cal{E}\subseteq \add\cA\ast\add\Sigma\cA$ if and only if there is a triangle
\begin{equation}
\label{eq:bong-co-comp}
\begin{tikzcd}
X\arrow{r}{\varphi}&E_X\arrow{r}&E'_X\arrow{r}&\Sigma X
\end{tikzcd}
\end{equation}
in which $X\in\cal{A}$, the morphism $\varphi$ is a left $\cal{E}$-approximation of $X$, and $Q\in\add(E_X\oplus E'_X)$.

\begin{theorem}
\label{t:rel-2-term}
Let $\cal{A}$ be a non-positively graded dg category, and let $\cal{E}\subseteq\cal{A}$ be a full subcategory which is covariantly finite as a subcategory of~$\per\cA$. Consider the subcategory $\cal{A}\ast\Sigma\cal{A}\subseteq\per\cal{A}$ of $2$-term objects, and write
\[\lperp(\Sigma\cal{E})=\{X\in\cal{A}\ast\Sigma\cal{A}:\text{$\Hom_{\cal{A}}(X,\Sigma E)=0$ for all $E\in\cal{E}$}\}\]
for the left perpendicular category of~$\Sigma\cal{E}$ in~$\cal{A}\ast\Sigma\cal{A}$. Then
\begin{enumerate}
\item\label{t:rel-2-term-ext-closed} $\lperp(\Sigma\cal{E})$ is extension-closed in $\per\cal{A}$, hence extriangulated,
\item\label{t:rel-2-term-proj-inj} the projective objects in $\lperp(\Sigma\cal{E})$ are those of $\add\cal{A}$, while the injective objects are those of the Bongartz co-completion of $\cal{E}\subseteq\cal{A}\ast\Sigma\cal{A}$,
\item\label{t:rel-2-term-0Aus} $\lperp(\Sigma\cal{E})$ is a $0$-Auslander extriangulated category.
\end{enumerate}
\end{theorem}

Compare with \cite[Prop.~2.9]{GNP2}, where it was shown that~$\lperp(\Sigma\cal{E})\cap(\Sigma^{-1}\cal{E})\rperp$ is hereditary, in the case where~$\cE$ is rigid in~$\per\cA$, and with \cite[\S3.3.8]{GNP2} where it was explained that $\cal{A}\ast\Sigma\cal{A}$ is $0$-Auslander. Both results were stated for dg algebras, but apply to dg categories as well. 

\begin{theorem}
\label{t:rel-2-term-functors}
Keep notation as in Theorem~\ref{t:rel-2-term}. Then in the commutative diagram
\begin{equation}
\begin{tikzcd}[column sep=35pt]
\label{eq:$2$-term-diag}
\per{\cal{A}}\arrow{r}{\blank\ltens_{\cal{A}}\cohom{0}\cal{A}}&\per\cohom{0}\cal{A}\arrow{r}{\blank\ltens_{\cohom{0}\cal{A}}\cohom{0}\cal{A}/\cohom{0}\cal{E}}&\per\cohom{0}\cal{A}/\cohom{0}\cal{E}\\
\lperp(\Sigma\cal{E})\arrow[hookrightarrow]{u}\arrow{r}{G}&\lperp(\Sigma\cohom{0}\cal{E})\arrow[hookrightarrow]{u}\arrow{r}{F}&\harp{\proj\cohom{0}\cal{A}/\cohom{0}\cal{E}},\arrow[hookrightarrow]{u}
\end{tikzcd}
\end{equation}
the functors $F$ and $G$, given by restricting the derived tensor products, are full and essentially surjective extriangulated functors. Moreover, $\ker(F G)$ is the ideal of $\lperp(\Sigma\cal{E})$ generated by morphisms from an injective object to a projective object.
\end{theorem}

After some of the results of this paper were presented at ARTA IX in June 2023, Xiaofa Chen \cite{chen20230} proved versions of part \ref{t:rel-2-term-0Aus} of \cref{t:rel-2-term} and a part of \cref{t:rel-2-term-functors}  for exact dg categories.

\begin{remark}
When~$\cA$ is concentrated in degree zero, the leftmost square in \cref{t:rel-2-term-functors} is superfluous, since its horizontal functors are identities.
\begin{itemize}
 \item An important class of examples where this happens is when~$\cA$ is a finite-dimensional algebra~$\Lambda$ (viewed as a dg category with one object and morphisms in degree~$0$).
 \item This observation implies that we may apply Theorem~\ref{t:rel-2-term-functors} to $\cohom{0}\cal{A}$ to see that the kernel of the functor $F$ from~\eqref{eq:$2$-term-diag} is also that generated by morphisms from an injective to a projective object, now inside $\lperp{(\Sigma\cohom{0}\cal{E})}$.
\end{itemize}
\end{remark}

\begin{corollary} \label{cor:$2$-term_quotient_rigid}
The functors $F$, $G$, and $FG$ send rigid objects to rigid objects and silting objects to silting objects. Moreover, they induce isomorphisms of the posets of silting objects and hence, in particular, the Hasse graphs of these posets \cite{BrustleYang,AiharaIyama}.
\end{corollary}

\begin{proof}
Apply \cref{prop:poset_silting}.
\end{proof}

For the rest of the section we fix notation as in Theorems~\ref{t:rel-2-term} and \ref{t:rel-2-term-functors} and prove these results, dividing them into a series of lemmas.  To compress notation, we will be writing~$\Hom_{\cA}$ instead of~$\Hom_{\per\cA}$.

\begin{lemma}
\label{l:ext-closed}
The subcategory $\lperp(\Sigma\cal{E})\subseteq\cal{A}*\Sigma\cal{A}$ is extension-closed in $\per{\cal{A}}$.
\end{lemma}
\begin{proof}
This follows directly from the fact that $\Hom_{\cal{A}}(\blank,\Sigma E)$ is a cohomological functor for any object $E\in\cal{E}$.
\end{proof}

\begin{lemma}
\label{l:projectives}
An object of $\lperp(\Sigma\cal{E})$ is projective if and only if it is an object of $\add\cal{A}$.
\end{lemma}
\begin{proof}
In $\per\cal{A}$, there are no non-zero maps from an object of $\add\Sigma^i\cal{A}$ to an object of $\add\Sigma^j\cal{A}$ whenever $i<j$. In particular, this means that $\add\cal{A}\subseteq\lperp(\Sigma\cal{E})$. Let $P\in\add\cal{A}$ and $X\in\cal{A}*\Sigma\cal{A}$. Then $\Sigma X\in\Sigma(\cal{A}*\Sigma\cal{A})$ so $\Hom_{\cal{A}}(P,\Sigma X)=0$ because $\Hom_{\cal{A}}(P,\blank)$ is homological. Thus $P$ is projective in $\cal{A}*\Sigma\cal{A}$, hence in particular projective in $\lperp(\Sigma\cal{E})$.

Conversely, let $P\in\lperp(\Sigma\cal{E})\subseteq\cal{A}*\Sigma\cal{A}$ be projective, and consider a triangle
\[\begin{tikzcd}
X\arrow{r}&Y\arrow{r}&P\arrow{r}{\varphi}&\Sigma X
\end{tikzcd}\]
with $X,Y\in\add\cal{A}$. Since $X\in\add\cal{A}\subseteq\lperp(\Sigma\cal{E})$, where $P$ is projective, it follows that $\varphi=0$, and so $P$ is a summand of $Y\in\add\cal{A}$.
\end{proof}

\begin{lemma}
\label{l:injectives}
An object of $\lperp(\Sigma\cal{E})$ is injective if and only if it lies in the Bongartz co-completion of $\cal{E}$.
\end{lemma}
\begin{proof}
Applying the cohomological functor $\Hom_{\cal{C}}(\blank,\Sigma E)$ to \eqref{eq:bong-co-comp}, for any $E\in\cal{E}$, we obtain the exact sequence
\[\begin{tikzcd}
\Hom_{\cal{A}}(\Sigma E_X,\Sigma E)\arrow{r}{(\Sigma\varphi)^*}&\Hom_{\cal{A}}(\Sigma X,\Sigma E)\arrow{r}&\Hom_{\cal{A}}(E'_X,\Sigma E)\arrow{r}&\Hom_{\cal{A}}(E_X,\Sigma E).
\end{tikzcd}\]
Now $\Hom_{\cal{A}}(E_X,\Sigma E)=0$ since $E,E_X\in\add\cal{A}$, while $(\Sigma\varphi)^*$ is surjective since $\varphi$ is a left $\cal{E}$-approximation, and so $\Hom_{\cal{A}}(E'_X,\Sigma E)=0$. It follows that $\Hom_{\cal{A}}(Q,\Sigma E)=0$, i.e.\ that $Q\in\lperp(\Sigma\cal{E})$.

Now to see that $Q$ is injective in $\lperp(\Sigma\cal{E})$, choose $Y\in\lperp(\Sigma\cal{E})$ and apply $\Hom_{\cal{A}}(Y,\blank)$ to the triangle \eqref{eq:bong-co-comp}. This yields the exact sequence
\[\begin{tikzcd}
\Hom_{\cal{A}}(Y,\Sigma E_X)\arrow{r}&\Hom_{\cal{A}}(Y,\Sigma E'_X)\arrow{r}&\Hom_{\cal{A}}(Y,\Sigma^2X)
\end{tikzcd}\]
in which $\Hom_{\cal{A}}(Y,\Sigma E_X)=0$ because $Y\in\lperp(\Sigma\cal{E})$ and $E_X\in\add\cal{E}$, and $\Hom_{\cal{A}}(Y,\Sigma^2X)=0$ because $Y\in\cal{A}*\Sigma\cal{A}$ and $X\in\cal{A}$. Thus $\Hom_{\cal{A}}(Y,\Sigma E'_X)=0$, and hence $Q\in\add(E_X\oplus E'_X)$ is injective in $\lperp(\Sigma\cal{E})$.

Conversely, the Bongartz co-completion of $\cal{E}$ is silting in $\per{\cal{A}}$ (since the triangle $X\to E_X\to E'_X$ shows that the Bongartz co-completion generates $\cal{A}$, hence $\per{\cA}$) and hence maximal rigid. In particular, any injective object of $\lperp(\Sigma\cal{E})$ is contained in this Bongartz co-completion.
\end{proof}

\begin{corollary}
\label{c:injectives}
The injective objects of $\cal{A}*\Sigma\cal{A}$ are those in $\add\Sigma\cal{A}$.
\end{corollary}
\begin{proof}
This is Lemma~\ref{l:injectives} in the case that $\cal{E}=0\subseteq\per{\cA}$.
\end{proof}

\begin{lemma}
\label{l:0-Aus}
The category $\lperp(\Sigma\cal{E})$ is a $0$-Auslander extriangulated category.
\end{lemma}
\begin{proof}
This category is extriangulated as a result of Lemma~\ref{l:ext-closed}. The category $\cal{A}*\Sigma\cal{A}$ has enough projectives and all objects have projective dimension at most $1$, as a consequence of the triangles
\[\begin{tikzcd}
P_1\arrow{r}&P_0\arrow{r}&X\arrow{r}&\Sigma P_1,
\end{tikzcd}\]
with $P_i\in\cal{A}$, exhibiting that an object $X\in\per{\cal{A}}$ belongs to the subcategory $\cal{A}*\Sigma\cal{A}$. The same is therefore true for the subcategory $\lperp(\Sigma\cal{E})$, this having the same projective objects as $\cal{A}*\Sigma\cal{A}$ by Lemma~\ref{l:projectives}. Since $\cal{E}$ is covariantly finite, a triangle of the form \eqref{eq:bong-co-comp} exists for any $X\in\cal{A}$, i.e.\ any projective object in $\lperp(\Sigma\cal{E})$. By Lemma~\ref{l:injectives}, the term $E'_X$ in this triangle is injective and the term $E_X\in\cal{E}$ is projective-injective. Thus $\lperp(\Sigma\cal{E})$ has dominant dimension at least $1$.
\end{proof}

This completes the proof of Theorem~\ref{t:rel-2-term}. We observe that $\cohom{0}\cal{A}$ is also a non-positively graded dg category (concentrated in degree $0$), and that $\cohom{0}\cal{E}$ is covariantly finite in $\per\cohom{0}\cal{A}$ since $\blank\varltens{\cal{A}}\cohom{0}\cal{A}\colon\per\cA\to \per\cohom{0}\cA$ sends~$\cE$ to~$\cohom{0}\cE$ and takes left approximations to left approximations. Thus we may also apply Theorem~\ref{t:rel-2-term} (and the constituent Lemmas~\ref{l:ext-closed}--\ref{l:0-Aus}) to this input data, which we will do below in proving Theorem~\ref{t:rel-2-term-functors}.

\begin{lemma}
\label{l:G-full-dense}
The functor $G$ from \eqref{eq:$2$-term-diag} is well-defined, full and essentially surjective.
\end{lemma}
\begin{proof}
The functor $\blank\varltens{\cal{A}}\cohom{0}\cal{A}$ is essentially surjective because its restriction to~$\cA\subset\per\cA$ is an equivalence onto $\cohom{0}\cal{A}\subset\per\cohom{0}\cA$. For the same reason, it restricts to an essentially surjective functor from~$\cal{A}*\Sigma\cal{A}$ to~$\cohom{0}\cal{A}*\Sigma\cohom{0}\cal{A}=\harp{\proj\cohom{0}\cal{A}}$, and this restriction is full by a result of Brüstle and Yang \cite[Prop.~A.5]{BrustleYang}. (While their result is stated for dg algebras, the argument generalises directly to dg categories.) This implies both that the essential image of $G$ is $\lperp(\Sigma\cohom{0}\cal{E})$ as claimed, and that $G$ is itself full.
\end{proof}

While we use notation compatible with that of Theorem~\ref{t:rel-2-term-functors}, the proof of the next Lemma applies to $k$-linear categories (concentrated in degree $0$) in general, without requiring dg enhancements.

\begin{lemma}
\label{l:F-full-dense}
The functor $F$ from \eqref{eq:$2$-term-diag} is well-defined, full and essentially surjective.
\end{lemma}
\begin{proof}
The fact that $\bar{F}=\blank\varltens{\cohom{0}\cal{A}}\cohom{0}\cal{A}/\cohom{0}\cal{E}$ restricts to a full and essentially surjective functor $\proj\cohom{0}\cal{A}\to\proj\cohom{0}\cal{A}/\cohom{0}\cal{E}$ means that it also restricts to a full and essentially surjective functor $\harp{\proj\cohom{0}\cal{A}}\to\harp{\proj\cohom{0}\cal{A}/\cohom{0}\cal{E}}$.
Thus $F$ is full, and takes values in $\harp{\proj\cohom{0}\cal{A}/\cohom{0}\cal{E}}$ as claimed.

To see that $F$ is essentially surjective, choose $X=(X^{-1}\map{x}X^0)\in\harp{\proj\cohom{0}\cal{A}/\cohom{0}\cal{E}}$ and $\bar{X}=(\bar{X}^{-1}\map{\bar{x}}\bar{X}^0)\in\harp{\proj{\cohom{0}\cal{A}}}$ such that $\bar{F}\bar{X}=X$.
Let $\varphi\colon\bar{X}^{-1}\to E_X$ be a minimal left $\cohom{0}\cal{E}$-approximation of $\bar{X}^{-1}$. Since $E_X\in\cohom{0}\cal{E}\subseteq\ker(\bar{F})$, it follows that the object
\[\bar{X}'=\bar{X}^{-1}\stackrel{\left(\begin{smallmatrix}\bar{x}\\\varphi\end{smallmatrix}\right)}{\longrightarrow}\bar{X}^0\oplus E_X\]
satisfies $\bar{F}\bar{X}'=\bar{F}\bar{X}=X$.

To finish the proof, we show that $\bar{X}'\in\lperp(\Sigma\cohom{0}\cal{E})$. Pick $E\in\cal{E}$ and apply the functor $\Hom_{\cohom{0}\cal{C}}(\blank,\Sigma\cohom{0}E)$ to the triangle
\[\begin{tikzcd}
\bar{X}'=\bar{X}^{-1}\arrow{r}{\left(\begin{smallmatrix}\bar{x}\\\varphi\end{smallmatrix}\right)}&\bar{X}^0\oplus E_X\arrow{r}&\bar{X}'\arrow{r}&\Sigma\bar{X}^{-1}
\end{tikzcd}\]
to obtain the exact sequence
\[\begin{tikzcd}[ampersand replacement=\&,column sep=20pt]
\Hom_{\cohom{0}\cal{A}}(\Sigma\bar{X}^0\oplus\Sigma E_X,\Sigma E)\arrow{rr}{(\begin{smallmatrix}\bar{x}^*&\varphi^*\end{smallmatrix})}\&\&
\Hom_{\cohom{0}\cal{A}}(\Sigma\bar{X}^{-1},\Sigma E)\arrow{r}\&
\Hom_{\cohom{0}\cal{A}}(\bar{X}',\Sigma E)\arrow{r}\&0,
\end{tikzcd}\]
observing that $\Hom_{\cohom{0}\cal{A}}(\bar{X}^0\oplus E_X,\Sigma E)=0$ because $\bar{X}^0\oplus E_X$ is projective in $\lperp(\Sigma\cohom{0}\cal{E})$ as in Lemma~\ref{l:projectives}.
Since $\varphi$ is a left $\cohom{0}\cal{E}$-approximation, the first map in this sequence is surjective.
Thus $\Hom_{\cohom{0}\cal{A}}(\bar{X}',\Sigma E)=0$, and so $\bar{X}'\in\lperp(\Sigma\cohom{0}\cal{E})$.
\end{proof}

\begin{lemma}
\label{l:extri-fun}
Let $H\colon\cal{T}\to\cal{T}'$ be a triangle functor, let $\cal{C}\subseteq\cal{T}$ and $\cal{C}'\subseteq\cal{T}'$ be full and extension-closed subcategories, equipped with their canonical extriangulated structures, and assume that $H$ restricts to an additive functor $\cal{C}\to\cal{C}'$.
Then $H\colon\cal{C}\to\cal{C}'$ is an extriangulated functor. \end{lemma}
\begin{proof}
Write $\Sigma$ and $\Sigma'$ for the suspension functors on $\cal{T}$ and $\cal{T}'$, and let $\gamma\colon H\circ\Sigma \to \Sigma'\circ H$ be the natural transformation with respect to which $H$ is a triangle functor.
Then we may define a natural transformation $\Gamma\colon\EE_{\cal{C}}(\blank,\blank)\to\EE_{\cal{C}'}(H\blank,H\blank)$ with components
\[\Gamma_{X,Y}\colon\EE_{\cal{C}}(X,Y)=\Hom_{\cal{T}}(X,\Sigma Y)\to\Hom_{\cal{T}'}(HX,\Sigma'HY)=\EE_{\cal{C'}}(HX,HY)\]
defined by $\Gamma_{X,Y}(\varphi)=\gamma_Y\circ H\varphi$, for $X,Y\in\cal{C}$ and $\varphi\colon X\to\Sigma Y$. Checking that $\Gamma$ is indeed a natural transformation, and that $(H,\Gamma)\colon\cal{C}\to\cal{C}'$ is an extriangulated functor, then follows by a routine calculation from the fact that $(H,\gamma)\colon\cal{T}\to\cal{T}'$ is a triangle functor.
\end{proof}

\begin{corollary}
\label{c:extri-fun}
Both $F$ and $G$ are extriangulated functors.
\end{corollary}
\begin{proof}
This follows immediately from Lemma~\ref{l:extri-fun} since both $F$ and $G$ are restrictions of triangle functors.
\end{proof}

\begin{lemma}
\label{l:kerF}
The kernel of $FG$ is the ideal of $\lperp(\Sigma\cal{E})$ generated by morphisms from an injective object to a projective object.
\end{lemma}
\begin{proof}
We write $J$ for the relevant ideal, and first show that $J\subseteq\ker{FG}$. By Lemmas~\ref{l:projectives} and \ref{l:injectives}, it is enough to show that $\Hom_{\cal{A}/\cal{E}}(FGQ,FGP)=0$ for any $P\in\cal{A}$ and any $Q$ in the Bongartz co-completion of $\cal{E}$. Applying $\Hom_{\cohom{0}\cal{A}/\cohom{0}\cal{E}}(FG\blank,FGX)$ to the triangle \eqref{eq:bong-co-comp} we get an exact sequence
\[\begin{tikzcd}[column sep=8pt]
\Hom_{\cohom{0}\cal{A}/\cohom{0}\cal{E}}(\Sigma FGX,FGP)\arrow{r}&\Hom_{\cohom{0}\cal{A}/\cohom{0}\cal{E}}(FGE'_X,FGP)\arrow{r}&\Hom_{\cohom{0}\cal{A}/\cohom{0}\cal{E}}(FGE_X,FGP).
\end{tikzcd}\]
The leftmost space is zero because both $FGX$ and $FGP$ are concentrated in degree $0$, while the rightmost space is zero because $E_X\in\cal{E}$ and hence $FGE_X=0$. Thus we have $\Hom_{\cohom{0}\cal{A}/\cohom{0}\cal{E}}(FGE'_X,FGP)=0$ also. Since $Q\in\add(E_X\oplus E'_X)$ for some $X\in\cal{A}$, it follows that $\Hom_{\cohom{0}\cal{A}/\cohom{0}\cal{E}}(FGQ,FGP)=0$, and so $J\subseteq\ker{FG}$.

Conversely, let $f\colon X\to Y$ be a morphism from $\ker{FG}$. We write $X=\cone(X^{-1}\map{x}X^0)$ and $Y=\cone(Y^{-1}\map{y}Y^0)$, with~$X^{-1}, X^0, Y^{-1}, Y^0\in \add\cA$.   Then~$f$ fits into a commutative diagram

\[
\begin{tikzcd}
 X^{-1} \ar{r}{x}\ar{d}{f^{-1}} & X^0 \ar{r}{x'}\ar{d}{f^{0}} & X \ar{r}{x''}\ar{d}{f} & \Sigma X^{-1}\ar{d}{\Sigma f^{-1}} \\
 Y^{-1} \ar{r}{y} & Y^0 \ar{r}{y'} & Y \ar{r}{y''} & \Sigma Y^{-1}
\end{tikzcd}
\]
whose rows are triangles.  By assumption,~$FG(f) = 0$, so there exists~$h\colon FGX^0\to FGY^{-1}$ such that~$FGy\circ h = FGf^0$ and~$h\circ FGx = FGf^{-1}$; in other words, the pair~$(FGf^{-1},FGf^0)$ is null-homotopic, and~$h$ is a null-homotopy.  We lift~$h$ to a map~$\bar h\colon X^0\to Y^{-1}$. Because~$G$ restricts to an equivalence~$\add\cA \to \add\cohom{0}\cA$ and the kernel of~$F$ restricted to~$\add\cohom{0}\cA$ is the ideal of morphisms factoring through~$\add\cohom{0}\cE$, it follows that~$\bar h$ fits into a commutative square
\[
 \begin{tikzcd}
  X^{-1} \ar{r}{x}\ar{d}[swap]{f^{-1}} & X^0 \ar{dl}[swap]{\bar h}\ar{d}{f^{0}} \\
  Y^{-1} \ar{r}{y} & Y^0 
 \end{tikzcd}
\]
in which the two triangles commute modulo a morphism factoring through~$\add\cE$.  This leads to a commutative diagram
\[
\begin{tikzcd}
 X^{-1} \ar{r}{x}\ar{d}[swap]{f^{-1}-\bar h x} & X^0 \ar{r}{x'}\ar{d}[swap]{f^{0} - y \bar h} & X \ar{r}{x''}\ar{d}{f}\ar[dotted]{dl}[swap]{t} & \Sigma X^{-1}\ar{d}{\Sigma (f^{-1}-\bar h x)} \\
 Y^{-1} \ar{r}{y} & Y^0 \ar{r}{y'} & Y \ar{r}{y''} & \Sigma Y^{-1}
\end{tikzcd}
\]
where the two leftmost vertical morphisms factor through~$\add\cE$ and the rightmost one factors through~$\Sigma\cE$.  Since~$X\in \lperp(\Sigma\cE)$, this implies that the rightmost square composes to zero. In particular, there exists~$t\colon X\to Y^0$ such that~$f = y't$.

Now, since~$FG(y't) = FGf = 0$, the morphism~$FGt$ factors through~$FGy$ via a morphism~$u\colon FGX\to FGY^{-1}$.  Since~$F$ and~$G$ are full by Lemmas~\ref{l:G-full-dense} and~\ref{l:F-full-dense}, there exists a lift~$\bar u\colon X\to Y^{\-1}$ of~$u$, with the property that~$t-y\bar u$ is in~$\ker FG$.  Note that~$y'(t-y\bar u) = f$.

By Lemma~\ref{l:0-Aus}, the object $X$ has an injective co-resolution in~$\lperp(\Sigma\cE)$, which corresponds to a triangle~$X\xrightarrow{\ell'} Q_X\to Q'_X\xrightarrow{\Sigma\ell} \Sigma X$ in~$\per\cA$.  Consider the following diagram.
\[
 \begin{tikzcd}
 & \Sigma^{-1}Q'_X \ar{d}{\ell}\ar[dotted]{dl}[swap]{m} & \\
 X^0 \ar{r}{x'} & X\ar{r}{x''}\ar{d}{t-y\bar u} & \Sigma X^{-1} \\
 Y^{-1} \ar{r}{y} & Y^0 \ar{r}{y'} & Y
 \end{tikzcd}
\]
The composition~$x''\ell$ vanishes since $Q'_X\in \cA * \add\cA$ and~$X^{-1} \in \add\cA$, meaning that $\Hom_{\cA}(Q'_X, \Sigma^2 X^{-1})=0$.  Thus~$\ell$ factors through~$x'$ via a morphism~$m\colon\Sigma^{-1}Q'_X\to X^0$.  Note that~$(t-y\bar u)x' \in \ker FG$; since the domain~$X^0$ and codomain~$Y^0$ of this morphism are in~$\add\cA$, this implies that~$(t-y\bar u)x'$ factors through~$\add\cE$.  Since~$Q'_X \in \lperp(\Sigma\cE)$, this implies that~$(t-y\bar u)\ell = (t-y\bar u)x'm = 0$.  Thus~$t-y\bar u$ factors through~$\ell'$ via a morphism~$p\colon Q_X\to Y^0$, so that~$p\ell' = t-y\bar u$.  

Putting all this together, we get that~$f = y'(t-y\bar u) = y'p\ell'$, illustrated thus:
\[
 X\xrightarrow{\ell'} Q_X \xrightarrow{p} Y_0 \xrightarrow{y'} Y.
\]
Since~$Q_X$ is injective and~$Y_0$ is projective in~$\lperp(\Sigma\cE)$, this proves the claim that~$f$ is in the ideal generated by morphisms from an injective object to a projective object.
\end{proof}

This completes the proof of Theorem~\ref{t:rel-2-term-functors}.

\begin{corollary}
\label{c:kerG}
The kernel of $G$ is contained in the ideal of $\lperp(\Sigma\cal{E})$ generated by maps from injective to projective objects.
\end{corollary}

\begin{corollary}
There are equivalences of extriangulated categories
\begin{align*}
&\lperp(\Sigma\cohom{0}\cal{E})/(\cal{I}\to\cal{P})\isoto\harp{\proj\cohom{0}\cal{A}/\cohom{0}\cal{E}},\\
&\lperp(\Sigma\cal{E})/\ker(G)\isoto\lperp(\Sigma\cohom{0}\cal{E}),\\
&\lperp(\Sigma\cal{E})/(\cal{I}\to\cal{P})\isoto\harp{\proj\cohom{0}\cal{A}/\cohom{0}\cal{E}}
\end{align*}
induced from $F$, $G$ and $F G$ respectively, where $(\cal{I}\to\cal{P})$ denotes the ideal generated by maps from injective to projective objects in either $\lperp(\Sigma\cohom{0}\cal{E})$ or $\lperp(\Sigma\cal{E})$ as appropriate.
\end{corollary}
\begin{proof}
Note that the domains carry natural extriangulated structures, making the quotient functors extriangulated, by Theorem~\ref{t:extriangulatedQuotient}; in the second case this makes use of Corollary~\ref{c:kerG}. All three functors are full and essentially surjective by Lemmas~\ref{l:G-full-dense} and \ref{l:F-full-dense}. The second isomorphism then follows immediately, and the other two from Theorem~\ref{t:rel-2-term-functors} using the calculation of $\ker(F)$ and $\ker(F G)$.

Since both $F$ and $G$ are restrictions of triangle functors, they are also canonically extriangulated functors. Since the projection functor from Theorem~\ref{t:extriangulatedQuotient} is an extriangulated functor when combined with the identity natural transformation on extension groups (these being the same in the quotient as in the original category), the functors induced by $F$ and $G$ on the relevant quotient categories are also canonically extriangulated. Since they are additive equivalences, they are also equivalences of extriangulated categories by \cite[Prop.~2.13]{NOS}.
\end{proof}

For convenience, we state dual versions of Theorems~\ref{t:rel-2-term} and \ref{t:rel-2-term-functors}. 
We denote by~$\Transp\colon\per\cA \isoto \per\cA^\op$ the usual contravariant equivalence and call it the transpose.  It satisfies~$\Transp\Sigma=\Sigma^{-1}\Transp$.
Thus we may apply Theorems~\ref{t:rel-2-term} and \ref{t:rel-2-term-functors} to $\cal{E}^\op\subset\cal{A}^\op$ and then take the transpose of the result to obtain statements concerning $\cal{A}$. Since $\Transp(\cal{A}^\op*\Sigma\cal{A}^\op)=\Sigma^{-1}\cal{A}*\cal{A}$, we then apply $\Sigma$ to obtain Theorems~\ref{t:rel-2-term-dual} and \ref{t:rel-2-term-dual-functors} as stated below; this is why the category $\lperp(\Sigma\cal{E})$ from Theorem~\ref{t:rel-2-term} becomes $\cal{E}\rperp$ in the dual.

\begin{theorem}
\label{t:rel-2-term-dual}
Let $\cal{A}$ be a non-positively graded dg category, and let $\cal{E}\subseteq\cal{A}$ be a full subcategory which is contravariantly finite as a subcategory of~$\per\cA$. Write
\[\cal{E}\rperp=\{X\in\cal{A}*\Sigma\cal{A}:\text{$\Hom_{\per\cal{A}}(E,X)=0$ for all $E\in\cal{E}$}\}\]
for the right perpendicular category of $\cal{E}$ in~$\cal{A}*\Sigma\cal{A}$. Then
\begin{enumerate}
\item\label{t:rel-2-term-dual-ext-closed} $\cal{E}\rperp$ is extension-closed in $\per\cal{A}$, hence extriangulated,
\item\label{t:rel-2-term-dual-proj-inj} the injective objects in $\cal{E}\rperp$ are those of $\add\Sigma\cal{A}$, while the projective objects are those of the Bongartz completion of $\Sigma\cal{E}\subseteq\cal{A}*\Sigma\cal{A}$, and
\item\label{t:rel-2-term-dual-0Aus} $\cal{E}\rperp$ is a $0$-Auslander extriangulated category.
\end{enumerate}
\end{theorem}

We recall here that an object $P\in\per{\cal{A}}$ is in the \emph{Bongartz completion} of $\Sigma\cal{E}\subseteq \add\cA\ast\add\Sigma\cA$ if and only if there is a triangle
\begin{equation}
\label{eq:bong-comp}
\begin{tikzcd}
\Sigma E'_X\arrow{r}&\Sigma E_X\arrow{r}{\varphi}&X\arrow{r}&\Sigma X
\end{tikzcd}
\end{equation}
in which $X\in\cal{A}$, the morphism $\varphi$ is a right $\Sigma\cal{E}$-approximation of $X$, and $P\in\add(E_X\oplus E'_X)$.

\begin{theorem}
\label{t:rel-2-term-dual-functors} Keep notation as in Theorem~\ref{t:rel-2-term-dual}. Then in the commutative diagram
\begin{equation}
\begin{tikzcd}[column sep=35pt]
\label{eq:$2$-term-diag-dual}
\per{\cal{A}}\arrow{r}{\blank\ltens_{\cal{A}}\cohom{0}\cal{A}}&\per\cohom{0}\cal{A}\arrow{r}{\blank\ltens_{\cohom{0}\cal{A}}\cohom{0}\cal{A}/\cohom{0}\cal{E}}&\per\cohom{0}\cal{A}/\cohom{0}\cal{E}\\
\cal{E}\rperp\arrow[hookrightarrow]{u}\arrow{r}{G}&(\cohom{0}\cal{E})\rperp\arrow[hookrightarrow]{u}\arrow{r}{F}&\harp{\proj\cohom{0}\cal{A}/\cohom{0}\cal{E}},\arrow[hookrightarrow]{u}
\end{tikzcd}
\end{equation}
the functors $F$ and $G$, given by restricting the derived tensor products, are full and essentially surjective extriangulated functors. Moreover, $\ker{F}$ is the ideal of $(\cohom{0}\cal{E})\rperp$ generated by morphisms from an injective object to a projective object, and $\ker(F G)$ is the analogous ideal of~$\cal{E}\rperp$. 
\end{theorem}

We end this section with a discussion of the seemingly more general case where~$\cA$ is \emph{$1$-rigid} (that is to say,~$\Hom_{\cA}(\cA, \Sigma \cA) = 0$) instead of being non-positively graded.  We show that this situation can be reduced to the previous setting by considering the truncation~$\cA^{\leq 0}$, where~
\[
 (\cA^{\leq 0})^i(\blank,\blank) = \begin{cases} \cA^i(\blank,\blank), & i<0 \\
                                        \mathrm{Z}^0\cA^0(\blank,\blank), & i=0 \\
                                        0, & \text{otherwise}, \end{cases}
\]
and where the differential of~$\cA^{\leq 0}$ is obtained by restriction of that of~$\cA$.  Then the inclusion of~$\cA^{\leq 0}$ into~$\cA$ induces a functor
\[H = - \ltens_{\cA^{\leq 0}} \cA \colon \per\cA^{\leq 0}\to \per\cA.\]

\begin{proposition}
\label{p:1-rigid}
 Let~$\cA$ be a~$1$-rigid dg category, and consider~$\per\cA$ with the relative structure whose bifunctor is~$\bE_{\add\cA}$ \cite[Prop.~3.16]{HerschendLiuNakaoka}. Then
 \begin{enumerate}
  \item $\cA \ast \Sigma \cA$ is closed under extensions in~$\per\cA$ equipped with the above relative structure, and so inherits a structure of an extriangulated category, and
  \item the functor~$H = - \varltens{\cA^{\leq 0}} \cA$ induces an equivalence of extriangulated categories
  \[H\colon \cA^{\leq 0} \ast \Sigma \cA^{\leq 0}\to \cA \ast \Sigma \cA,\]
  where~$\cA \ast \Sigma \cA$ is taken with the above relative structure.
 \end{enumerate}
\end{proposition}
\begin{proof}
That~$\cA \ast \Sigma \cA$ is closed under extensions follows from the fact that~$\add\cA$ is projective for the relative structure.
 
 We prove that~$H\colon\cA^{\leq 0} \ast \Sigma \cA^{\leq 0}\to \cA \ast \Sigma \cA$ is an equivalence of additive categories in several steps, as follows.
 \begin{itemize}
  \item First, observe that $H$ induces an equivalence~$\add\cA^{\leq 0} \to \add\cA$.
  \item Secondly, observe that the essential image of~$H$ restricted to~$\cA^{\leq 0} \ast \Sigma \cA^{\leq 0}$ lies in~$\cA \ast \Sigma \cA$, and~$H\colon \cA^{\leq 0} \ast \Sigma \cA^{\leq 0}\to \cA \ast \Sigma \cA$ is essentially surjective.
  \item If~$X\in \cA^{\leq 0} \ast \Sigma\cA^{\leq 0}$, let~$X^{-1}\to X^0 \to X \to \Sigma X^{-1}$ be a triangle with~$X^{-1},X^0\in \add\cA^{\leq 0}$.  If~$P\in\add\cA^{\leq 0}$, then~$\Hom_{\cA^{\leq 0}}(P,X) \cong \Hom_{\cA}(HP,HX)$ thanks to the diagram
  \[
  \adjustbox{scale=.8,center}{
   \begin{tikzcd}
    \Hom_{\cA^{\leq 0}}(P,X^{-1}) \ar{r}\ar{d} & \Hom_{\cA^{\leq 0}}(P,X^{0}) \ar{r}\ar{d} & \Hom_{\cA^{\leq 0}}(P,X) \ar{r}\ar{d} & \Hom_{\cA^{\leq 0}}(P,\Sigma X^{-1}) \ar{r}\ar{d} & \Hom_{\cA^{\leq 0}}(P,\Sigma X^{0}) \ar{d} \\
    \Hom_{\cA}(HP,HX^{-1}) \ar{r} & \Hom_{\cA}(HP,HX^{0}) \ar{r} & \Hom_{\cA}(HP,HX) \ar{r} & \Hom_{\cA}(HP,\Sigma HX^{-1}) \ar{r} & \Hom_{\cA}(HP,\Sigma HX^{0}) 
   \end{tikzcd}
   }
   \]
  where the two leftmost vertical arrows are isomorphisms by the previous arguments, the two rightmost vertical lines consist of zero spaces since~$\cA$ is~$1$-rigid, and so the middle morphism is an isomorphism by the five lemma.
  \item Similarly, one proves that~$\Hom_{\cA^{\leq 0}}(X,\Sigma P) \cong \Hom_{\cA}(HX, H\Sigma P)$.
  
  \item One proves that~$\Hom_{\cA^{\leq 0}}(X, P) \cong \Hom_{\cA}(HX, H P)$ by using the diagram
  \[
  \adjustbox{scale=.8,center}{
   \begin{tikzcd}
    \Hom_{\cA^{\leq 0}}(\Sigma X^0,P) \ar{r}\ar{d} & \Hom_{\cA^{\leq 0}}(\Sigma X^{-1},P) \ar{r}\ar{d} & \Hom_{\cA^{\leq 0}}(X,P) \ar{r}\ar{d} & \Hom_{\cA^{\leq 0}}(X^0,P) \ar{r}\ar{d} & \Hom_{\cA^{\leq 0}}(X^{-1},P) \ar{d} \\
    \Hom_{\cA}(H\Sigma X^0,HP) \ar{r} & \Hom_{\cA}(H\Sigma X^{-1},HP) \ar{r} & \Hom_{\cA}(HX,HP) \ar{r} & \Hom_{\cA}(HX^0,HP) \ar{r} & \Hom_{\cA}(HX^{-1},HP) 
   \end{tikzcd}   
   }
  \]
  where the two rightmost vertical arrows are isomorphisms by the above arguments, the two leftmost ones are isomorphisms because~$\cohom{-1}\cA^{\leq 0} \cong \cohom{-1}\cA$, and so the middle one is an isomorphism by the five lemma.
  
  \item Similarly, one proves that~$\Hom_{\cA^{\leq 0}}(\Sigma P, X) \cong \Hom_{\cA}(H\Sigma P, H X)$.
  
  \item Finally, if~$X,Y$ are both in~$\cA^{\leq 0} * \Sigma \cA^{\leq 0}$, then the five lemma gives that~$\Hom_{\cA^{\leq 0}}(X, Y) \cong \Hom_{\cA}(HX, HY)$ by using the above reasoning and the diagram
  \[
  \adjustbox{scale=.8,center}{
   \begin{tikzcd}
    \Hom_{\cA^{\leq 0}}(\Sigma X^0,Y) \ar{r}\ar{d} & \Hom_{\cA^{\leq 0}}(\Sigma X^{-1},Y) \ar{r}\ar{d} & \Hom_{\cA^{\leq 0}}(X,Y) \ar{r}\ar{d} & \Hom_{\cA^{\leq 0}}(X^0,Y) \ar{r}\ar{d} & \Hom_{\cA^{\leq 0}}(X^{-1},Y) \ar{d} \\
    \Hom_{\cA}(H\Sigma X^0,HY) \ar{r} & \Hom_{\cA}(H\Sigma X^{-1},HY) \ar{r} & \Hom_{\cA}(HX,HY) \ar{r} & \Hom_{\cA}(HX^0,HY) \ar{r} & \Hom_{\cA}(HX^{-1},HY). 
   \end{tikzcd}   
   }
  \]
  Thus~$H$ is an equivalence of additive categories.
 \end{itemize}
 
 It remains to be shown that~$H$ is an equivalence of extriangulated categories.  That it is an extriangulated functor follows, by Lemma~\ref{l:extri-fun}, from the fact that it is a restriction of a triangulated functor.  Thus we only need to show that it induces an isomorphism of bifunctors~$\Ext_{\cA^{\leq 0}}(\blank,\blank)\cong \bE_{\add \cA}(H\blank,H\blank)$.  Recall that~$\bE_{\add \cA}(U,V)$ is the subset of~$\Hom_{\cA}(U,\Sigma V)$ consisting of those~$f$ such that $fa = 0$ for any~$a\colon A\to U$ with~$A\in \add\cA$.  
 
 Firstly,~$H\colon\Hom_{\cA^{\leq 0}}(X,\Sigma Y) \to \Hom_{\cA}(HX,H\Sigma Y)$ is injective, as can be deduced from the four lemma and the diagram
 \[
  \adjustbox{scale=.9,center}{
   \begin{tikzcd}
    \Hom_{\cA^{\leq 0}}(\Sigma X^0,\Sigma Y) \ar{r}\ar{d} & \Hom_{\cA^{\leq 0}}(\Sigma X^{-1},\Sigma Y) \ar{r}\ar{d} & \Hom_{\cA^{\leq 0}}(X,\Sigma Y) \ar{r}\ar{d} & 0\ar{d}  \\
    \Hom_{\cA}(H\Sigma X^0,H\Sigma Y) \ar{r} & \Hom_{\cA}(H\Sigma X^{-1},H\Sigma Y) \ar{r} & \Hom_{\cA}(HX,H\Sigma Y) \ar{r} & \Hom_{\cA}(HX^0,H\Sigma Y), 
   \end{tikzcd}
   }
 \]
 where the two leftmost vertical arrows are isomorphisms by the above arguments and the rightmost one is obviously injective.

 Secondly, if~$X,Y\in\cA^{\leq 0} \ast \Sigma\cA^{\leq 0}$, then~$\Hom_{\cA^{\leq 0}}(\cA^{\leq 0},\Sigma Y) = 0$ since~$\cA^{\leq 0}$ is non-positively graded.  Thus the image of~$H\colon\Hom_{\cA^{\leq 0}}(X,\Sigma Y) \to \Hom_{\cA}(HX,H\Sigma Y)$ lies in~$\bE_{\add \cA}(HX,HY)$.  
 
 Conversely, let~$u\in\bE_{\add\cA}(HX,H\Sigma Y)$. Then its precomposition with~$HX^0\to HX$ vanishes, since~$HX^0\in \add\cA$.  Thus~$u$ is a composition of two morphisms~$HX\to H\Sigma X^{-1} \to \Sigma Y$, both of which have preimages under~$H$ since this functor induces isomorphisms~$\Hom_{\cA^{\leq_0}}(X,\Sigma X^{-1})\cong \Hom_{\cA}(HX, H\Sigma X^{-1})$ and~$\Hom_{\cA^{\leq_0}}(\Sigma X^{-1}, \Sigma Y) \cong \Hom_{\cA}(\Sigma HX^{-1}, \Sigma HY)$.  Thus~$u$ itself has a preimage under~$H$.  
 
 This shows that~$H\colon\Hom_{\cA^{\leq 0}}(X,\Sigma Y) \to \bE_{\add \cA}(HX,HY)$ is an isomorphism, completing the proof.
\end{proof}

\section{Application to Frobenius exact categories} \label{section:Frobenius}

In this section, we are concerned with the case of an exact Frobenius category~$\cC$ with a cluster-tilting subcategory~$\cT$ (see \cref{def:CT}).  We use the extriangulated substructures~$\cC^{\cT}$ and~$\cC_{\cT}$ from Proposition~\ref{p:HLN-substruct}, in which we will show that $\cT$ becomes the category of injective or projective objects, respectively; moreover, we prove in Theorems~\ref{t:addEquivalenceYoneda} and~\ref{t:addEquivalenceYonedaDual} that sending an object of~$\cC$ to its~$\cT$-(co)-resolution can be done functorially, and that the underlying functor is extriangulated.  This allows us to interpret~$\cC^{\cT}$ and~$\cC_{\cT}$ in the language of Section~\ref{s:rel-2-term}, see Theorem~\ref{t:ClustCat-to-2term-frobenius}.

Our main motivation for this section is the additive categorification of cluster algebras by means of Frobenius categories, as they appear for instance in~\cite{GLS-ParFlag,JKS,FK,KPQ,Pressland22, Pressland23}.

\subsection{Extriangulated substructures and the index}

When $\cal{C}$ is a Frobenius extriangulated category and $\cal{T}\subseteq\cal{C}$ is cluster-tilting, we may characterise conflations in $\cal{C}_{\cal{T}}$ and $\cal{C}^{\cal{T}}$ in terms of the triangulated stable category $\underline{\cal{C}}$. Note that $\cal{T}$ is in particular rigid, i.e.\ $\EE_{\cal{C}}(T,T')=0$ for all $T,T'\in\cal{T}$.

\begin{proposition}
\label{p:T-proj-substruct}
If $\cal{C}$ is Frobenius and $\cal{T}$ is cluster-tilting, then a conflation $X\infl Y\defl Z\dashrightarrow$ 
in $\cal{C}$ is a conflation in $\cal{C}_{\cal{T}}$ if and only if in the induced triangle
\[\begin{tikzcd}
X\arrow{r}&Y\arrow{r}{g}&Z\arrow{r}{\delta}&\Sigma X
\end{tikzcd}\]
in the stable category $\stab{\cal{C}}$, the morphism $\delta$ factors over $\Sigma\cal{T}$. Moreover, $\EE_{\cal{T}}$ is the maximal extriangulated substructure of $\cal{C}$ in which objects of $\cal{T}$ are projective, and simultaneously the maximal extriangulated substructure in which the cosyzygies of objects in~$\cT$ are injective. 
\end{proposition}
\begin{proof}
Let $T\in\cal{T}$ and note that $(\delta_{\sharp})_T$ factors as
\[\Hom_{\cal{C}}(T,Z)\longrightarrow\stabHom_{\cal{C}}(T,Z)\stackrel{\delta_*}{\longrightarrow}\stabHom_{\cal{C}}(T,\Sigma X)=\EE(T,X),\]
where the first map is the canonical surjection. Thus $(\delta_{\sharp})_T=0$ if and only if $\delta_*=0$, if and only if $g_*\colon\stabHom_{\cal{C}}(T,Y)\to\stabHom_{\cal{C}}(T,Z)$ is surjective. This is equivalent to asking that $\delta\in(\Sigma\cal{T})$ by \cite[Thm.~2.3]{KoenigZhu} (applied to $\Sigma\cal{T}\subseteq\stab{\cal{C}}$, which is cluster-tilting since $\cal{T}$ is).

The second statement is proved in greater generality in \cite[Prop.~3.17]{GNP2}, but for convenience we give our own argument here. Let $\delta\in\EE(Z,X)=\stabHom_{\cal{C}}(Z,\Sigma X)$, with associated conflation
\[\begin{tikzcd}
X\arrow[infl]{r}{f}&Y\arrow[defl]{r}{g}&Z\arrow[dashed]{r}{\delta}&.
\end{tikzcd}\]
Assume $\FF$ is an extriangulated substructure of $\cal{C}$ in which objects of $\cal{T}$ are projective, and that $\delta\in\FF(Z,X)$. 
Then any morphism $T\to Z$ with $T\in\cal{T}$ factors over $g$ in $\cal{C}$, hence also in $\stab{\cal{C}}$. Thus $g_*$, defined as above, is surjective.
Since $\cal{T}$ is cluster-tilting, it follows from \cite[Thm.~2.3]{KoenigZhu} again that $\delta\in(\Sigma \cal{T})$ when viewed as a morphism in $\stab{\cal{C}}$, and hence $\delta\in\EE_{\cal{T}}(Z,X)$.

Conversely, assume that $\delta\in(\Sigma \cal{T})$. Since $\cal{T}$ is rigid, it follows that the morphism $g_*$ is surjective for any $T\in\cal{T}$. In other words, given any morphism $h\colon T\to Z$ in $\cal{C}$, there is a projective-injective object $\Pi$ and morphisms $p\colon T\to \Pi$ and $q\colon\Pi\to Z$ such that $h-qp$ factors over $g$.
Since $\Pi$ is in particular projective, $q$ factors over the deflation $g$, which implies that $h$ itself factors over $g$. Thus $T$ is projective in $\cal{C}_{\cal{T}}$.

The arguments concerning injectivity of cosyzygies of objects in $\cal{T}$ are similar, using that $X\in\cal{C}$ is a cosyzygy of an object in $\cal{T}$ if and only if it lies in $\Sigma\cal{T}$ when viewed as an object of $\stab{\cal{C}}$, together with the fact that $\delta\in(\Sigma\cal{T})$ if and only if $\delta\in\EE_{\cal{T}}(Z,X)$. 
\end{proof}

The following statement is dual to Proposition~\ref{p:T-proj-substruct}.

\begin{proposition}
\label{p:T-inj-substruct}
If $\cal{C}$ is Frobenius and $\cal{T}$ is cluster-tilting, then a conflation $X\infl Y\defl Z\dashrightarrow$ 
in $\cal{C}$ is a conflation in $\cal{C}^{\cal{T}}$ if and only if in the induced triangle
\[\begin{tikzcd}
X\arrow{r}{f}&Y\arrow{r}&Z\arrow{r}{\delta}&\Sigma X
\end{tikzcd}\]
in the stable category $\underline{\cal{C}}$, the morphism $\delta$ factors over $\cal{T}$. Moreover, $\EE^{\cal{T}}$ is the maximal extriangulated substructure of $\cal{C}$ in which objects of $\cal{T}$ are injective, and simultaneously the maximal extriangulated substructure in which the syzygies of objects in~$\cT$ are projective. 
\end{proposition}

Because of Proposition~\ref{p:T-proj-substruct}, we call the extriangulated substructure $\EE_{\cal{T}}$ the \emph{$\cal{T}$-projective extriangulated structure} on $\cal{C}$. Dually, we call $\EE^{\cal{T}}$ the \emph{$\cal{T}$-injective extriangulated structure}. Note that the projective-injective objects in $\cal{C}$ remain projective-injective in either of these substructures, since all of them lie in $\cal{T}$, and are simultaneously syzygies and cosyzygies of $0\in\cal{T}$.

\begin{corollary}
\label{c:easy_confl}
Let $X\infl Y\defl Z\dashrightarrow$ be a conflation in $\cal{C}$. If $X\in\cal{T}$, then this is also a conflation in $\cal{C}_{\cal{T}}$, whereas if $Z\in\cal{T}$ then it is also a conflation in $\cal{C}^{\cal{T}}$.
\end{corollary}

For the rest of the subsection, we assume that $\cT$ is essentially small, so that it has a well-defined Grothendieck group $\GrothGp{\cal{T}}$. Since $\cal{T}$ is rigid, there are no non-split extensions between objects in $\cal{T}$, and so $\GrothGp{\cal{T}}$ is just the split Grothendieck group of $\cT$ as an additive category. If $\cT$ is Krull--Schmidt, this is just the free abelian group
generated by the set of isomorphism classes of indecomposable objects in $\cT$.

\begin{definition}
\label{d:ind-coind}
For an object $X\in\cal{C}$, the \emph{index} of $X$ with respect to $\cal{T}$ is the Grothendieck group element $[R_X]-[K_X]\in\GrothGp{\cal{T}}$ for any conflation
\begin{equation}
\label{eq:index-conflation}
\begin{tikzcd}
K_X\arrow[infl]{r}&R_X\arrow[defl]{r}&X\arrow[dashed]{r}&\phantom{0}
\end{tikzcd}
\end{equation}
in $\cal{C}$ such that $R_X,K_X\in\cal{T}$. Dually, the \emph{coindex} of $X$ with respect to $\cal{T}$ is the Grothendieck group element $[L_X]-[C_X]\in\GrothGp{\cal{T}}$ for any conflation
\begin{equation}
\label{eq:coindex-conflation}
\begin{tikzcd}
X\arrow[infl]{r}&L_X\arrow[defl]{r}&C_X\arrow[dashed]{r}&\phantom{0}
\end{tikzcd}
\end{equation}
in $\cal{C}$ such that $L_X,C_X\in\cal{T}$.
\end{definition}

It follows from the definition of a cluster-tilting subcategory that $R_X,K_X\in\cal{T}$ in the conflation \eqref{eq:index-conflation} if and only if the deflation $R_X\defl X$ is a right $\cal{T}$-approximation of $X$. Such a deflation always exists, as in Remark~\ref{r:sff}, since $\cal{C}$ has enough projectives.

This means that there exists a conflation of the form \eqref{eq:index-conflation} for any $X\in\cal{C}$. If
$K_X'\infl R_X'\defl X$ is a second such conflation, then $[R_X]-[K_X]=[R_X']-[K_X']$ by \cite[Lem.~4.36, Rem.~4.37]{PPPP} (see also \cite[Lem.~3.8]{wang2023indices}), and so the index is well-defined. One shows dually that the coindex is also well-defined, with conflations \eqref{eq:coindex-conflation} characterised by the fact that the inflation $X\infl L_X$ is a left $\cal{T}$-approximation of $X$.

From now on, we assume that $(\cal{C},\EE,\s)$ is a weakly idempotent complete Frobenius extriangulated category and that $\cal{T}$ is a cluster-tilting subcategory of $\cal{C}$.
We note that $(\cal{C},\EE,\s)$ satisfies the condition:
\begin{enumerate}[label=(WIC)]
    \item\label{WIC}\cite[Cond.~5.8]{NakaokaPalu} Let $h=gf$ be a morphism in $\cal{C}$. If $h$ is a deflation, then $g$ is a deflation. Dually, if $h$ is an inflation, $f$ is an inflation. 
\end{enumerate}
Indeed, for an arbitrary extriangulated category $(\cal{C},\EE,\s)$, condition \ref{WIC} is equivalent to weak idempotent completeness of $\cal{C}$ by \cite[Prop.~2.7]{Klapproth} (and \cite[Prop.~3.33]{Msapato}, \cite[\S II.1.3]{Tattar-thesis}).

\begin{remark}
When $\cal{C}$ is weakly idempotent complete we do not in fact need to appeal to Remark~\ref{r:sff} for the existence of the conflation \eqref{eq:index-conflation}, since any right $\cal{T}$-approximation is automatically a deflation. Indeed, given a right $\cal{T}$-approximation $r\colon R\to X$, a projective cover of $X$ factors over $r$ since $\cal{T}$ contains all projective objects in $\cal{C}$, and so $r$ is a deflation by \ref{WIC}. Similarly, any left $\cal{T}$-approximation is an inflation. Our argument that the index and coindex are well-defined is, however, valid without the weak idempotent completeness assumption.
\end{remark}

\begin{definition}
Let $\mathbb{A}$ be an abelian group, and let $\varphi\colon\cal{C}\to\mathbb{A}$ be a function. We say that $\varphi$ is \emph{additive} on a conflation $X\infl Y\defl Z\dashrightarrow$ in $\cal{C}$ if $\varphi(X)+\varphi(Z)=\varphi(Y)$.
\end{definition}

\begin{proposition}
\label{p:ind-additive}
The index is additive on all conflations in $\cal{C}_{\cal{T}}$, while the coindex is additive on all conflations in $\cal{C}^{\cal{T}}$.
\end{proposition}
\begin{proof}
First note that \eqref{eq:index-conflation} is a conflation in $\cal{C}_{\cal{T}}$ by Corollary~\ref{c:easy_confl}, since $K_X\in\cal{T}$. Moreover, since $\cal{T}$ is the full subcategory of projective objects in $\cal{C}_{\cal{T}}$, the conflation \eqref{eq:index-conflation} is even a projective resolution of $X$ in this extriangulated category, and so the result follows by (an extriangulated version of) the horseshoe lemma. The proof for $\cal{C}^{\cal{T}}$ is dual.
\end{proof}

For generalised cluster categories in the sense of Amiot \cite{Amiot09}, under the assumption that the exchange matrix associated to $\cal{T}$ has full rank, the conflations in $\cal{C}_{\cal{T}}$ are precisely those on which the index is additive (and similarly for $\cal{C}^{\cal{T}}$ and the coindex), as a consequence of \cite[Prop.~2.2]{Palu08}. However, this is not the case in general, as follows. Let $\cal{C}$ be the cluster category of type $\mathsf{A}_1$, with indecomposable objects $T$ and $\Sigma T$, and Auslander--Reiten triangles 
\[\begin{tikzcd}T\arrow{r}&0\arrow{r}&\Sigma T\arrow{r}&\Sigma T,\end{tikzcd}\qquad
\begin{tikzcd}\Sigma T\arrow{r}&0\arrow{r}&T\arrow{r}&T.\end{tikzcd}\]
Taking $\cal{T}=\add(T)$, the index and coindex are additive on both of these triangles (since $\Sigma T$ has both index and coindex equal to $-[T]$). However, only the first is a conflation in $\cal{C}_{\cal{T}}$, and only the second is a conflation in $\cal{C}^{\cal{T}}$.

We may obtain a similar result in the case that $\cal{C}$ is an exact category; here the assumption that the cluster-tilting subcategory $\cal{T}$ has finite global dimension replaces the assumption on the rank of the exchange matrix, cf.~\cite[Rem.~4.5]{FK}.

\begin{proposition}
\label{p:ind-additive-exact}
Assume that $\cal{C}$ is a Frobenius exact category and that $\cal{T}$ has finite global dimension. Then the conflations in $\cal{C}_{\cal{T}}$ are precisely those from $\cal{C}$ on which the index is additive.
\end{proposition}
\begin{proof}
Applying $\Hom_{\cal{C}}(\cal{T},\blank)$ (that is, the functor $X\mapsto\Hom_{\cal{C}}(\blank,X)|_\cal{T}$) to a conflation \eqref{eq:index-conflation} in $\cal{C}$, which is a short exact sequence since $\cal{C}$ is exact, yields the exact sequence
\[\begin{tikzcd}
0\arrow{r}&\Hom_{\cal{C}}(\cal{T},K_X)\arrow{r}&\Hom_{\cal{C}}(\cal{T},R_X)\arrow{r}&\Hom_{\cal{C}}(\cal{T},X)\arrow{r}&\EE_{\cal{C}}(\cal{T},K_X)=0.
\end{tikzcd}\]
This is a projective resolution of the $\cal{T}$-module $\Hom_{\cal{C}}(\cal{T},X)$ since $R_X,K_X\in\cal{T}$. Thus the Yoneda isomorphism $\GrothGp{\cal{T}}\isoto\GrothGp{\proj\cal{T}}$ takes the index of $X$ to the class of $[\Hom_{\cal{C}}(\cal{T},X)]$.

Now let
\begin{equation}
\label{eq:conf-for-ind-add}
\begin{tikzcd}
X\arrow[infl]{r}&Y\arrow[defl]{r}{g}&Z\arrow[dashed]{r}{\delta}&\phantom{}
\end{tikzcd}
\end{equation}
be a conflation in $\cal{C}$, on which the index is assumed to be additive. Applying $\Hom_{\cal{C}}(\cal{T},\blank)$ again, we get a short exact sequence
\[\begin{tikzcd}
0\arrow{r}&\Hom_{\cal{C}}(\cal{T},X)\arrow{r}&\Hom_{\cal{C}}(\cal{T},Y)\arrow{r}{g_*}&\Hom_{\cal{C}}(\cal{T},Z)\arrow{r}&M\arrow{r}&0
\end{tikzcd}\]
for some $\cal{T}$-module $M$, defined as the cokernel of $g_*$. In $\GrothGp{\proj\cal{T}}$, we have
\[[M]=[\Hom_{\cal{C}}(\cal{T},Z)]-[\Hom_{\cal{C}}(\cal{T},Y)]+[\Hom_{\cal{C}}(\cal{T},X)]=0\]
by the assumption that the index is additive on \eqref{eq:conf-for-ind-add}. But because $\cal{T}$ has finite global dimension, this means that $M=0$, and so $g_*$ is surjective. Thus $\delta\in\EE_{\cal{T}}(Z,X)$ by Proposition~\ref{p:T-proj-substruct}.
\end{proof}
By considering $\cal{C}^\op$, it follows from Proposition~\ref{p:ind-additive-exact} that if $\cal{T}^\op$ has finite global dimension then the conflations in $\cal{C}^{\cal{T}}$ are precisely those from $\cal{C}$ on which the coindex is additive.

\begin{remark}
The assumption of $\cal{T}$ having finite global dimension is rather mild, as can be seen using \cite[Prop.~3.7]{Pressland17}. Indeed, it follows from this result that $\cal{T}$ has finite global dimension whenever
\begin{enumerate}
\item\label{i:GP} $\cal{C}\simeq\GP(B)$ is equivalent to the category of Gorenstein projective modules for an Iwanaga--Gorenstein algebra $B$, 
\item\label{i:add} $\cal{T}=\add(T)$ is the additive closure of an object $T$, and
\item\label{i:Noeth} the algebra $\Lambda=\End_{B}(T)^\op$ is Noetherian.
\end{enumerate}
In fact, using a result of Kalck--Iyama--Wemyss--Yang \cite[Thm.~2.7]{KIWY}, we may even show that in the presence of conditions \ref{i:add} and \ref{i:Noeth}, condition \ref{i:GP} holds if and only if $\cal{C}$ is idempotent complete and $\cal{T}$ (or equivalently $\Lambda$) has finite global dimension. See also \cite[Thm.~3.34]{KW} for a version of this statement without conditions \ref{i:add} and \ref{i:Noeth}. Conditions \ref{i:GP}--\ref{i:Noeth} hold in many cases of interest, for example for the categorifications of partial flag varieties by Geiß--Leclerc--Schröer \cite{GLS-ParFlag}, the Grassmannian cluster categories of Jensen--King--Su \cite{JKS}, and the categorifications of positroid varieties by the fifth author \cite{Pressland22}.
\end{remark}

\subsection{\texorpdfstring{$2$-term complexes for exact categories}{2-term complexes for exact categories}}
\label{ss:exact-2term}

In this subsection, we gather the technical tools on exact categories that will be necessary in \cref{ss:Frobenius} in order to apply \cref{t:rel-2-term} to the case of exact categories with cluster-tilting subcategories.

Let $\cal{C}$ be a weakly idempotent complete exact category with enough projectives and enough injectives.
Let $\cal{P}\subseteq\cal{C}$ be the full subcategory of projective objects and $\cal{I}\subseteq\cal{C}$ the full subcategory of injective objects.
We further assume that every object in $\cal{C}$ has injective dimension at most one (which implies that the projective dimension of every object in $\cal{C}$ is also at most one).

As usual, we write
\[\harp{\proj{\cal{I}}}\subseteq\Kb{\proj{\cal{I}}}\]
for the full subcategory of complexes concentrated in degrees $0$ and $-1$. We will always use the extriangulated structure on $\harp{\proj{\cal{I}}}$ arising from its embedding as an extension-closed subcategory of the triangulated category $\Kb{\proj\cal{I}}$. Consistent with Section~\ref{s:rel-2-term}, we denote by $\yo\colon\cal{I}\to\proj{\cal{I}}$ the covariant Yoneda equivalence $I\mapsto\Hom_{\cal{C}}(\blank,I)|_{\cal{I}}$. This induces a further equivalence $\cal{I}\ast\Sigma\cal{I}\isoto\harp{\proj{\cal{I}}}$, where $\cal{I}\ast\Sigma\cal{I}$ is taken inside the homotopy category $\Kb{\cal{I}}$ of bounded complexes with terms in $\cal{I}$.

Since every object of $\cal{C}$ has finite injective dimension by assumption, every object of the bounded derived category of $\cal{C}$ is quasi-isomorphic to a bounded complex of objects from $\cal{I}$, and so this bounded derived category identifies with $\Kb{\proj\cal{I}}$ under $\yo$.

We may thus abuse notation by identifying $\cal{P}\subseteq\cal{C}$ with its inclusion into this bounded derived category, and thus obtain the orthogonal subcategory $\cal{P}\rperp\subseteq\harp{\proj{\cI}}$. The induced extriangulated structure on $\cal{P}\rperp$, as an extension-closed subcategory of $\harp{\proj\cal{I}}$, is also compatible with the embedding into $\Kb{\proj\cal{I}}$.

\begin{theorem}
\label{thm:exact categories and 2-term}
Let $\cal{C}$ be a weakly idempotent complete exact category with enough projectives~$\cal{P}$ and enough injectives~$\cal{I}$.
Assume that all objects of $\cal{C}$ have injective dimension at most one.
Let $\Psi\colon\cal{C}\to\harp{\proj{\cI}}\subseteq \Kb{\proj\cI}$ be the canonical inclusion of~$\cC$ in its derived category.

Then the extriangulated subcategory $\cal{P}\rperp$ of~$\harp{\proj{\cI}}$ is an exact category, and~$\Psi$ induces an equivalence~$\Psi\colon\cal{C}\to\cal{P}\rperp$ of exact categories.
\end{theorem}

Note that the action of~$\Psi$ on an object~$X$ is given as follows: fix
a conflation $X\overset{i_X}{\infl} I_X^0\overset{p_X}{\defl} I_X^1$, with $I_X^0,I_X^1\in\cal{I}$.  Then
\[\Psi(X)=(\begin{tikzcd}\yo I_X^0\arrow{r}{\yo p_X}&\yo I_X^1\end{tikzcd}).\]

\begin{proof}
The fact that $\Psi$ is a well-defined functor and is fully faithful is fairly well-known (see e.g.\ \cite[Prop.\ in \S A.7]{Positselski}).  That its image lies in~$\cP\rperp$ follows from the fact that $\Hom_{\Kb{\cal{I}}}(\cP,\Psi-) \cong \Ext^1_{\cC}(\cP, -) = 0$.

In order to prove that its essential image coincides with $\cal{P}\rperp$, we observe that $\yo I_0 \xrightarrow{\yo d}\yo I_1$ belongs to $\cal{P}\rperp$ if and only if any morphism in $\cal{C}$ from a projective to $I_1$ factors through $d$.
The latter holds if and only if $d$ is a deflation because $\cal{C}$ has enough projectives and is weakly idempotent complete.
This shows that $\Psi$ is essentially surjective.

Next, we show that $\cal{P}\rperp$ is exact.
Using \cite[Cor.~3.18]{NakaokaPalu}, it is enough to show that its inflations are monomorphisms and its deflations are epimorphisms.
Consider a conflation in $\harp{\proj\cal{I}}$.
It is the image under the Yoneda functor of a componentwise split conflation in $\cal{I}\ast\Sigma\cal{I}$, coming from a commutative diagram
\begin{equation}
\label{eq:sample-confl}
\begin{tikzcd}
I\arrow[infl]{r}\arrow[defl]{d}{d}&E\arrow[defl]{r}{e}\arrow[defl]{d}{d''}&I'\arrow[defl]{d}{d'}\\
J\arrow[infl]{r}&F\arrow[defl]{r}&J'
\end{tikzcd}
\end{equation}
in $\cal{C}$, in which all objects lie in $\cal{I}$ and each row is a (necessarily split) conflation.

Consider a morphism $\alpha$ from some $I_0\stackrel{i}{\defl} I_1$ to $I\defl J$ such that the composition with the inflation to $E\defl F$ is null-homotopic.
Let $I_1\xrightarrow{h} E$ be such a homotopy, and let $r$ be a deflation splitting the inflation $I\infl E$. Because $i$ is an epimorphism, $rh$ gives a null-homotopy of $\alpha$, and we have thus shown that the inflations of $\cal{P}\rperp$ are monomorphic. The argument is illustrated in the following diagram.
\[\begin{tikzcd}
I_0\arrow{r}{\alpha_0}\arrow[defl]{d}{i}&I\arrow[infl]{r}\arrow[defl]{d}{d}&E\arrow[bend right,defl]{l}[swap]{r}\arrow[defl]{d}{d''}\\
I_1\arrow{r}[swap]{\alpha_1}\arrow[bend left=10pt, dotted]{urr}[pos=0.2]{h}&J\arrow[infl]{r}&F
\end{tikzcd}\]

Next, consider a morphism $\beta$ from $I'\defl J'$ to some $J_0\stackrel{j}{\defl} J_1$ with a null-homotopy $F\xrightarrow{k}J_0$ of its composition with the deflation from $E\twoheadrightarrow F$.
We claim that $\beta$ is null-homotopic:
first let $s$ split the deflation $F\defl J'$ and note that the composition $ks$ gives a morphism $J'\to J_0$ such that $jks=\beta_1\colon J'\to J_1$.
However, the composition $ksd'$ might not agree with $\beta_0\colon I'\to J_0$ and so we have to modify $ks$ to obtain the desired homotopy.
Let $X,Y,Z$ be the kernels indicated in the commutative diagram below
\[\begin{tikzcd}
Z\arrow[defl]{r}{p}\arrow[infl]{d}{z}&Y\arrow{r}\arrow[infl]{d}{y}&X\arrow[infl]{d}{x}\\
E\arrow[defl]{r}{e}\arrow[defl]{d}{d''}&I'\arrow{r}{\beta_0}\arrow[defl]{d}{d'}\arrow[dotted]{ur}{b}&J_0\arrow[defl]{d}{j}\\
F\arrow[defl]{r}\arrow[bend left=10pt, dotted]{urr}[pos=0.2]{k}&J'\arrow[bend left]{l}{s}\arrow{r}{\beta_1}&J_1
\end{tikzcd}\]
and note that, by the $3\times 3$ lemma, the induced morphism $Z\xrightarrow{p}Y$ is a deflation.
Because the compositions $jksd'$ and $j\beta_0$ coincide, there is a morphism $I'\xrightarrow{b}X$ such that $\beta_0 = ksd' + xb$.
Moreover, we have $xbyp = \beta_0yp - ksd'yp = \beta_0yp = \beta_0ez = \beta_0kd''z=0$.
Since $x$ is a monomorphism and $p$ is an epimorphism, we have $by=0$, giving us a morphism $J'\xrightarrow{k'}X$ such that $k'd'=b$.
The required null-homotopy is then given by $ks +xk'$, and we conclude that the deflations of $\cal{P}\rperp$ are epimorphic.

To see that $\Psi$ is an exact functor, consider a conflation $Y\infl E\defl X$ in $\cal{C}$ and apply the horseshoe lemma to construct the diagram
\[\begin{tikzcd}[ampersand replacement=\&, row sep=30pt, column sep=30pt]
Y\arrow[infl]{r}{y}\arrow[infl]{d}[swap]{i_Y}\&E\arrow[defl]{r}{x}\arrow[infl]{d}{\left(\begin{smallmatrix}u\\ i_Xx\end{smallmatrix}\right)}\&X\arrow[infl]{d}{i_X}\\
I_Y^0\arrow[infl]{r}{\left(\begin{smallmatrix}1\\0\end{smallmatrix}\right)}\arrow[defl]{d}[swap]{p_Y}\&I_Y^0\oplus I_X^0\arrow[defl]{r}{\left(\begin{smallmatrix}0&1\end{smallmatrix}\right)}\arrow[defl]{d}{\left(\begin{smallmatrix}-p_Y&v\\0&p_X\end{smallmatrix}\right)}\&I_X^0\arrow[defl]{d}{p_X}\\
I_Y^1\arrow[infl]{r}{\left(\begin{smallmatrix}-1\\0\end{smallmatrix}\right)}\&I_Y^1\oplus I_X^1\arrow{r}{\left(\begin{smallmatrix}0&1\end{smallmatrix}\right)}\&I_X^1
\end{tikzcd}\]
whose rows and columns are conflations.
Here the map $E\xrightarrow{u}I_Y^0$ exists because $y$ is an inflation and $I_Y^0$ is injective.
Having chosen $u$, there is a map $v\colon I_X^0\to I_Y^1$ such that
\[\begin{tikzcd}
Y\arrow[infl]{r}{y}\arrow[equal]{d}&E\arrow[defl]{r}{x}\arrow{d}{u}&X\arrow{d}{vi_X}\\
Y\arrow[infl]{r}{i_Y}&I_Y^0\arrow[defl]{r}{p_Y}&I_Y^1
\end{tikzcd}\]
is a morphism of conflations.
To see this, observe that there is a unique map $X\to I_Y^1$ making the diagram commute, which must factor over $i_X$ since $I_Y^1$ is injective and $i_X$ is an inflation.
\end{proof}

\subsection{\texorpdfstring{$2$-term complexes and cluster-tilting for exact categories}{2-term complexes and cluster-tilting for exact categories}} \label{ss:Frobenius}

Let $\cal{C}$ be a weakly idempotent complete Frobenius exact category, let $\cal{P}\subseteq\cal{C}$ be the full subcategory of projective objects, and let $\cal{T}\subseteq\cal{C}$ be cluster-tilting. In particular, this means that $\cal{P}\subseteq\cal{T}$, and so identifying $\cal{P}$ with its image under $\yo\colon\cal{T}\isoto\proj\cal{T}$ we may consider the subcategories $\cal{P}\rperp$ and $\lperp{(\Sigma\cal{P})}$ inside $\harp{\proj{\cal{T}}}$, as in Section~\ref{s:rel-2-term}.

\begin{remark}
In many cases of interest there will be an object $T$ such that $\cal{T}=\add{T}$, and we may set $\Lambda=\End_{\cal{C}}(T)^{\op}$. There is an idempotent element $e\in \Lambda$ given by projection onto a maximal projective summand of $T$, so that $\Lambda/\Lambda e\Lambda=\stabEnd_{\cal{C}}(T)^{\op}$ is the projectively stable endomorphism algebra. Under the natural (Yoneda) identification
\[\harp{\cal{T}}=\harp{\proj \Lambda}:=\{P^\bullet\in\Kb{\proj \Lambda}:\text{$P^i=0$ for $i\neq 0,1$}\}\]
the full subcategory $\cal{P}\rperp$ is identified with
\[(\Lambda e)\rperp=\{P^\bullet\in\harp{\proj \Lambda}:e\cohom{0}(P^\bullet)=0\}.\]
The subcategory $\lperp(\Sigma\cal{P})$ does not have such a straightforward description inside $\Kb{\proj\Lambda}$, but may also be described as the image of
\[(e\Lambda)\rperp=\{P^\bullet\in\harp{\proj\Lambda^\op}:\cohom{0}(P^\bullet)e=0\}\]
under $\Sigma^{-1}\circ\Transp$, where $\Transp=\Hom_{\Lambda^\op}(\blank,\Lambda)\colon\Kb{\proj\Lambda^\op}\isoto\Kb{\proj\Lambda}$ is the contravariant transpose functor. Note here that we use left modules, to aid comparison with examples of Frobenius exact cluster categorifications such as \cite{GLS-ParFlag,JKS,Pressland22,Pressland23}.
\end{remark}

Our aim in this section is to prove a stronger version of Theorem~\ref{thm:exact categories and 2-term} when~$\cC$ is a Frobenius exact category as above, by noting that Theorem~\ref{thm:exact categories and 2-term} applies to the exact category~$\cC^{\cT}$.  We need a technical result first.

\begin{lemma}[{cf.~\cite[Prop.~4(b)]{KelRei}}]
\label{l:KR-lem}
Let $\cal{C}$ be a weakly idempotent complete extriangulated category with enough projectives $\cal{P}$, and let $g\colon T^{1}\to T^0$ be a morphism such that the $\cal{C}$-module $\coker(\Hom_{\cal{C}}(\blank,T^{1})\to\Hom_{\cal{C}}(\blank,T^0))$ vanishes on $\cal{P}$. Then $g$ is a deflation.
\end{lemma}

\begin{proof}
Since $\cal{C}$ has enough projectives, we may take a projective precover $p\colon P\to T^0$ of $T^0$ to obtain the deflation
\[\begin{tikzcd}[ampersand replacement=\&]
T^{1}\oplus P\arrow[defl]{r}{(\begin{smallmatrix}g&p\end{smallmatrix})}\&T^0.
\end{tikzcd}\]
However, since $\coker(\Hom_{\cal{C}}(\blank,T^{1})\to\Hom_{\cal{C}}(\blank,T^0))$ vanishes on $P$, we have $p=gq$ for some $q\colon P\to T^{1}$. It follows that
\[\begin{tikzcd}[ampersand replacement=\&]
T^{1}\oplus P\arrow[defl]{r}{(\begin{smallmatrix}g&p\end{smallmatrix})}\arrow{d}[swap]{(\begin{smallmatrix}1&q\end{smallmatrix})}\&T^0\arrow[equal]{d}\\
T^{1}\arrow{r}{g}\&T^0
\end{tikzcd}\]
is a retraction, and so $g$ is itself a deflation, using that $\cal{C}$ is weakly idempotent complete.
\end{proof}

Dually, if $\cal{C}$ is a weakly idempotent complete extriangulated category with enough injectives $\cal{I}$, then any morphism $f\colon T_1\to T_0$ with the property that $\coker(\Hom_{\cal{C}}(T_0,\blank)\to\Hom_{\cal{C}}(T_1,\blank))$ vanishes on $\cal{I}$ is an inflation.

Let $\widetilde{\cP}$ be the category of projectives in~$\cC^{\cT}$, which in particular contains the category $\cP$ of projectives in $\cal{C}$, and recall from Proposition~\ref{p:T-inj-substruct} that the category of injectives in~$\cC^{\cT}$ is nothing but $\cT$. Using notation as in Section~\ref{ss:exact-2term}, we get a full subcategory~$(\widetilde{\cP})\rperp \subset \harp{\proj\cT}$.
Since~$\cP \subset \widetilde{\cP}$, we see that~$(\widetilde{\cP})\rperp \subset \cP\rperp$ inside~$\harp{\proj\cT}$. Both categories  $(\widetilde{\cP})\rperp$ and $\cP\rperp$ carry extriangulated structures induced by their respective embeddings into~$\harp{\proj\cT}$, and we know from Theorem~\ref{thm:exact categories and 2-term} that this extriangulated structure on $(\widetilde{\cP})\rperp$ is exact.

\begin{corollary} \label{cor:widetilde_p_perp_vs_p_perp}
\begin{enumerate}
\item\label{c:no-tilde} The subcategories $(\widetilde{\cP})\rperp \subset \cP\rperp$ of ~$\harp{\proj\cT}$ coincide.
\item\label{c:exact} The extriangulated structure on $\cP\rperp$ induced by its embedding into~$\harp{\proj\cT}$ is exact.
\end{enumerate}
\end{corollary}

\begin{proof}
Let $g\colon T^{1}\to T^0$ be a morphism in $\cal{T}$ such that~$\yo g\colon \yo T^{1}\to\yo T^0$ belongs to $\cal{P}^\perp$. Then $g$ is a deflation in $\cal{C}$ by Lemma~\ref{l:KR-lem}, and a deflation in $\cC^{\cT}$ by Corollary~\ref{c:easy_confl}. By the argument in the proof of Theorem~\ref{thm:exact categories and 2-term} applied to $\cC^{\cT}$, it follows that $\yo g\colon\yo T^{1}\to\yo T^0$ belongs to $(\widetilde{\cP})\rperp$. In other words, we have $\cal{P}^\perp \subset (\widetilde{\cP})\rperp$. Since we already had the converse inclusion, this implies part \ref{c:no-tilde}. Part \ref{c:exact} follows immediately.
\end{proof}

Recall from Proposition~\ref{p:T-inj-substruct} that~$\cC^{\cT}$ has enough injectives, that its injectives are the objects of~$\cT$, and that all its objects have injective dimension at most one.  Moreover, it is exact, because $\cal{C}$ is exact.  Thus the inclusion of~$\cC^{\cT}$ in its derived category $\Kb{\proj\cal{\cT}}$ induces an inclusion~$\cC^{\cT}\to \harp{\proj\cT}$.

\begin{theorem}
\label{t:addEquivalenceYoneda}
Let $\cal{C}$ be a weakly idempotent complete Frobenius exact category, let $\cal{T}\subset\cal{C}$ be a cluster-tilting subcategory, and let $\cal{P}$ be the category of projective-injective objects in $\cal{C}$. 
Let~$\Psi\colon\cC^{\cT} \to \harp{\proj\cT}$ be induced from the natural inclusion of~$\cC^{\cT}$ in its derived category.
Then~$\Psi$ induces an equivalence of extriangulated categories $\Psi\colon\cal{C}^{\cal{T}}\to\cal{P}\rperp$.
\end{theorem}

\begin{proof}
Applying Theorem~\ref{thm:exact categories and 2-term} to~$\cC^{\cT}$ yields an equivalence~$\Psi\colon\cC^{\cT} \xrightarrow{\sim} (\widetilde{\cP})\rperp\subseteq\harp{\proj\cT}$. The result then follows by Corollary~\ref{cor:widetilde_p_perp_vs_p_perp}.
\end{proof}

\begin{corollary}
\label{c:0-Auslander-substructure}
In the context of Theorem~\ref{t:addEquivalenceYoneda}, the category $\cal{C}^{\cT}$ is~$0$-Auslander.
\end{corollary}
\begin{proof}
 By Theorem~\ref{t:addEquivalenceYoneda}, we have that $\cC^{\cT}$ is equivalent to~$\cP\rperp$ as an extriangulated category.  The result then follows from Theorem~\ref{t:rel-2-term-dual} by taking~$\cA = \cT$ and~$\cE = \cP$.
\end{proof}
Note that Corollary~\ref{c:0-Auslander-substructure} can also be proved directly by using Corollary~\ref{c:easy_confl}, without referring to Theorem~\ref{t:rel-2-term-dual}. 
The dual of Theorem~\ref{t:addEquivalenceYoneda} is the following.

\begin{theorem}
\label{t:addEquivalenceYonedaDual}
With assumptions as in Theorem~\ref{t:addEquivalenceYoneda}, let~$\Psi^{\op}\colon\cC_{\cT}\to \harp{\proj\cT}$ be induced from the inclusion of the exact category~$\cC_{\cT}$ into its derived category.  Then~$\Psi^\op$ induces an equivalence of extriangulated categories~$\cC_{\cT}\to\lperp(\Sigma \cP)$.
\end{theorem}

Note that the functor~$\Psi^\op$ is defined on objects as follows. Fix for each $X\in\cal{C}$ a short exact sequence
\[\begin{tikzcd}
0\arrow{r}&K_X\arrow{r}{k_X}\arrow{r}&R_X\arrow{r}{r_X}&X\arrow{r}&0
\end{tikzcd}\]
with $K_X,R_X\in\cal{T}$. Then
\[\Psi^\op(X)=(\begin{tikzcd}[column sep=35pt]\yo R_X\arrow{r}{\yo k_X}&\yo K_X\end{tikzcd}).\]

\begin{corollary}
 In the context of Theorem~\ref{t:addEquivalenceYonedaDual}, the category~$\lperp(\Sigma \cP)$ is exact.
\end{corollary}

\begin{corollary}
 In the context of Theorem~\ref{t:addEquivalenceYonedaDual}, the category~$\cC_{\cT}$ is~$0$-Auslander.
\end{corollary}

\begin{remark}
Assumptions in all the results in this subsection can be weakened: we do not need $\cC$ to be Frobenius. Indeed, all the proofs apply to the case where $\cC$ is a weakly idempotent complete exact category with enough projectives and enough injectives, and $\cT$ is a cluster-tilting subcategory  of $\cC$. In fact, even this is not completely necessary: to prove Theorem \ref{t:addEquivalenceYoneda}, one does not need $\cC$ to have enough injectives, but only enough $\cC^{\cT}$-injectives, i.e.\ for each object in $\cC$ to admit a conflation \eqref{eq:coindex-conflation}. Dually, to prove Theorem \ref{t:addEquivalenceYonedaDual}, one does not need $\cC$ to have enough projectives, but only enough $\cC_{\cT}$-projectives, i.e.\ for each object in $\cC$ to admit a conflation \eqref{eq:index-conflation}. 
The stronger, but more concise, assumptions we impose are motivated by examples related to the additive categorification of cluster algebras.
\end{remark}

Combining the previous results, we get the following.

\begin{theorem}
\label{t:ClustCat-to-2term-frobenius}
Let $\cal{C}$ be a weakly idempotent complete Frobenius exact category, and let $T\in\cal{C}$ be a cluster-tilting object. Then there are full and essentially surjective extriangulated functors
\[\Phi_T\colon\cal{C}_{T}\to\harp{\proj{\stab{A}}},\qquad\Phi^{T}\colon\cal{C}^T\to\harp{\proj{\stab{A}}},\]
where $\stab{A}=\stabEnd_{\cal{C}}(T)^\op$ is the stable endomorphism algebra of $T$. In each case the kernel is that generated by maps from injective to projective objects (in the relevant extriangulated substructure on $\cal{C}$).
\end{theorem}
\begin{proof}
Combine either Theorem~\ref{t:addEquivalenceYonedaDual} with Theorem~\ref{t:rel-2-term-functors} or Theorem~\ref{t:addEquivalenceYoneda} with Theorem~\ref{t:rel-2-term-dual-functors}. In either case, the relevant functor is induced from the composition of the extriangulated equivalence obtained in Section~\ref{section:Frobenius} with the functor $FG$ from Section~\ref{s:rel-2-term}.
\end{proof}

\section{Application to Higgs categories} \label{section:Higgs}

\subsection{Higgs categories and fundamental domains} \label{ss:Higgs}

Observe that Theorem~\ref{thm:exact categories and 2-term} and Corollary~\ref{cor:widetilde_p_perp_vs_p_perp}, used in
the proof of Theorem~\ref{t:addEquivalenceYoneda}, depended crucially on the fact that we were working in an exact category.
Thus,
Theorem~\ref{t:addEquivalenceYoneda} does not generalise in a straightforward way to the case that $\cal{C}$ is a Frobenius extriangulated category (for example, a triangulated category), even though conflations of the form \eqref{eq:coindex-conflation} may still be constructed in this case.

We do, however, have an analogue of Theorem~\ref{t:addEquivalenceYoneda} in situations in which we have access to certain algebraic enhancements. Precisely, we restrict to certain extriangulated cluster categories arising from ice quivers with potential, called Higgs categories, constructed by Yilin Wu \cite{Wu23}. These include Amiot's generalised cluster categories as special cases. These categories come with a preferred initial cluster-tilting subcategory $\cal{T}$, which has a natural dg enhancement $\proj{\Gamma}$, where $\Gamma$ is the relative Ginzburg dg algebra of the ice quiver with potential \cite{Ginzburg,KellerCYComp,Wu23} (see Definition~\ref{d:relGinz} below) and which is concentrated in non-positive degrees. If $\cal{C}$ is the Higgs category associated to $\Gamma$ in \cite{Wu23}, we will be able to realise the extriangulated categories $\cal{C}_{\cal{T}}$ and $\cal{C}^{\cal{T}}$ as appropriate orthogonal categories in $\Gamma*\Sigma\Gamma\subseteq\per{\Gamma}$, again putting us in a situation to which the results of Section~\ref{s:rel-2-term} may be applied.

Before recalling Wu's construction, we begin with some preliminaries that are not specific to relative Ginzburg algebras. To aid comparison with Section~\ref{section:Frobenius}, we continue to use left modules.

\begin{proposition}
Let $\Lambda$ be a dg algebra concentrated in non-positive degrees and let $e\in \Lambda$ be an idempotent. Then the subcategories $\lperp(\Sigma\Lambda e)$ and $(\Lambda e)\rperp$ of $\Lambda*\Sigma\Lambda$ are $0$-Auslander extriangulated categories.
\end{proposition}
\begin{proof}
This is a direct application of Theorems~\ref{t:rel-2-term} and \ref{t:rel-2-term-dual}.
\end{proof}

When $\Gamma$ is the relative Ginzburg algebra of an ice quiver with potential, the extriangulated category $\Gamma*\Sigma\Gamma$ is denoted by $\FundDom{\Gamma}$ in \cite{Wu23} (or just $\cal{F}$ in \cite{Amiot09}). The next proposition demonstrates that if $e$ is the idempotent determined by the frozen vertices then the subcategory $\lperp(\Sigma\Gamma e)\subseteq\Gamma*\Sigma\Gamma$ is that denoted by $\RelFundDom{\Gamma}$ in \cite{Wu23}, and called the relative fundamental domain.

\begin{proposition}
Let $\Lambda$ be a dg algebra concentrated in non-positive degrees and let $e\in \Lambda$ be an idempotent. Then the subcategory $\lperp(\Sigma\Lambda e)$ is the full subcategory of $\per\Lambda$ consisting of cones of morphisms $f$ in $\add{\Lambda}$ for which $\Hom_{\Lambda}(f,\Lambda e)$ is surjective, while $(\Lambda e)\rperp$ is the full subcategory on cones of those $f$ for which $\Hom_{\Lambda}(\Lambda e,f)$ is surjective.
\end{proposition}
\begin{proof}
Let $f\colon P_1\to P_0$ be a morphism in $\add{\Lambda}$, and consider the triangle
\[\begin{tikzcd}
P_1\arrow{r}{f}&P_0\arrow{r}&X\arrow{r}&\Sigma P_1
\end{tikzcd}\]
in $\per{\Lambda}$. Applying $\Hom_{\Lambda}(\blank,\Lambda e)$ yields the long exact sequence
\[\begin{tikzcd}[column sep=23pt]
\Hom_{\Lambda}(P_0,\Lambda e)\arrow{r}{f^*}&\Hom_{\Lambda}(P_1,\Lambda e)\arrow{r}&\Hom_{\Lambda}(\Sigma^{-1}X,\Lambda e)\arrow{r}&\Hom_{\Lambda}(\Sigma^{-1}P_1,\Lambda e)=0,
\end{tikzcd}\]
where the last term is zero since $\Lambda$ is concentrated in non-positive degrees. Hence we have $\Hom_{\Lambda}(\Sigma^{-1}X,\Lambda e)=\Hom_{\Lambda}(X,\Sigma\Lambda e)=0$ if and only if $f^*=\Hom_{\Lambda}(f,\Lambda e)$ is surjective. The statement concerning $(\Lambda e)\rperp$ is proved similarly.
\end{proof}

\begin{lemma}
\label{l:proj-inj-RelFundDom}
An object $X\in(\Lambda e)\rperp$ is injective if and only if it is isomorphic to $\Sigma P$ for $P\in\add{\Lambda}$, and is projective if and only if it is isomorphic to the cone of a morphism $P_1\longmap{p}P_0$ with $P_0\in\add\Lambda$ and $P_1\in\add\Lambda e$.
\end{lemma}

We note that \cref{l:proj-inj-RelFundDom} sheds new light on \cite[Lem.~2.12]{MarshPalu}.

\begin{proof}
By Theorem~\ref{t:rel-2-term-dual}\ref{t:rel-2-term-dual-proj-inj} the injective objects are exactly those as claimed. Moreover, the projective objects are those in the Bongartz completion of $\Sigma\Lambda e\subseteq\Lambda*\Sigma\Lambda$, all of which have the required form by construction. Thus it only remains to check that every object of this form is projective.

So take $X=(P_1\longmap{p}P_0)\in(\Lambda e)\rperp$ with $P_1\in\add \Lambda e$. Then for $Y=(Q_1\longmap{q}Q_0)\in(\Lambda e)\rperp$, a map $X\to\Sigma Y$ is given by a morphism $\varphi\colon P_1\to Q_0$. Since $P_1\in\add{\Lambda e}$ and $\Hom_\Lambda(\Lambda e,q)$ is surjective, we have
\[\im(\varphi)\subset eQ_0\subset\im(q)\]
and so, since $P_1$ is projective, $\varphi$ factors over $q$. Thus the map $X\to\Sigma Y$ determined by $\varphi$ is null-homotopic, and we conclude that $\Hom_\Lambda(X,\Sigma Y)=0$, and hence that $X$ is projective.
\end{proof}

\begin{remark}
Let $n$ be the number of isomorphism classes of indecomposable projective (or injective, or simple) $\Lambda$-modules. Then it follows directly from Lemma~\ref{l:proj-inj-RelFundDom} that there are $n$ indecomposable injective objects in $(\Lambda e)\rperp$. In fact, it follows from the same lemma that there are also $n$ indecomposable projective objects in $(\Lambda e)\rperp$---these are are the shifts of indecomposable summands of $\Lambda e$ (which are the indecomposable projective-injective objects), together with the minimal projective presentations of the projective $\Lambda/\Lambda e\Lambda$-modules viewed as $\Lambda$-modules.
\end{remark}

Using Proposition~\ref{p:1-rigid}, all the results of the section so far also apply under the weaker assumption that $\Lambda$ is $1$-rigid. We now recall Wu's construction of the Higgs category of an ice quiver with potential, beginning with the definition of the relative Ginzburg algebra, which is in fact concentrated in non-positive degrees.

\begin{definition}
\label{d:relGinz}
Let $(Q,F,W)$ be an ice quiver with potential, meaning that $(Q,W)$ is a quiver with potential and $F$ is a (not necessarily full) subquiver of $Q$. The \emph{relative Ginzburg (dg) algebra} $\Gamma=\Gamma(Q,F,W)$ attached to this data has underlying algebra $\K\widehat{Q}^{\rel}$, where $\widehat{Q}^{\rel}$ is the quiver obtained from $Q$ by adjoining
\begin{itemize}
\item an arrow $\alpha^*\colon j\to i$ for each arrow $\alpha\colon i\to j$ in $Q_1\setminus F_1$, and
\item a loop $t_i\colon i\to i$ for each vertex $i\in Q_0\setminus F_0$.
\end{itemize}
The grading on this algebra is induced by giving paths of $Q$ degree $0$, the arrows $\alpha^*$ degree $-1$, and the loops $t_i$ degree $-2$. Finally, the differential is completely determined by the values
\begin{align*}
d\alpha^*&=\der{\alpha}{W},\\
dt_i&=\sum_{\alpha\in Q_1}e_i[\alpha,\alpha^*]e_i,
\end{align*}
where $\der{\alpha}{W}$ denotes the cyclic derivative of $W$ with respect to $\alpha$. The \emph{frozen Jacobian algebra} $A=J(Q,F,W)$ of the ice quiver with potential is the $0$-th cohomology $\cohom{0}\Gamma$.
We fix the idempotent $e=\sum_{i\in F_0}e_i\in\Gamma$, the sum of the primitive idempotents at the frozen vertices. Being an idempotent, $e$ necessarily has degree $0$, and we abuse notation by denoting its image in $A=\cohom{0}\Gamma$ again by $e$.
\end{definition}

Wu \cite{Wu23} shows how the data of the relative Ginzburg algebra $\Gamma$, and its distinguished idempotent $e$, can be used to construct a Frobenius stably $2$-Calabi--Yau extriangulated category; we briefly summarise this construction. 
Let $\pp{\Gamma}=\Gamma/\Gamma e\Gamma$---equivalently, this is the Ginzburg algebra of the quiver with potential $(\pp{Q},\pp{W})$, where $\pp{Q}$ is the quiver obtained from $Q$ by removing all frozen vertices and their incident arrows, and $\pp{W}$ is obtained from $W$ by removing all cycles through frozen vertices. For any dg algebra $\Lambda$, the perfectly valued derived category is
\[\pvd{\Lambda}=\Bigl\{M\in\dcat{\Lambda}:\sum_{n\in\Z}\dim\cohom{n}M<\infty\Bigr\}.\]
This is sometimes called the bounded derived category of $\Lambda$, but the alternative notation and terminology emphasises that its elements are bounded both horizontally---in the sense that they have cohomology in a bounded set of degrees---and vertically, in the sense that these cohomology groups are individually finite-dimensional. Since $\Gamma$ is homologically smooth \cite{Wu23}, we may treat $\pvd{\pp{\Gamma}}$ as a full subcategory of $\per{\Gamma}$, where it is identified with the thick subcategory generated by simple $\Gamma$-modules $S_i$ for $i\in Q_0\setminus F_0$.

Now assume that $A=\cohom{0}\Gamma$ is a finite-dimensional algebra, in which case we call $(Q,F,W)$ \emph{Jacobi-finite}; we restrict to this case for simplicity, but note that the non-Jacobi-finite case has been addressed recently in \cite{KW}. The following definition is due to Yilin Wu \cite{Wu23}.

\begin{definition}
Let $(Q,F,W)$ be a Jacobi-finite ice quiver with potential, with relative Ginzburg algebra $\Gamma$. Then the \emph{relative cluster category} of $(Q,F,W)$ is
\[\RelClustCat{\Gamma}=\per{\Gamma}/\pvd{\pp{\Gamma}},\]
and the \emph{Higgs category} $\HiggsCat{\Gamma}$ of $(Q,F,W)$ is the image of $\lperp(\Sigma\Gamma e)\subseteq\Gamma*\Sigma\Gamma$ under the projection to $\RelClustCat{\Gamma}$.
\end{definition}

The Higgs category is a Frobenius extriangulated category in the sense of \cite{NakaokaPalu}, whose full subcategory of projective-injective objects is additively generated by the image $P$ of $\Gamma e$. Moreover, $\HiggsCat{\Gamma}$ has almost split sequences \cite{IyamaNakaokaPalu}, and hence an autoequivalence $\tau$ of the stable category $\StabHiggsCat{\Gamma}$. When $F=\varnothing$, so that $\pp{\Gamma}=\Gamma$ and $\lperp(\Sigma\Gamma e)=\Gamma*\Sigma\Gamma$, we have an equality $\HiggsCat{\Gamma}=\RelClustCat{\Gamma}$ by \cite[Prop.~2.9]{Amiot09}, and this category is Amiot's generalised cluster category of $(Q,W)$. Despite the terminology, in the general case it is the Higgs category $\HiggsCat{\Gamma}$ that additively categorifies a cluster algebra in the traditional way, with reachable rigid objects in bijection with cluster monomials, and so on. In particular, the image of $\add{\Gamma}$ in $\HiggsCat{\Gamma}$ is a cluster-tilting subcategory. This subcategory has an additive generator $T$, and there is an isomorphism $\End_{\HiggsCat{\Gamma}}(T)^{\op}\isoto A$. This isomorphism identifies the idempotent $e$ with projection onto the projective generator $P$.
The Higgs category is $2$-Calabi--Yau extriangulated in the sense of \cite{chang2019cluster}, as its stable category is a $2$-Calabi--Yau triangulated category.

Wu has shown that the projection $\lperp(\Sigma\Gamma e)\to\HiggsCat{\Gamma}$ is an additive equivalence. Using this, we will show that the projection of $(\Gamma e)\rperp$ to $\RelClustCat{\Gamma}$ is also an additive equivalence onto its image, which we will denote by $\HiggsCatDual{\Gamma}$.

\begin{proposition}
\label{p:Wu-vs-FGPPP}
The contravariant equivalence $\Sigma\Transp\colon\per{\Gamma^\op}\to\per{\Gamma}$ induces contravariant equivalences $\lperp(\Sigma e\Gamma)\isoto(\Gamma e)\rperp$ and $\HiggsCatWu{\Gamma^\op}\isoto\HiggsCatDual{\Gamma}$.
\end{proposition}
\begin{proof}
The fact that $\Sigma\Transp(\lperp(\Sigma e\Gamma))=(\Gamma e)\rperp$ is a direct calculation, so we need only show that the resulting equivalence descends to the Higgs categories. For this, it is enough to show that $\Sigma\Transp$ descends to a contravariant equivalence $\RelClustCat{\Gamma^\op}\isoto\RelClustCat{\Gamma}$, which we do now.

Recall that the transpose $\Transp=\RHom(\blank,\Gamma)$ is a contravariant equivalence $\per{\Gamma^\op}\isoto\per{\Gamma}$, with the property that $\Transp(\Gamma^{\op})=\Gamma$, $\Transp(e\Gamma)=\Gamma e$, and $\Sigma\Transp=\Transp\Sigma^{-1}$. We show that $\Sigma\Transp$ restricts to a contravariant equivalence $\pvd{\pp{\Gamma}^\op}\isoto\pvd{\pp{\Gamma}}$, and thus induces a contravariant equivalence $\RelClustCat{\Gamma^\op}\isoto\RelClustCat{\Gamma}$ of the quotient categories. Since $\pvd{\pp{\Gamma}}$ is a thick subcategory of $\per{\Gamma}$ by construction, it suffices to check that $\Transp\colon\pvd{\pp{\Gamma}^\op}\isoto\pvd{\pp{\Gamma}}$.

We use the characterisation that $M\in\pvd{\pp{\Gamma}}$ if and only if $\sum_{n\in\Z}\dim\Hom_\Gamma(\Gamma,\Sigma^nM)<\infty$ and $\Hom_\Gamma(\Gamma e,\Sigma^n M)=0$ for all $n\in\Z$. Recall also that the relative $3$-Calabi--Yau property of $\Gamma$ means that
\[\Hom_\Gamma(M,X)=\Hom_\Gamma(X,\Sigma^3M)^*\]
for all $M\in\pvd{\pp{\Gamma}}$ and all $X\in\per{\Gamma}$.

So let $M\in\pvd{\pp{\Gamma}^{\op}}$. Then
\begin{align*}
\sum_{n\in\Z}\dim\Hom_\Gamma(\Gamma,\Sigma^n\Transp M)&
=\sum_{n\in\Z}\dim\Hom_{\Gamma^\op}(\Sigma^{-n}M,\Gamma^\op)\\
&=\sum_{n\in\Z}\dim\Hom_{\Gamma^{\op}}(\Gamma^{\op},\Sigma^{3-n}M)<\infty.
\end{align*}
Similarly,
\[\Hom_\Gamma(\Gamma e,\Sigma^n\Transp M)=\Hom_{\Gamma^\op}(\Sigma^{-n}M,e\Gamma)=\Hom_{\Gamma^{\op}}(e\Gamma,\Sigma^{3-n}M)^*=0\]
for all $n\in\Z$, and so $\Transp M\in\pvd{\pp{\Gamma}}$.
\end{proof}

\begin{corollary}
\label{c:fund-dom}
The projection $(\Gamma e)\rperp\to\HiggsCatDual{\Gamma}$ is an equivalence of additive categories.
\end{corollary}
\begin{proof}
By Proposition~\ref{p:Wu-vs-FGPPP}, there is a commutative diagram
\[\begin{tikzcd}[column sep=30pt]
\lperp(\Sigma e\Gamma)\arrow{r}{\Sigma\Transp}\arrow{d}{\wr}&(\Gamma e)^\perp\arrow{d}\\
\HiggsCat{\Gamma}\arrow{r}{\Sigma\Transp}&\HiggsCatDual{\Gamma}
\end{tikzcd}\]
in which the horizontal arrows are (contravariant) equivalences. Since the left-hand vertical arrow is an equivalence of (additive) categories by \cite[Prop.~5.20]{Wu23} (applied to $\Gamma^\op$, which is the relative Ginzburg algebra of the opposite quiver with potential), so is the right-hand vertical arrow.
\end{proof}

\begin{remark}
\label{r:Wu-comparison}
Since $\HiggsCat{\Gamma^\op}=\HiggsCat{\Gamma}^\op$, it follows from Proposition~\ref{p:Wu-vs-FGPPP} that $\HiggsCatDual{\Gamma}\simeq\HiggsCat{\Gamma}$, although they need not coincide as subcategories of $\RelClustCat{\Gamma}$. Since this equivalence is even restricted from an automorphism of $\RelClustCat{\Gamma}\simeq(\RelClustCat{\Gamma^\op})^\op$, it is an equivalence of extriangulated categories. Thus $\HiggsCatDual{\Gamma}$ is equivalent to $\HiggsCat{\Gamma}$ from the point of view of cluster categorification.

Wu's fundamental domain $\lperp(\Sigma\Gamma e)$ has the advantage that it contains $\Gamma$ itself, the image of which defines the initial cluster-tilting object in $\HiggsCat{\Gamma}$ \cite[Prop.~5.49]{Wu23}.
On the other hand, the fundamental domain $(\Gamma e)\rperp$, from which $\Sigma\Gamma$ projects to an initial cluster tilting object in $\HiggsCatDual{\Gamma}$, is more compatible with the cluster character of \cite{FuPengZhang}, which takes shifts of projectives to initial cluster variables, and with the association of initial cluster variables to negative simple roots in \cite[Thm.~1.9]{FZ-CA2}.
\end{remark}

We conclude this subsection with analogues of Theorem~\ref{t:addEquivalenceYoneda} for Higgs categories.

\begin{theorem}
\label{t:HiggsCatEquiv}
Let $\Gamma$ be the relative Ginzburg algebra of an ice quiver with potential, and let $\cal{T}\subseteq\HiggsCat{\Gamma}$ be the cluster-tilting subcategory given by the image of $\Gamma$ under the defining projection $\lperp(\Sigma\Gamma e)\to\HiggsCat{\Gamma}$. Then this projection becomes an equivalence of extriangulated categories when $\HiggsCat{\Gamma}$ is equipped with the $\cal{T}$-projective extriangulated structure.
\end{theorem}

\begin{proof}
The projection $\pi\colon\lperp{(\Sigma\Gamma e)}\to\HiggsCat{\Gamma}$ is restricted from a triangle functor, and so by Lemma~\ref{l:extri-fun} it is an extriangulated functor for the usual extriangulated structure on $\HiggsCat{\Gamma}$.
In particular, for any $X,Y\in\lperp(\Sigma\Gamma e)$, we have a map $\Hom_{\Gamma}(X,\Sigma Y)\to\EE_{\HiggsCat{\Gamma}}(\pi X,\pi Y)=\Hom_{\RelClustCat{\Gamma}}(\pi X,\Sigma\pi Y)$, induced from the additive functor $\pi$ and the natural transformation $\pi\Sigma\to\Sigma\pi$ making it a triangle functor.
It suffices to check that the image of this map coincides with the subgroup $\EE_{\cal{T}}(\pi X,\pi Y)\leq\EE_{\HiggsCat{\Gamma}}(\pi X,\pi Y)$ defining the $\cal{T}$-projective extriangulated structure on the Higgs category.

By Proposition~\ref{p:T-proj-substruct}, this subgroup is the set of maps $\pi X\to\Sigma\pi Y$ in $\RelClustCat{\Gamma}$ factoring over an object from $\Sigma\cal{T}$.
Thus our result will follow from showing that $f\colon\pi X\to\pi\Sigma Y$ factors over $\Sigma\cal{T}$ if and only if $f=\pi g$ for some $g\colon X\to\Sigma Y$ in $\per{\Gamma}$.

Let $g\colon X\to\Sigma Y$.
Since $X,Y\in\Gamma*\Sigma\Gamma$, we may form a map of triangles
\begin{equation}
\label{eq:SigmaGamma}
\begin{tikzcd}
\Gamma_1^X\arrow{r}\arrow{d}{0}&
\Gamma_0^X\arrow{r}\arrow{d}{0}&
X\arrow{d}{g}\arrow{r}&
\Sigma\Gamma_1^X\arrow{d}{0}\arrow[dotted]{dl}
\\
\Sigma\Gamma_1^Y\arrow{r}&
\Sigma\Gamma_0^Y\arrow{r}&
\Sigma Y\arrow{r}&
\Sigma^2\Gamma_1^Y,
\end{tikzcd}
\end{equation}
noting that $\Hom(\Gamma,\Sigma\Gamma)=0$ because $\Gamma$ is concentrated in non-positive degrees (so in particular $1$-rigid). This leads to the existence of the dotted arrow, showing that $g$ factors over $\add\Sigma\Gamma$, and hence $\pi g$ factors over $\Sigma\cal{T}$.

Conversely, assume $f\colon\pi X\to\pi\Sigma Y$ factors as the composition of $f_1\colon \pi X\to\Sigma T$ and $f_2\colon\Sigma T\to\pi\Sigma Y$ for some $T=\pi\Gamma'\in\cal{T}$.
Since $\Sigma T,\Sigma Y\in\Sigma\HiggsCat{\Gamma}$, there is a morphism $g_2\colon\Sigma\Gamma'\to\Sigma Y$ with $\pi g_2=f_2$, by \cite[Prop.~5.20]{Wu23}.
Applying $\Hom_\Gamma(\blank,\Sigma\Gamma')$ to the top row of \eqref{eq:SigmaGamma} yields the diagram
\[\begin{tikzcd}
\Hom_\Gamma(\Sigma\Gamma_0^X,\Sigma\Gamma')\arrow{r}\arrow{d}&
\Hom_\Gamma(\Sigma\Gamma_1^X,\Sigma\Gamma')\arrow{r}\arrow{d}&
\Hom_\Gamma(X,\Sigma\Gamma')\arrow{r}\arrow{d}&
0\\
\Hom_{\RelClustCat{\Gamma}}(\Sigma T_0^X,\Sigma T)\arrow{r}&
\Hom_{\RelClustCat{\Gamma}}(\Sigma T_1^X,\Sigma T)\arrow{r}&
\Hom_{\RelClustCat{\Gamma}}(\pi X,\Sigma T)\arrow{r}&
0,
\end{tikzcd}\]
in which $T_i^X=\pi\Gamma_i^X$. By \cite[Prop.~5.2]{Wu23} again, the first two vertical arrows are isomorphisms, and hence so is the third. In particular, there exists $g_1\colon X\to\Sigma\Gamma'$ such that $\pi g_1=f_1$. We therefore conclude that $f=f_2f_1=\pi(g_2g_1)$, as required.
\end{proof}

For convenience, we also state the dual.

\begin{theorem}
\label{t:HiggsCatDualEquiv}
Let $\Gamma$ be the relative Ginzburg algebra of an ice quiver with potential, and let $\cal{T}\subseteq\HiggsCatDual{\Gamma}$ be the cluster-tilting subcategory given by the image of $\Sigma\Gamma$ under the defining projection $(\Gamma e)\rperp\to\HiggsCatDual{\Gamma}$. Then this projection becomes an equivalence of extriangulated categories when $\HiggsCatDual{\Gamma}$ is equipped with the $\cal{T}$-injective extriangulated structure.
\end{theorem}

The following corollary is immediate.

\begin{corollary} \label{cor:silting_in_Higgs}
The defining projection $\pi$ in Theorem~\ref{t:HiggsCatEquiv} induces a bijection between 
 silting objects in $\lperp(\Sigma\Gamma e)\subseteq\Gamma*\Sigma\Gamma$ and
silting objects in $\HiggsCat{\Gamma}$ considered with the $\cT$-projective structure.
It commutes with mutations and induces a bijection between $\Gamma$-reachable silting objects and $\cT$-reachable silting objects. 
The dual statement applies to the projection in Theorem~\ref{t:HiggsCatDualEquiv}.
\end{corollary}

Let us note that when we consider $\HiggsCat{\Gamma}$ with its Frobenius extriangulated structure, both defining projections send silting objects to cluster-tilting objects. We already mentioned earlier that the image of $\add \Gamma$ is a cluster-tilting subcategory $\cT = \add T$, where $T$ is the image of $\Gamma$. We sketch the proof of the general statement similar to the one presented in \cite{BrustleYang}.

\begin{proposition}
\label{prop:cluster-tilting_in_Higgs}
The defining projection $\pi$ in Theorem~\ref{t:HiggsCatEquiv} induces a bijection between 
 silting objects in $\lperp(\Sigma\Gamma e)\subseteq\Gamma*\Sigma\Gamma$ and
cluster-tilting objects in $\HiggsCat{\Gamma}$ considered with its Frobenius exact structure defined in \cite{Wu23}.
It commutes with mutations and induces the bijection between $\Gamma$-reachable silting objects and $\cT$-reachable cluster-tilting objects. The dual statement applies to the projection in Theorem~\ref{t:HiggsCatDualEquiv}.
\end{proposition}

\begin{proof}
The proof of \cite[Prop.~4.7]{BrustleYang} applies with minor modifications. By \cite[Cor.~5.38]{Wu23} for $n = 2$, an object $X$ in $\lperp(\Sigma\Gamma e)\subseteq\Gamma\ast\Sigma\Gamma$ is rigid if and only if its image $\pi(X)$ is rigid in $\HiggsCat{\Gamma}$. It is silting if and only if $X$ has precisely $|\Gamma|$ direct summands; see \cite[Prop.~3.14]{BrustleYang} and the references therein. At the same time, its image $\pi(X)$ is cluster-tilting if and only if it has the same number of direct summands by \cite[Thm.~3.13, Rem.~3.14]{wang2023indices}. Thus, the projection $\pi$ induces the desired bijection. The image of a mutation conflation in $\lperp(\Sigma\Gamma e)$ is a mutation conflation in $\HiggsCat{\Gamma}$ since $\pi$ sends conflations to conflations and, being an additive equivalence, sends minimal approximations to minimal approximations. 
\end{proof}

By comparing Corollary~\ref{cor:silting_in_Higgs} with Proposition~\ref{prop:cluster-tilting_in_Higgs}, we immediately get the following corollary.

\begin{corollary}
An object is ($\cT$-reachable) cluster-tilting in $\HiggsCat{\Gamma}$ considered with its Frobenius extriangulated structure if and only if it is ($\cT$-reachable) silting in $\HiggsCat{\Gamma}$ considered with the $\cT$-projective structure.
\end{corollary}

See also \cite[\S3.3.11]{GNP2} for a discussion on the relationship between rigid objects in extriangulated categories and those in the same additive categories considered with suitable $0$-Auslander relative structures.

In fact, the observation that cluster-tilting objects in a $\Hom$-finite $2$-Calabi--Yau triangulated category $\cC$ with a fixed cluster-tilting object $T$ coincide with silting objects in $\cC$ considered with the $\add T$-projective relative structure was one of the main original motivations to introduce and study $0$-Auslander categories, see \cite[\S3.3.6]{GNP2}, \cite[\S4.7.1]{PPPP}, \cite[\S5]{Palu2023}.

\subsection{\texorpdfstring{Relationship to $2$-term complexes}{Relationship to 2-term complexes}} \label{ss:cl_vs_$2$-term}

We begin by stating the main theorem, which is now a direct application of earlier results.

\begin{theorem}
\label{t:ClustCat-to-2term}
Let~$\cC=\HiggsCat{\Gamma}$ be the Higgs category of an ice quiver with potential and let~$T$ be the image of~$\Gamma$ in~$\cC$.  Then there is a full and essentially surjective extriangulated functor
\[
 \Phi_T\colon\cC_T\to\harp{\proj{\stab{A}}},
\]
where $\stab{A}=\stabEnd_{\cal{C}}(T)^\op$ is the stable endomorphism algebra of $T$.

Dually, if~$\cC = \HiggsCatDual{\Gamma}$ and~$T$ is the image of~$\Sigma\Gamma$ in~$\cC$, then there is a full and essentially surjective extriangulated functor
\[
 \Phi^{T}\colon\cC^T\to\harp{\proj{\stab{A}}}.
\]
In both cases, the kernel is the ideal generated by morphisms from injective to projective objects (in the relevant extriangulated substructure on~$\cC$).
\end{theorem}
\begin{proof}
Combine Theorem~\ref{t:HiggsCatEquiv} with Theorem~\ref{t:rel-2-term-functors} or, for the dual result, Theorem~\ref{t:HiggsCatDualEquiv} with Theorem~\ref{t:rel-2-term-dual-functors}. 
\end{proof}

\begin{corollary} 
\label{c:self-injective}
Let $(Q,W)$ be a Jacobi-finite quiver with potential whose Jacobian algebra $A$ is self-injective. Then there is an additive equivalence
\[\cal{C}_{Q,W}\isoto \harp{\proj{A}},\]
where $\cal{C}_{Q,W}$ is the Amiot cluster category of $(Q,W)$.
\end{corollary}

\begin{proof}
We have $\cal{C}_{Q,W}=\cal{C}_{\Gamma}=\HiggsCat{\Gamma}$ for $\Gamma$ the Ginzburg algebra of $(Q,W)$. Let $T$ denote the image of $\Gamma$ in $\cal{C}_{Q,W}$.
Since $A$ is self-injective, by \cite[Prop.~4.6]{KoenigZhu} or \cite[Prop.~3.6]{iyama2013stable} we have $\Sigma^2T = T$. Thus
\[\cal{C}(\Sigma T, T) = \cal{C}(\Sigma T, \Sigma^2T) = 0,\]
by the rigidity of $T$. Combining Theorem~\ref{t:HiggsCatEquiv} with Theorem~\ref{t:rel-2-term}\ref{t:rel-2-term-proj-inj}, we see that $T$ is a projective generator, and $\Sigma T$ an injective cogenerator, of $\cal{C}_{Q,W}$ equipped with the $T$-projective extriangulated structure, and so there are no non-zero maps from injective to projective objects in this extriangulated category. Thus the full and essentially surjective functor $\Phi_T$ from Theorem~\ref{t:ClustCat-to-2term} is also faithful, and hence an additive equivalence. \end{proof}

We may also obtain a slightly more refined theorem, involving the full endomorphism algebra $A=\End_{\cal{C}}(T)^\op$.

\begin{theorem}
\label{t:ClustCat-to-rel2term}
Let $\cal{C}=\HiggsCat{\Gamma}$ be the Higgs category of an ice quiver with potential, and let $T$ be the initial cluster-tilting object given by the projection of $\Gamma$. Then there is a full and essentially surjective extriangulated functor
\[\Psi_T\colon\cal{C}_{T}\to\lperp(\Sigma Ae)\subseteq\harp{\proj{A}},\]
for $A=\End_{\cal{C}}(T)=\cohom{0}\Gamma$, whose kernel is contained in that generated by maps from injective to projective objects (in the $T$-projective extriangulated substructure on $\cal{C}$).

Dually, if $\cal{C}=\HiggsCatDual{\Gamma}$ and $T$ is the projection of $\Sigma\Gamma$ then there is a full and essentially surjective extriangulated functor
\[\Psi^{T}\colon\cal{C}^T\to(Ae)\rperp\subseteq\harp{\proj{A}}\]
with kernel contained in that generated by maps from injective to projective objects (now in the $T$-injective extriangulated structure on $\cal{C}$).
\end{theorem}

\begin{proof}
This is proved by exactly the same strategy as Theorem~\ref{t:ClustCat-to-2term}, but using only the functor $G$ from Theorem~\ref{t:rel-2-term-functors} or Theorem~\ref{t:rel-2-term-dual-functors} instead of the composition $FG$; since $\ker(FG)$ coincides exactly with the ideal generated by maps from injective to projective objects, $\ker(G)$ is at least contained in this ideal.
\end{proof}

\begin{remark}
Each of the four full and essentially surjective functors appearing in Theorems~\ref{t:ClustCat-to-2term} and \ref{t:ClustCat-to-rel2term} induces an extriangulated equivalence on the quotient of its domain by its cokernel, this quotient being naturally extriangulated by Theorem~\ref{t:extriangulatedQuotient}. The codomain of $\Psi^T$ from Theorem~\ref{t:ClustCat-to-rel2term} has a straightforward description as the full subcategory of $2$-term complexes from $\harp{\proj{A}}$ with the property that their $0$-th cohomology is a module for the stable endomorphism algebra $\stab{A}$.

The conclusions of Theorem~\ref{t:ClustCat-to-rel2term} are also valid when $\cal{C}$ is an exact category as in Theorem~\ref{t:ClustCat-to-2term-frobenius}, and $T\in\cal{C}$ is any cluster-tilting object, but since in this case the functor $G$ from Section~\ref{s:rel-2-term} is itself an equivalence (the relevant additive dg category being concentrated in degree $0$), it does not provide any information beyond that in Theorems~\ref{t:addEquivalenceYoneda} and \ref{t:addEquivalenceYonedaDual}. Indeed, in this case the functors $\Psi_T$ and $\Psi^T$ are equivalences of exact categories, obtained by composing the equivalences from the two aforementioned theorems with the Yoneda equivalence $\yo\colon\cal{T}*\Sigma{\cal{T}}\isoto\harp{\proj{A}}$, where $\cal{T}=\add{T}$.
\end{remark}

\begin{corollary}
Each of the four full and essentially surjective functors appearing in Theorems~\ref{t:ClustCat-to-2term} and \ref{t:ClustCat-to-rel2term} sends rigid objects to rigid objects and induces isomorphisms of posets of silting objects (in the $0$-Auslander extriangulated domains and codomains of these functors).  
\end{corollary}

\begin{proof}
Combine Corollary~\ref{cor:$2$-term_quotient_rigid} with Corollary~\ref{cor:silting_in_Higgs}.
\end{proof}

\begin{remark}
If we consider the domains of the  functors in Theorems~\ref{t:ClustCat-to-2term} and \ref{t:ClustCat-to-rel2term}  with their Frobenius 2-Calabi--Yau structures, then thanks to Proposition~\ref{prop:cluster-tilting_in_Higgs} we see that these functors induce bijections from cluster-tilting objects in $\HiggsCat{\Gamma}$ and $\HiggsCatDual{\Gamma}$ to silting objects in respective codomains. While this makes intuitive sense, the functors are not extriangulated for these structures; they are extriangulated only for $T$-projective, resp.\ $T$-injective, structures on the domains.
\end{remark}

\section{Examples} \label{section:Examples}
\begin{example}
Let $\cal{C}$ be the (triangulated) cluster category of type $\sf{A}_2$, with cluster-tilting object $T=T_1\oplus T_2$, as shown in Figure~\ref{fig:cc-a2}. Then $A=\End_{\cal{C}}(T)^{\op}$ is isomorphic to the path algebra of the $\sf{A}_2$ quiver $1\to 2$, so the category $\harp{\proj A}$ is as shown in Figure~\ref{fig:harp-a2}.
\begin{figure}[b]
\begin{tikzpicture}[scale=1.3]
\node at (-1,1) (T1b) {$T_1$};
\node at (0,0) (X) {$X$};
\node at (1,1) (Y) {$Y$};
\node at (2,0) (Z) {$Z$};
\node at (3,1) (T2) {$T_2$};
\node at (4,0) (T1) {$T_1$};
\node at (5,1) (Xb) {$X$};
\draw[->] (T1b) -- (X);
\draw[->] (X) -- (Y);
\draw[->] (Y) -- (Z);
\draw[->] (Z) -- (T2);
\draw[->] (T2) -- (T1);
\draw[->] (T1) -- (Xb);
\draw[dashed] (-1,0) -- (X) -- (Z) -- (T1) -- (5,0);
\draw[dashed] (T1b) -- (Y) -- (T2) -- (Xb);
\draw[dotted] (-0.5,1.3) -- (-0.5,-0.3);
\draw[dotted] (4.5,1.3) -- (4.5,-0.3);
\end{tikzpicture}
\caption{The Auslander--Reiten quiver of the cluster category $\cal{C}$ of type $\sf{A}_2$.}
\label{fig:cc-a2}\bigskip
\begin{tikzpicture}[scale=1.5]
\node at (0,0) (P2) {$|(0\to P_2)$};
\node at (1,1) (P1) {$|(0\to P_1)$};
\node at (2,0) (S1) {$(P_2\to P_1)$};
\node at (3,1) (P2[1]) {$(P_2\to 0)|$};
\node at (4,0) (P1[1]) {$(P_1\to 0)|$};
\draw[->] (P2) -- (P1);
\draw[->] (P1) -- (S1);
\draw[->] (S1) -- (P2[1]);
\draw[->] (P2[1]) -- (P1[1]);
\draw[dashed] (P1) -- (P2[1]);
\draw[dashed] (P2) -- (S1) -- (P1[1]);
\end{tikzpicture}
\captionof{figure}{The Auslander--Reiten quiver of $\harp{\proj A}$ for $A$ the path algebra of the quiver $1\to 2$.}
\label{fig:harp-a2}
\end{figure}

The reader may observe that one obtains a category equivalent to $\harp{\proj A}$ from $\cal{C}$ by factoring out the ideal generated by the irreducible morphism $T_1\to \Sigma T_1=X$ (which generates the whole ideal of morphisms $T\to \Sigma T$). The equivalence takes $\add{T}$ to the injective objects in $\harp{\proj A}$, i.e.\ the shifts of projective $A$-modules, and takes $\add{\Sigma^{-1}T}$ to the projective objects.

One can further check that the three Auslander--Reiten triangles from $\cal{C}$ that are sent to conflations in $\harp{\proj A}$ are all additive for the coindex with respect to $T$, whereas the two which are not (because some of their morphisms are in the kernel), are not additive for this function. Since the exchange matrix of the $\mathsf{A}_2$ quiver has full rank, this additivity property characterises the conflations in the $\add(T)$-injective extriangulated structure on $\cal{C}$ by \cite[Prop.~2.2]{Palu08}.
\end{example}

\begin{example}
\label{eg:a2_preproj}
Let $\cal{E}$ be the module category of the preprojective algebra of type $\sf{A}_2$, which is a Frobenius exact category, as shown in Figure~\ref{fig:a2-pproj}. Let $T=P_1\oplus P_2\oplus S_1\in\cal{E}$, which is a cluster tilting object, and let $A=\End_{\cal{E}}(T)^\op$, which is isomorphic to the path algebra of the quiver
\begin{equation}
\label{eq:EndT-eg}
\begin{tikzcd}1\arrow{rr}{q}&&2\arrow{dl}{r}\\&a\arrow{ul}{p}\end{tikzcd}
\end{equation}
modulo relations $qp$ and $rq$. Under the isomorphism, vertex $a$ corresponds to the summand $S_1$ of $T$, whereas $1$ corresponds to $P_1$ and $2$ to $P_2$. For $e=e_1+e_2$ (corresponding to the projective-injective summands of $T$), the category $(Ae)\rperp$ is shown in Figure~\ref{fig:pproj-a2-relharp}.
\begin{figure}[b]
\begin{tikzpicture}[scale=1.3]
\node at (-1,1) (P1b) {$|\Pi_1|$};
\node at (0,0) (S1) {$S_1$};
\node at (1,1) (P2) {$|\Pi_2|$};
\node at (2,0) (S2) {$S_2$};
\node at (3,1) (P1) {$|\Pi_1|$};
\node at (4,0) (S1b) {$S_1$};
\draw[->] (P1b) -- (S1);
\draw[->] (S1) -- (P2);
\draw[->] (P2) -- (S2);
\draw[->] (S2) -- (P1);
\draw[->] (P1) -- (S1b);
\draw[dashed] (-1,0) -- (S1) -- (S2) -- (S1b);
\draw[dotted] (-0.5,1.3) -- (-0.5,-0.3);
\draw[dotted] (3.5,1.3) -- (3.5,-0.3);
\end{tikzpicture}
\caption{The Auslander--Reiten quiver of the module category of the preprojective algebra of type $\sf{A}_2$. The indecomposable projective modules are denoted by $\Pi_i$ to avoid confusion with the projective modules of the algebra $A=\End_{\cal{E}}(T)^\op$.}
\label{fig:a2-pproj}\bigskip
\begin{tikzpicture}[yscale=1.5,xscale=1.8]
\node at (-1,1) (P1[1]b) {$|(P_1\to0)|$};
\node at (0,0) (Pa[1]) {$(P_a\to0)|$};
\node at (1,1) (P2[1]) {$|(P_2\to 0)|$};
\node at (2,0) (Sa) {$|(P_1\to P_a)$};
\node at (3,1) (P1[1]) {$|(P_1\to 0)|$};
\node at (4,0) (Pa[1]b) {$(P_a\to0)|$};
\draw[->] (P1[1]b) -- (Pa[1]);
\draw[->] (Pa[1]) -- (P2[1]);
\draw[->] (P2[1]) -- (Sa);
\draw[->] (Sa) -- (P1[1]);
\draw[->] (P1[1]) -- (Pa[1]b);
\draw[dashed] (Sa) -- (Pa[1]b);
\draw[dotted] (-0.5,1.3) -- (-0.5,-0.3);
\draw[dotted] (3.5,1.3) -- (3.5,-0.3);
\draw (Pa[1]) edge[dotted, bend left=40] (Sa);
\end{tikzpicture}
\captionof{figure}{The Auslander--Reiten quiver of $(Ae)\rperp$, for $A$ as shown in \eqref{eq:EndT-eg} and $e=e_1+e_2$. While the sequence $(P_a\to0)\to(P_2\to 0)\to(P_1\to P_a)$ is not an almost split conflation (indeed, it starts in an injective and ends in a projective), the composition of the two morphisms is still zero.}
\label{fig:pproj-a2-relharp}
\end{figure}

The reader may observe that $\cal{E}$ and $(Ae)\rperp$ are equivalent as additive categories, as predicted by Theorem~\ref{t:addEquivalenceYoneda}, and become equivalent as exact categories if we equip $\cal{E}$ with the $\add(T)$-injective substructure, in which the sequence $S_1\to\Pi_2\to S_2$ is no longer exact. Indeed, since $\gldim{A}=3$ we may use Proposition~\ref{p:ind-additive-exact} to see that the exact sequences in this substructure are exactly those for which the coindex with respect to $T$ is additive. We may calculate that the coindex of $S_2$ is $[\Pi_1]-[S_1]$ (the coindices of the other indecomposables in $\cal{E}$ being their own classes, since all of them lie in $\add{T}$). Thus the coindex is not additive on the almost split sequence $S_1\to\Pi_2\to S_2$, which we exclude, whereas it is additive on the almost split sequence $S_2\to P_1\to S_1$, which we retain.
\end{example}

\begin{example}
Let $\cal{C}$ be the cluster category of type $\mathsf{A}_3$.
Let $T = T_1 \oplus T_2 \oplus T_3$ be a cluster-tilting object whose endomorphism algebra is the path algebra of a linearly oriented $\mathsf{A}_3$ quiver. Let $T' = T_1 \oplus T_2' \oplus T_3$ be its mutation at $T_2$. Then $\End_{\cal{C}}(T')^\op$ is the path algebra of the cyclic quiver \eqref{eq:EndT-eg} modulo relations $qp$, $pr$ and $rq$. We write $\cT = \add T$ and $\cT' = \add T'$. The Auslander--Reiten quiver of the triangulated category $\cal{C}$ is drawn in Figure \ref{fig:cluster_A_3}. 

\begin{figure}[b]
 \[\begin{tikzpicture}[scale=0.9]
 \draw (-2,1) node (t) {$\cal{C}=\cal{C}(\mathsf{A}_3)$:};
 \foreach \n/\x/\y/\o in {1/0/0/T_1, 2/2/0/\Sigma T'_2, 3/4/0/T'_2, 4/6/0/\Sigma T_3, 5/1/1/T_2, 6/3/1/, 7/5/1/\Sigma T_2, 8/2/2/T_3, 9/4/2/\Sigma T_1, 4b/0/2/\Sigma T_3, 1b/6/2/T_1, 5b/7/1/T_2}
 {\draw (\x,\y) node (\n) {${\o}$};}
 \draw (-0.3,1) node (l) {};
 \draw (7.2,2) node (ru) {};
 \draw (7.2,0) node (rd) {};
 \foreach \n in {6}
 {\filldraw (\n) circle (0.05);}
 \foreach \n/\i in {1/1, 5/2, 8/3}
 \foreach \s/\t in {1/5, 5/8, 5/2, 8/6, 2/6, 6/9, 6/3, 9/7, 3/7, 7/4, 4b/5, 7/1b, 4/5b, 1b/5b}
 {\path[->] (\s) edge (\t);}
 \foreach \l/\r in {1/2, 2/3, 3/4, 5/6, 6/7, 8/9, 4b/8, 9/1b, l/5, 7/5b, 1b/ru, 4/rd}
 {\path[dashed] (\r) edge (\l);}
 \draw[dotted] (0.5,-0.3) -- (0.5,2.3);
 \draw[dotted] (6.5,-0.3) -- (6.5,2.3);
 \end{tikzpicture}\]
 \caption{The Auslander--Reiten quiver of the cluster category of type $\mathsf{A}_3$ considered as a triangulated category.}
 \label{fig:cluster_A_3}\bigskip
  \[\hspace{0.8em}\begin{tikzpicture}[scale=0.9]
  \draw (-2,1) node (t) {$\cal{C}_{\cal{T}}$:};
  \foreach \n/\x/\y/\o in {1/0/0/|T_1, 2/2/0/\Sigma T'_2, 3/4/0/T'_2, 4/6/0/\Sigma T_3|, 5/1/1/|T_2, 6/3/1/, 7/5/1/\Sigma T_2|, 8/2/2/|T_3, 9/4/2/\Sigma T_1|, 4b/0/2/\Sigma T_3|, 1b/6/2/|T_1, 5b/7/1/|T_2}
  {\draw (\x,\y) node (\n) {${\o}$};}
  \draw (-0.3,1) node (l) {};
  \draw (7.2,2) node (ru) {};
  \draw (7.2,0) node (rd) {};
  \foreach \n in {6}
  {\filldraw (\n) circle (0.05);}
  \foreach \n/\i in {1/1, 5/2, 8/3}
  \foreach \s/\t in {1/5, 5/8, 5/2, 8/6, 2/6, 6/9, 6/3, 9/7, 3/7, 7/4, 1b/5b}
  {\path[->] (\s) edge (\t);}
  \foreach \l/\r in {1/2, 2/3, 3/4, 5/6, 6/7, 8/9, 1b/ru}
  {\path[dashed] (\r) edge (\l);}
  \draw[dotted] (0.5,-0.3) -- (0.5,2.3);
  \draw[dotted] (6.5,-0.3) -- (6.5,2.3);
\foreach \s/\t in {4b/5, 7/1b, 4/5b}
  {\path[->] (\s) edge (\t);}
  \end{tikzpicture}\]
  
  \[\hspace{-3.3em}\begin{tikzpicture}[scale=0.9]
  \draw (-2.5,1) node (t) {$\cal{C}_{\cal{T}}/( \Sigma \cT \to \cT)$:};
  \foreach \n/\x/\y/\o in {1/0/0/|T_1, 2/2/0/\Sigma T'_2, 3/4/0/T'_2, 4/6/0/\Sigma T_3|, 5/1/1/|T_2, 6/3/1/, 7/5/1/\Sigma T_2|, 8/2/2/|T_3, 9/4/2/\Sigma T_1|}
  {\draw (\x,\y) node (\n) {${\o}$};}
  \draw (-0.3,1) node (l) {};
  \draw (7.2,2) node (ru) {};
  \draw (7.2,0) node (rd) {};
  \foreach \n in {6}
  {\filldraw (\n) circle (0.05);}
  \foreach \n/\i in {1/1, 5/2, 8/3}
  \foreach \s/\t in {1/5, 5/8, 5/2, 8/6, 2/6, 6/9, 6/3, 9/7, 3/7, 7/4}
  {\path[->] (\s) edge (\t);}
  \foreach \l/\r in {1/2, 2/3, 3/4, 5/6, 6/7, 8/9}
  {\path[dashed] (\r) edge (\l);}
  \end{tikzpicture}\]
 \captionof{figure}{On the top, the Auslander--Reiten quiver of the cluster category of type $\mathsf{A}_3$ considered as an extriangulated category with the $\cT$-projective structure. As in Example~\ref{e:3-cycleQuotient}, there are some zero relations which are not consequences of Auslander--Reiten conflations (and which are not shown). On the bottom, the quotient of this category by the ideal $( \Sigma \cT \to \cT)$.}
 \label{fig:cluster_A_3_linear}
\end{figure}

We consider the category $\cal{C}$ with the $\cT$-projective extriangulated structure. As an example of Theorem \ref{t:ClustCat-to-2term}, we see that the ideal quotient $\cal{C}_{\cal{T}}/(\Sigma\cal{T}\to\cal{T})$ is equivalent to the category $\harp{\End_{\cC}(T)^{\op}}$. See Figure \ref{fig:cluster_A_3_linear} and compare with the right hand side of Figure \ref{fig::Higgs}.

Consider now $\cal{C}$ with the $\cT'$-projective extriangulated structure. Applying Theorem \ref{t:ClustCat-to-2term} again, we see that the ideal quotient $\cal{C}_{\cal{T}'}/(\Sigma\cal{T}'\to\cal{T}')$ is equivalent to the category $\harp{\End_{\cC}(T')^{\op}}$. Moreover, since $\End_{\cC}(T')^{\op}$ is a self-injective algebra, this is in fact an example of \cref{c:self-injective}: we have 
\[
\cal{C}_{\cal{T}'} = \cal{C}_{\cal{T}'}/(\Sigma\cal{T}'\to\cal{T}') \overset\sim\to \harp{\End_{\cC}(T')^{\op}},
\]
as there are no non-zero maps from injectives to projectives in the extriangulated category $\cal{C}_{\cal{T}'}$. This is depicted in Figure \ref{fig:cluster_A_3_cyclic}.
Compare with the Auslander--Reiten quiver of the category $\harp{{\End (T')}^{\op}}$ in Figure \ref{fig:AR_for_harp_cyclic_A3} from Example \ref{e:3-cycleQuotient}.

\begin{figure}[t]
 \[\hspace{-3.3em}\begin{tikzpicture}[scale=0.9]
 \draw (-3.5,1) node (t) {$\cal{C}_{\cal{T'}} = \cal{C}_{\cal{T'}}/( \Sigma \cT' \to \cT')$:};
 \foreach \n/\x/\y/\o in {1/0/0/|T_1, 2/2/0/\Sigma T'_2|, 3/4/0/|T'_2, 4/6/0/\Sigma T_3|, 5/1/1/T_2, 6/3/1/, 7/5/1/\Sigma T_2, 8/2/2/|T_3, 9/4/2/\Sigma T_1|, 4b/0/2/\Sigma T_3|, 1b/6/2/|T_1, 5b/7/1/T_2}
 {\draw (\x,\y) node (\n) {${\o}$};}
 \draw (-0.3,1) node (l) {};
 \draw (7.2,2) node (ru) {};
 \draw (7.2,0) node (rd) {};
 \foreach \n in {6}
 {\filldraw (\n) circle (0.05);}
 \foreach \n/\i in {1/1, 5/2, 8/3}
 \foreach \s/\t in {1/5, 5/8, 5/2, 8/6, 2/6, 6/9, 6/3, 9/7, 3/7, 7/4, 4b/5, 7/1b, 4/5b, 1b/5b}
 {\path[->] (\s) edge (\t);}
 \foreach \l/\r in {1/2, 3/4, 5/6, 6/7, 8/9, l/5, 7/5b, 1b/ru}
 {\path[dashed] (\r) edge (\l);}
 \draw[dotted] (0.5,-0.3) -- (0.5,2.3);
 \draw[dotted] (6.5,-0.3) -- (6.5,2.3);
 \end{tikzpicture}\]

 \caption{The Auslander--Reiten quiver of the cluster category of type $\mathsf{A}_3$ considered as an extriangulated category with the $\cT'$-projective structure. The ideal $( \Sigma \cT' \to \cT')$ is trivial, hence the category is equal to its quotient by this ideal.}
 \label{fig:cluster_A_3_cyclic}
\end{figure}
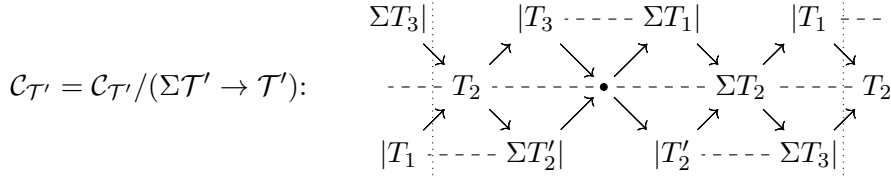

We see that extriangulated categories with different underlying additive categories---in this case $\harp{\End_{\cC}(T)^{\op}}$ and $\harp{\End_{\cC}(T')^{\op}}$---can appear as extriangulated ideal quotients, as in Theorem \ref{t:ClustCat-to-2term}, of the same additive category $\cal{C}$ considered with two different extriangulated structures.
\end{example}

\section{Application to gentle algebras} \label{section:gentle}

In this section, all modules will be right modules, for compatibility with \cite{PaluPilaudPlamondon-nonkissing}. A \emph{gentle quiver with relations} is a pair~$\bar Q = (Q,R)$, where~$Q$ is a quiver (always assumed to be finite in this paper) and~$R$ is a set of paths of length~$2$ in~$Q$, such that
\begin{itemize}
\item any vertex of~$Q$ is the source of at most~$2$ arrows, and the target of at most~$2$ arrows,
\item for any arrow~$\beta$ of~$Q$, there is at most one arrow~$\alpha$ such that~$t(\alpha)=s(\beta)$ and~$\beta\alpha \in R$, and at most one arrow~$\gamma$ such that~$s(\gamma)=t(\beta)$ and~$\gamma\beta \in R$, and
\item for any arrow~$\beta$ of~$Q$, there is at most one arrow~$\alpha'$ such that~$t(\alpha')=s(\beta)$ and~$\beta\alpha' \notin R$, and at most one arrow~$\gamma'$ such that~~$s(\gamma')=t(\beta)$ and~$\gamma'\beta \notin R$.
\end{itemize}
In this case, the algebra~$A = kQ/(R)$ is a \emph{gentle algebra}.  Note that we do not require $A$ to be finite-dimensional at this stage.

The following definition first appeared in \cite{Asashiba12} (see \cite[\S7]{Asashiba16} for the English translation).  It was rediscovered in \cite{BrustleDouvilleMousavandThomasYildirim} and \cite{PaluPilaudPlamondon-nonkissing}; we will use the terminology of the latter.

\begin{definition} \label{def: blossom}
Let~$\bar Q$ be a gentle quiver with relations.  Its \emph{blossoming quiver with relations} is the quiver with relations~$\bar Q\blossom = (Q\blossom, R\blossom)$ obtained from~$\bar Q$ defined as follows:
\begin{enumerate}
\item the quiver~$Q\blossom$ is obtained by adding vertices and arrows to~$Q$ so that
\begin{itemize}
	\item each original vertex of~$Q$ is the source of two arrows and the target of two arrows in~$Q\blossom$, and
	\item each vertex of~$Q\blossom$ which is not a vertex of~$Q$ has valency precisely one; such vertices are called \emph{blossoming vertices};
\end{itemize}
\item the set of relations~$R\blossom$ contains~$R$, and is chosen so that~$(Q\blossom, R\blossom)$ is a gentle quiver with relations.
\end{enumerate}
We denote by~$A\blossom$ the gentle algebra~$kQ\blossom/=(R\blossom)$. Note that~$\bar Q\blossom$ is only unique up to (non-unique) isomorphism of quiver with relations.
\end{definition}

We now fix a gentle quiver with relations~$\bar Q$, with corresponding gentle algebra~$A$, and let~$e\blossom \in A\blossom$ be the sum of the idempotents corresponding to the blossoming vertices of~$Q\blossom$.  The canonical isomorphism~$A \cong A\blossom/( e\blossom)$ induces an extriangulated functor
\[\begin{tikzcd}[column sep=70pt]
\har{\proj A\blossom} \arrow{r}{F = \blank \otimes_{A\blossom} A\blossom/( e\blossom)}& \har{\proj A}
\end{tikzcd}
\]
The first main result of this section is the following.

\begin{proposition}\label{prop::gentle-har}
The functor~$F = \blank \otimes_{A\blossom} A\blossom/(e\blossom)$ induces an equivalence of~$k$-linear categories
\[
{}^\perp (\Sigma e\blossom A\blossom)/(e\blossom A\blossom) \xrightarrow{\sim} \har{\proj A},
\]
where~${}^\perp (\Sigma e\blossom A\blossom)$ is the full subcategory of~$\har{\proj A\blossom}$ whose objects are those~$X$ such that there are no non-zero morphisms from~$X$ to~$\Sigma e\blossom A\blossom$, and~$(e\blossom A\blossom)$ is the ideal of morphisms factoring through an object in~$\add e\blossom A\blossom$.
\end{proposition}
\begin{proof}
Most of the statement is a direct application of Theorem~\ref{t:rel-2-term}, taking~$\cal{A}=A$ viewed as a dg category concentrated in degree~$0$ and~$\cal{E}=\add e\blossom A\blossom$. The only additional item to prove is the more precise description of the kernel of $F$.

We first claim that the kernel of the functor~$F$ is the ideal of morphisms of~$\har{\proj A\blossom}$ homotopic to those of the form~$f=(f^{-1},f^0)$, with~$f^{-1}$ and~$f^0$ factoring through~$\add (e\blossom A\blossom)$ in~$\mod A\blossom$.  

\begin{itemize}
\item[\emph{Step 1.}] It is clear that~$e\blossom A\blossom$ and~$\Sigma e\blossom A\blossom$ are in the kernel of~$F$.  Moreover, by restricting~$F$ to~$\proj A\blossom$, we see that a morphism between two objects of~$\har{\proj A\blossom}$ concentrated in degree~$0$ is in the kernel of~$F$ is and only if it factors through~$\add( e\blossom A\blossom )$.  

\smallskip

\item[\emph{Step 2.}] Assume that a morphism~$f\colon P\to Q$ is in the kernel of~$F$.  Let us depict this morphism by
\[
\begin{tikzcd}
P^{-1} \ar{r}{p}\ar{d}{f^{-1}} & P^0\ar{d}{f^{0}} \\
Q^{-1} \ar{r}{q} & Q^0.
\end{tikzcd}
\]
Since~$f$ is in the kernel of~$F$, the morphism~$Ff$ is null-homotopic in~$\har{\proj A}$.  By lifting this homotopy (which is possible since~$F$ is a full functor when restricted to~$\proj A\blossom$), we get a morphism~$s\colon P^0\to Q^{-1}$ such that~$(f^{-1} -sp)$ and~$(f^0 - qs)$ both vanish when applying~$F$.  By Step~1, these morphisms both factor through~$\add(e\blossom A\blossom)$.  Since~$f=(f^{-1},f^0)$ is homotopic to~$(f^{-1} -sp, f^0 - qs)$, the claim is proved.
\end{itemize}

Let~$P,Q\in {}^\perp(\Sigma e\blossom A\blossom)$, and let~$f\colon P\to Q$ be a morphism.  By the above argument, up to replacing $f$ by a homotopic map we may assume that both~$f^{-1}$ and~$f^0$ factor through~$\add(e\blossom A\blossom)$.  

Since~$P,Q\in {}^\perp(\Sigma e\blossom A\blossom)$, we have that~$P^{-1}$ and~$Q^{-1}$ have no direct factor in~$\add(e\blossom A\blossom)$. This implies that $P^{-1}$ is left Hom-orthogonal to the simple projectives (i.e.\ projectives at sinks) while $Q^{-1}$ is right Hom-orthogonal to projectives at sources. Coupled with the fact that~$f^{-1}$ factors through~$\add(e\blossom A\blossom)$, this implies that~$f^{-1} =0$, since the blossoming vertices are either sinks or sources.

Therefore,~$f$ is concentrated in degree~$0$; moreover, since~$f^0$ factors through~$\add(e\blossom A\blossom)$ in~$\proj A\blossom$, this implies that~$f$ factors through~$\add(e\blossom A\blossom)$ in~$\har{\proj A\blossom}$.
\end{proof}

From now on, assume that the gentle algebra~$A$ is finite-dimensional. This implies that the algebra~$A\blossom$ is also finite-dimensional (since there are no cycles in $Q\blossom$ which were not present in $Q$ already; see Definition \ref{def: blossom}). 

The second main result of this Section allows us to reconstruct~$\har{\proj A}$ from the module category of~$A\blossom$. Consider the full subcategory of~$\mod A\blossom$ defined by
\[
\widetilde{\cW}_A = \{ M\in\mod A\blossom : \text{$\Hom_{A\blossom}(Q,\tau M) = 0$ for all $Q\in\add(e\blossom A\blossom)$}\}.
\]

\begin{proposition}\label{prop::mod-blossom-yields-har-A}
\begin{enumerate} 
\item\label{gentle1} The cohomology functor~$\cohom{0}\colon\har{\proj A\blossom} \to \mod A\blossom$ induces an equivalence 
\(
{}^\perp (\Sigma e\blossom A\blossom) \xrightarrow{\sim} \widetilde{\cW}_A.
\)

\item\label{gentle2} In particular, we have an equivalence~$\har{\proj A} \xrightarrow{\sim} \widetilde{\cW}_A /(e\blossom A\blossom)$.
\end{enumerate}
\end{proposition}
\begin{proof}
Point \ref{gentle2} follows from~\ref{gentle1} and \cref{prop::gentle-har}.  To prove \ref{gentle1}, recall (for instance from \cite[Lem.~2.6]{Plamondon13}) that if~$M$ and~$N$ are~$A\blossom$-modules and if~$P_M$ and~$P_N$ are their minimal projective presentations viewed as objects of~$\har{\proj A\blossom}$, then
\[
\Hom_{A\blossom}(M,\tau N) \cong D\Hom_{\har{\proj A\blossom}}(P_N, \Sigma P_M).
\]
Thus an~$A\blossom$-module~$M$ is in~$\widetilde{\cW}_A$ if and only if for each~$Q\in\add(e\blossom A\blossom)$, the space of morphisms~$\Hom_{\har{\proj A\blossom}}(P_M,\Sigma Q)$ vanishes.  This implies that the restriction of~$\cohom{0}$ to~${}^\perp (\Sigma e\blossom A\blossom)$ has image exactly~$\widetilde{\cW}_A$.

Since~$\cohom{0}$ is full, all that remains is to prove that the restriction of~$\cohom{0}$ to~${}^\perp (\Sigma e\blossom A\blossom)$ is faithful.  We know that the kernel of~$\cohom{0}$ is~$(\Sigma A\blossom)$.  Let us show that the intersection of this ideal with~${}^\perp (\Sigma e\blossom A\blossom)$ is zero.

Let~$f\colon\bar P\to \bar Q$ be a morphism in~$(\Sigma A\blossom)$, with~$\bar P$ and~$\bar Q$ in~${}^\perp (\Sigma e\blossom A\blossom)$.  Since~$f\in(\Sigma A\blossom)$, it is homotopic to
\[
\begin{tikzcd}
\bar P^{-1} \ar{r}{p}\ar{d}{f^{-1}} & \bar P^0 \ar{d}{0} \\
\bar Q^{-1} \ar{r}{q} & \bar Q^0,
\end{tikzcd}
\]
where we note that the rightmost morphism vanishes.  Let~$0\to \bar Q^{-1}\xrightarrow{\iota} R$ be the injective envelope of~$Q^{-1}$.  Since~$\bar Q^{-1}$ is a projective~$A\blossom$-module, its socle is supported on the blossoming vertices of~$Q\blossom$.  Thus~$R$ is a projective-injective~$A\blossom$-module, and is therefore in~$\add(e\blossom A\blossom)$.  Thus we have a commutative diagram
\[
\begin{tikzcd}
\bar P^{-1} \ar{r}{p}\ar{d}{f^{-1}} & \bar P^0 \ar{d}{0} \\
\bar Q^{-1} \ar{r}{q} \ar{d}{\iota} & \bar Q^0 \ar{d} \\
R \ar{r} & 0
\end{tikzcd}
\]
Since~$\bar Q \in {}^\perp (\Sigma e\blossom A\blossom)$, there exists a morphism~$\bar \iota\colon\bar Q^0 \to R$ such that~$\bar \iota q = \iota$.  Hence
\[
\iota f^{-1} = \bar \iota q f^{-1} = \bar \iota 0 = 0,
\]
and since~$\iota$ is injective, this implies that~$f^{-1}=0$.  Thus~$f=0$, and the proposition is proved.
\end{proof}

Our main motivation for~\cref{prop::mod-blossom-yields-har-A} is \cref{corollary: category of walks} below.
We first recall some definition and results from~\cite[\S6]{GNP2}.

Define~$\cal{W}_A$ to be the full additive subcategory of~$\mod A\blossom$ whose indecomposable objects are the band modules and the non-simple string modules starting and ending at a blossoming vertex.
The category~$\cal{W}_A$ is extension-closed in~$\mod A\blossom$ \cite[Prop.~6.21]{GNP2} and, when endowed with the inherited exact structure, it is a $0$-Auslander exact category \cite[Cor.~6.23]{GNP2}.
Moreover,~$\cal{W}_A$ categorifies the combinatorics of the non-kissing complex~\cite{McConville,BrustleDouvilleMousavandThomasYildirim,PaluPilaudPlamondon-nonkissing}.
We note that the category $\widetilde{\cW}_A$ is $0$-Auslander by Proposition~\ref{prop::mod-blossom-yields-har-A} and Theorem~\ref{t:rel-2-term}.

\begin{corollary}
\label{corollary: category of walks}
There is an equivalence of extriangulated categories
\[
 \mathcal{W}_A / J \simeq \har{\proj A},
\]
where~$J$ is the ideal generated by the morphisms with injective domain and projective codomain, and coincides with the ideal of morphisms factoring through a projective-injective object.
\end{corollary}

\begin{proof}
First note that a string $A\blossom$-module~$M$ belongs to~$\widetilde{\cW}_A$ if and only if~$\tau M$ is not supported at blossoming vertices, if and only if both endpoints of the string of~$M$ are blossoming vertices.
It follows that~$\mathcal{W}_A=\widetilde{\cW}_A/(e_\text{sink}\blossom A\blossom)$, where~$e_\text{sink}\blossom$ is the sum of the idempotents corresponding to the sinks of~$Q\blossom$. 
The equivalence is then a restatement of \cref{prop::gentle-har}.
The statement on the ideal~$J$ immediately follows from string combinatorics noting that the string of
 \begin{itemize}
  \item a non-projective injective indecomposable starts and ends at sources,
  \item a projective-injective indecomposable starts at a source and ends at a sink,
  \item a non-injective projective indecomposable starts and ends at sinks.\qedhere
 \end{itemize}
\end{proof}

\begin{remark}
 The equality~$\mathcal{W}_A=\widetilde{\cW}_A/(e_\text{sink}\blossom A\blossom)$ provides a new proof that~$\cW_A$ is~$0$-Auslander, as it expresses it as a quotient of another~$0$-Auslander category by some projective-injective objects. 
\end{remark}

\bibliographystyle{plain}

\end{document}